\documentclass{amsart}

\usepackage[utf8]{inputenc}
\usepackage[T1]{fontenc}
\usepackage[a4paper,margin=2.5cm]{geometry}
\usepackage{amsmath,amssymb,amsfonts,amsthm,mathtools,mathrsfs}
\usepackage{xcolor}
\usepackage{hyperref}
\usepackage[nameinlink,capitalize,noabbrev]{cleveref}

\hypersetup{
  colorlinks=true,
  linkcolor=red!70!black,
  citecolor=green!45!black,
  urlcolor=red!70!black,
  linktoc=page
}
\numberwithin{equation}{section}

\newcommand{\R}{\mathbb{R}}
\newcommand{\N}{\mathbb{N}}
\newcommand{\dx}{\textup{d}x}

\newcommand{\calL}{\mathcal{L}}

\newcommand{\scal}{\textup{scal}}
\newcommand{\calU}{\mathcal{U}}
\newcommand{\calH}{\mathcal{H}}
\newcommand{\PotDann}{\omega}
\newcommand{\dist}{\textup{dist}}
\newcommand{\Sf}{\mathbb{S}}
\newcommand{\Fm}{\mathcal{F}}
\newcommand{\callF}{\mathcal{F}}

\newcommand{\Hm}{\mathcal{H}}
\newcommand{\Lm}{\mathcal{L}}
\newcommand{\eps}{\varepsilon}
\newcommand{\Id}{\mathrm{Id}}
\newcommand{\qc}{\mathrm{qc}}
\newcommand{\qcinfty}{{\mathrm{qc},\infty}}

\newcommand\weakto{\rightharpoonup}
\newcommand{\abs}[1]{\lvert #1\rvert}

\newcommand{\cbra}[1]{\left\{#1\right\}}
\newcommand{\diff}{\mathop{}\!\mathrm{d}}
\newcommand\res{\mathop{\hbox{\vrule height 7pt width .5pt depth 0pt \vrule height .5pt width 6pt depth 0pt}}\nolimits}

\newtheorem{theorem}{Theorem}[section]

\newtheorem{proposition}[theorem]{Proposition}
\newtheorem{lemma}[theorem]{Lemma}

\newtheorem{remark}[theorem]{Remark}

\title{Phase-Field Approximation of Vectorial Cohesive Models}

\author{Francesco Colasanto}
\address{Technical University of Munich (TUM), Department of Mathematics,
85748 Garching, Munich, Germany}
\email{francesco.colasanto@tum.de}

\date{\today}

\begin{document}

\begin{center}
    \begin{abstract}
        We study the $\Gamma$-convergence of a general class of anisotropic phase-field functionals modelling cohesive fracture in the multidimensional, vector-valued, and geometrically nonlinear setting. The limiting energy comprises a quasiconvexified bulk term, a Cantor part governed by its recession function, and a cohesive surface energy density defined via an asymptotic cell problem. After establishing the structural properties and equivalent formulations of the limiting surface density, we show that, whenever the recession function of the elastic energy is given by the square of an operator norm compatible with the phase-field anisotropy, the vectorial cell problem explicitly reduces to a one-dimensional variational problem. As a consequence, we extend the cohesive law reconstruction procedure of \cite{alessi2025phasefieldpart2} to the vectorial framework, explicitly designing phase-field potentials that recover prescribed classes of anisotropic, jump-dependent traction-separation laws of the form $g(\zeta,\nu)=\gamma(\nu)g_{\scal}(|\zeta|)$.
    \end{abstract}
\end{center}
\maketitle

\section{Introduction}

Free-discontinuity energies provide a natural variational framework for the analysis of fracture, damage, image segmentation, and other interfacial phenomena. In the context of fracture mechanics, one is naturally led to consider functionals in which a bulk elastic term competes with a surface contribution concentrated on the jump set of the deformation. In the brittle case, this paradigm traces back to Griffith's theory \cite{Gri21} and its variational reformulation in \cite{FM98}, which was later developed within the general framework of the variational approach to fracture \cite{AB95,BFM08}. In cohesive fracture, the surface density depends explicitly on the crack opening, thereby retaining information regarding the amplitude of the jump. This distinctive feature makes cohesive models substantially richer and more complex than their brittle counterparts, from both mechanical and analytical standpoints. Classical cohesive laws were studied in \cite{Dug60,Bar62}, and a modern mathematical discussion of their role in fracture modelling can be found, for instance, in \cite{BFM08,Mar23}.

A major breakthrough in the variational approximation of free-discontinuity energies was the phase-field regularization introduced in \cite{AT90,AT92} for the Mumford--Shah functional. In the fracture setting, this strategy was adapted to brittle energies in \cite{BFM00,Foc01}, providing a regular family of elliptic functionals whose minimizers can be effectively computed and whose singular limit rigorously reproduces the sharp-interface model. The importance of this approach is now well established: on the one hand, $\Gamma$-convergence establishes a rigorous bridge between diffuse and sharp-interface descriptions \cite{dalmaso1993introduction,Bra98}; on the other hand, the diffuse formulation is substantially more tractable in numerical simulations and has become a standard tool in computational fracture mechanics (cf. \cite{PM10b,PM10a}). This dual utility explains the sustained interest in phase-field models across both the mathematical literature and the engineering community.

The cohesive case, however, is considerably subtler. From the modelling perspective, gradient-damage and phase-field descriptions of cohesive fracture have been proposed and studied in several forms; see, among many others, \cite{LCK11,LG11,Wu17,CdB21,WNN20,LCM23,LCM25}. A key challenge is that a faithful cohesive approximation must simultaneously encode the elastic response, the damage localization mechanism, and the effective traction-separation law. Crucially, the regularization should allow for the decoupling of fracture toughness, critical stress, and internal length, a requirement that is paramount in applications but absent in the classical brittle Ambrosio-Tortorelli scaling. The scalar case was placed on a rigorous variational foundation in \cite{CFI16}, where a phase-field approximation of cohesive fracture energies was proven to converge to a free-discontinuity functional characterized by a genuinely cohesive surface density. More recently, a broader scalar framework was developed in a comprehensive three-part work; the first part \cite{ACF26} demonstrated that a single family of phase-field energies can encompass several models from the mechanics literature (including those in \cite{CFI16,Wu17,FFL21,LCM23,LCM25}) and can also reproduce brittle-type regimes through a suitable tuning of the scaling parameters. Building upon these theoretical foundations,  \cite{alessi2025phasefieldpart2} presents a systematic procedure for constructing phase-field models that reproduce prescribed cohesive laws $g_0$, that satisfy the following:
\begin{equation}\label{e:ipotesi legge coesiva}
\begin{cases}
    & \textup{$g_0\in C^1([0,\infty))$, $g^{-1}_0(\{0\})=\{0\}$, $g_0$ bounded and non-decreasing;} \\ 
    & \textup{$g_0$ is concave, $g_0'(0)=\varsigma_0\in (0,\infty)$, $g'_0$ is strictly decreasing on $[0,s_{\textup{frac}})$}\\ 
    & \textup{$R\in W^{1,p}((0,\varsigma^2_0))$ for some $p>2$, and it is convex},
\end{cases}
\end{equation}
where 
\begin{equation*}
    s_{\textup{frac}}:=\sup \{s\in [0,\infty) \, : \, g_0(s)<\lim_{t\to \infty}g_0(t)\}
\end{equation*}
and $R:[0,\varsigma_0^2]\to [0,\infty)$ is given by
\begin{equation*}
    R(t)=
    \begin{cases}
        g_0((g'_0)^{-1}(\sqrt{\varsigma_0^2-t})) & \, \textup{if $t\in [0,\varsigma_0^2)$} \\ 
        g_0(s_{\textup{frac}}) & \, \textup{if $t=\varsigma_0^2$}. \\ 
    \end{cases}
\end{equation*}

 Finally, \cite{alessi2026phasefieldpart3} explores physically relevant applications and validates the theoretical results from an engineering perspective.

In higher dimensions and in the vector-valued setting, the scenario is significantly more delicate. The transition from scalar $BV$ functions to vector-valued $GBV$ maps is accompanied by genuine geometric and analytical difficulties: the surface density generally depends on both the jump amplitude and the orientation, and the bulk part of the limit functional necessitates quasiconvexification. A rigorous vectorial result for a geometrically nonlinear cohesive phase-field model was obtained in \cite{CFI24} for the linear cohesive regime. The full $\Gamma$-limit was identified in terms of a quasiconvex bulk density, its recession function, and a surface density derived via a localized cell formula. This achievement leaves open the problem of extending such analysis to the more flexible scalar families introduced in \cite{ACF26}, and of understanding in a unified manner how the parameters governing the diffuse model influence the limiting cohesive law in the multidimensional vectorial case.

The purpose of this paper is precisely to address this issue within the linear cohesive regime. We consider the family of functionals
\begin{equation}\label{eq:intro-Feps}
\mathcal F_\varepsilon(u,v;A):=
\int_A \Big(\varphi(\varepsilon f^2(v))\,\Phi(\nabla u)
+\frac{\omega(1-v)}{4\varepsilon}+\varepsilon \gamma(\nabla v)^2\Big)\,\diff x,
\end{equation}
defined on pairs $(u,v)\in H^1(\Omega,\mathbb R^N\times [0,1])$, where $\Omega\subset\mathbb R^n$ is a bounded and Lipschitz domain, $u$ is the displacement field, and $v$ is the phase-field variable. Here, $\Phi$ is a continuous bulk density possessing quadratic growth and a quadratic recession function, the scalar factor $\varphi(\varepsilon f^2(v))$ dictates the degradation of the elastic response, while $\gamma$ is a generic norm on $\R^n$. The model is governed by three asymptotically relevant quantities: the saturation value $\varphi(\infty)$, the initial slope $\varphi'(0)$, and the parameter
\begin{equation}
\varsigma:=\lim_{t\to0^+}\Big(\frac{\omega(t)}{Q(t)}\Big)^{1/2},
\end{equation}
which quantifies the balance between the damage penalization and the elastic degeneracy. In the linear cohesive regime $\varsigma\in(0,\infty)$, these parameters determine the effective bulk density through a specific function $h_{\varsigma}$, and the effective small-jump slope through the combination $(\varphi'(0))^{1/2}\varsigma$. Simultaneously, $\varphi(\infty)$ controls the large-gradient regime of the diffuse model, thereby orchestrating the transition towards brittle or degenerate asymptotic behaviours in the extreme regimes.

Inspired by the methodologies introduced in \cite{CFI16} and \cite{CFI24}, our main theorem extends the vectorial $\Gamma$-convergence result of \cite{CFI24} to this general class of phase-field energies. Specifically, we prove that, for $\varsigma\in(0,\infty)$, the functionals $\mathcal F_\varepsilon$ converge in the strong $L^1$-topology to a free-discontinuity functional of the form
\begin{equation}
\mathcal F_{\varsigma}(u,1;A)=
\int_A h_{\varsigma}^{\qc}(\nabla u)\,\diff x
+\int_A h_{\varsigma}^{\qc,\infty}(\diff D^c u)
+\int_{J_u\cap A} g([u],\nu_u)\,\diff\mathcal H^{n-1},
\end{equation}
for $u\in (GBV(\Omega))^N$, and $\infty$ otherwise. In this expression, $h_{\varsigma}^{\qc}$ denotes the quasiconvex envelope of the effective bulk density, $h_{\varsigma}^{\qc,\infty}$ its recession function, and $g$ represents the cohesive surface density. 

A distinctive feature of this formulation, with respect to \cite{ACF26}, is the presence of a general norm $\gamma$ on $\mathbb{R}^n$ penalizing the phase-field gradient through the term $\epsilon\gamma(\nabla v)^2$. From a mechanical and geometrical standpoint, $\gamma$ encodes the directional dependence of the fracture toughness, allowing the model to capture material anisotropies. In Proposition~\ref{prop:rho-gamma-identification}, under the assumption $\Phi_{\infty}(\xi) = \|\xi\|_{\rho,\gamma}^2$ (cf. \eqref{eq:rho-gamma-tensor-norm}), we prove that the surface energy density reduces to the explicit split form
\begin{equation}
 g(\zeta,\nu)=\gamma(\nu)g_{\rm scal}(\rho(\zeta)),
\end{equation}
where $\rho$ is a norm on $\R^N$ and $g_{scal}$ is defined by the one-dimensional variational problem
\begin{equation}\label{e:definizione di g scal}
    g_{\textup{scal}}(s):=\inf_{(\alpha,\beta) \in \mathcal{U}_{s}(0,1)} 
    \int_0^1 \PotDann(1-\beta)^{\frac12}\Big(\varphi'(0)\frac{\ell(\beta)}{Q(1-\beta)}|\alpha'|^2 +|\beta'|^2\Big)^{\frac12}\,\mathrm{d}t,
\end{equation}
over the admissible functional space
\begin{align}\label{e:glispaziu}
  \mathcal{U}_{s}(T_0,T_1) := \Big\{(\alpha,\beta) \in H^1((T_0,T_1),\R^2) :\, \alpha(T_0)=0,\ \alpha(T_1)=s,\ 0\leq \beta \leq 1,\ \beta(T_0)=\beta(T_1)=1\Big\},
\end{align}
for $T_0, T_1 \in \mathbb{R}$. By doing so, we decouple the directional anisotropy $\gamma(\nu)$ from the opening resistance $g_{\scal}(\rho(\zeta))$ and exactly recover the one-dimensional cohesive energy density along any arbitrary jump direction. Crucially, this dimensional reduction bridges the gap between the vectorial theory and the one-dimensional inverse design strategy developed in~\cite{alessi2025phasefieldpart2}. It enables us to explicitly solve the reconstruction problem in the special vectorial anisotropic setting: given any target cohesive law of the form $g_0(\zeta,\nu) = \gamma_0(\nu)\hat{g}_0(|\zeta|)$ with $\hat{g}_0$ satisfying~\eqref{e:ipotesi legge coesiva},
we can systematically construct the phase-field potentials, either by prescribing the degradation function or the damage potential, so that the corresponding diffuse functionals $\Gamma$-converge to a sharp-interface model endowed with the exact surface density $g_0$ (cf.\ Theorems~\ref{thm:ricostruzione fissata degrad} and~\ref{thm:ricostruzione fissata damage pot}).

Beyond the principal $\Gamma$-convergence statement, this paper provides a comprehensive analysis of the surface term. We establish an equivalent large-cell characterization of $g$, a reduction to competitors satisfying mollified boundary data and tangential periodicity, as well as quantitative growth estimates, subadditivity, and continuity with respect to both the jump amplitude and the normal direction (cf. Section \ref{sec:surface-density}). These properties are instrumental for the proof of the lower bound and are especially critical for the gluing constructions utilized in the upper bound.

We point out that, as in \cite{CFI24}, our analysis is entirely framed within the geometrically nonlinear setting, meaning that the elastic energy penalizes the full deformation gradient. The passage to the linearized setting, which depends strictly on the symmetrized gradient, would naturally lead to the space of generalized functions of bounded deformation $(G)BD$. Adapting our strategy to the linearized framework is far from trivial. Indeed, a key ingredient of the proofs, both for the lower and the upper bound, is the use of a De Giorgi-type truncation argument in the codomain. It is a well-known intrinsic difficulty of $BD$ spaces that the symmetrized gradient of a truncated vector field cannot be bounded solely by the symmetrized gradient of the original function, thereby preventing a direct translation of the truncation techniques.

Finally, it is worth emphasizing once again that the linear cohesive regime addressed herein should be considered the natural counterpart, within this general framework, to the model analyzed in \cite{CFI24}. For what concerns the superlinear regime $\varsigma=\infty$ (cf. \eqref{eq:def-varsigma}), it introduces additional analytical issues: as shown in \cite{CFI25}, the full vectorial $\Gamma$-convergence analysis naturally gives rise to two surface densities, $g_{\inf}$ and $g_{\sup}$, which can be shown to coincide under suitable technical assumptions (such as the projection property of the recession function). While the linear and superlinear theories share a common underlying logic, this technical nuance highlights subtle differences in how sharp-interface relaxations and diffuse approximations behave across the two settings, thus motivating our choice to focus the present analysis specifically on the linear regime.

\section{Setting and statement of results}\label{sec:setting}

Throughout the work, $n\geq 2$ and $N\geq 1$ denote fixed positive integers, while $\Omega\subset \mathbb R^n$ is a bounded, open domain with Lipschitz boundary.   
The symbols $\cdot$ and $|\cdot|$ stand, respectively, for the Euclidean scalar product and the associated norm on $\mathbb{R}^d$, and $\{e_1,\dots,e_d\}$ denotes the standard basis of $\mathbb{R}^d$, for every $d\in \N$. For any two vectors $a \in \R^N$ and $b \in \R^n$, we denote by $a \otimes b \in \R^{N \times n}$ the matrix whose entries are given by 
$$(a \otimes b)_{ij} := a_i b_j$$ 
for $1 \le i \le N$ and $1 \le j \le n$.
For any real number $t$, the notation $\lfloor t \rfloor$ denotes its integer part.

\vspace{0.3em}
We denote by 
\[
\mathbb{S}^{n-1}:=\{y\in\mathbb{R}^n : |y|=1\}
\]
the $(n-1)$–dimensional unit sphere in $\mathbb{R}^n$. For every $\nu\in\mathbb{S}^{n-1}$, we denote by $\Pi_\nu\subset\mathbb{R}^n$ the hyperplane orthogonal to $\nu$, namely
\begin{equation}\label{e:definizione dell'iperpiano nu}
\Pi_\nu := \{y\in\mathbb{R}^n : y\cdot \nu=0\},
\end{equation}
and for all $\zeta\in\mathbb{R}^N$, we define the \emph{jump function} $u_{\zeta,\nu}\colon\mathbb{R}^n\to\mathbb{R}^N$ by
\begin{equation}\label{e:definizione della funzione di salto nei preliminari}
u_{\zeta,\nu}(y):=
\begin{cases}
\zeta, & \text{if } y\cdot\nu \ge 0, \\[4pt]
0, & \text{if } y\cdot\nu < 0,
\end{cases}
\end{equation}
that represents a simple discontinuity with jump $\zeta$ across the plane $\Pi_\nu$ and will be repeatedly used in the cell formula defining the surface energy density of limit functionals.

A mapping \(\gamma\colon\R^n\to[0,\infty)\) is a
\emph{norm} if
\begin{enumerate}
\item \(\gamma(p)=0\) if and only if \(p=0\);
\item \(\gamma(\lambda p)=\abs{\lambda}\gamma(p)\) for all
      \(p\in\R^n\) and \(\lambda\in\R\);
\item \(\gamma(p+q)\leq\gamma(p)+\gamma(q)\) for all
      \(p,q\in\R^n\).
\end{enumerate}
Its \emph{dual norm} is the mapping
\(\gamma^\circ\colon\R^n\to[0,\infty)\) defined by
\begin{equation}\label{eq:def-polar-norm}
 \gamma^\circ(r)
 :=
 \sup\cbra{|p\cdot r|:\gamma(p)\leq1},
 \end{equation}
 and in particular the following holds
 \begin{equation}\label{e:polar inequality}
     |p\cdot r|\leq \gamma(p)\gamma^\circ(r), \qquad \textup{for every $p,r\in \R^n$}.
 \end{equation}
 Given $\rho:\R^N\to [0,\infty)$ a norm, we consider the operatorial norm $\|\cdot\|_{\rho,\gamma}:\R^{N\times n}\to [0,\infty)$, given by 
 \begin{equation}\label{eq:rho-gamma-tensor-norm}
 \|\xi\|_{\rho,\gamma}
 :=
 \sup\left\{
 \gamma(\xi^{\mathsf T}a)\colon
 a\in\R^N,\ \rho^\circ(a)\le1
 \right\},
\end{equation}
endowing $\R^N$ of the norm $\rho^\circ$ and $\R^n$ of $\gamma$.

\vspace{0.3em}
For $x\in\mathbb{R}^n$ and $r>0$, we denote the open ball of center $x$ and radius $r$ by 
\begin{equation*}
B_{r}(x):=\{y\in\mathbb{R}^n : |y-x|<r\}, 
\end{equation*}
with the following convention $B_r:=B_r(0)$. Moreover, for every $\nu\in\mathbb{S}^{n-1}$, we denote by $Q^{\nu}_{r}(x)$ an open cube of side length $r$ centered at $x$, with one of its faces orthogonal to $\nu$, adopting the shorthand notations $Q^{\nu}_r:=Q^{\nu}_{r}(0)$ and $Q^{\nu}:=Q^{\nu}_1$.

If $O$ is an open subset of $\mathbb{R}^n$, we denote by $\mathcal{A}(O)$, $\mathcal{A}_{\infty}(O)$, and $\mathcal{B}(O)$ the families of bounded open, Lipschitz bounded open, and Borel subsets of $O$, respectively. 
For every $B\subseteq\mathbb{R}^d$, we write $\mathbf 1_B\colon\mathbb{R}^n\to\{0,1\}$ for the characteristic function of $B$. If $\mu$ is a Borel measure and $B$ a Borel set, the restriction of $\mu$ to $B$ is denoted by $\mu\res B$ and is defined by $\mu\res B(A):=\mu(A\cap B)$. 
We denote by $\mathcal{L}^n$ the $n$–dimensional Lebesgue measure in $\mathbb{R}^n$ and by $\mathcal{H}^{n-1}$ the $n-1$–dimensional Hausdorff measure on $\R^n$. 

To handle the lack of a priori $L^\infty$ bounds and to extend our analysis to the generalized space $GBV$, we will extensively rely on a family of smooth truncation maps. Specifically, let $(a_k)_{k\in\mathbb{N}}$ be a strictly increasing sequence of positive numbers with $a_k \uparrow \infty$, such that the operators $\mathcal{T}_k \in C^1(\mathbb{R}^N; \mathbb{R}^N)$ defined by
\begin{equation}\label{e:Tk}
\mathcal{T}_k(\zeta):=
\begin{cases}
\zeta, & \text{if } |\zeta| \le a_k,\\[4pt]
0, & \text{if } |\zeta| \ge a_{k+1},
\end{cases}
\end{equation}
satisfying the uniform Lipschitz bound $\textup{Lip}(\mathcal{T}_k)\le 1$.

Finally, we employ the standard notation for Lebesgue spaces $L^p(O,\mathbb{R}^k)$ ($L_{\textup{loc}}^p(O,\mathbb{R}^k)$) and Sobolev spaces $W^{1,p}(O,\mathbb{R}^k)$ ($W_{\textup{loc}}^{1,p}(O,\mathbb{R}^k)$) for $p\in[1,\infty]$, $k\in \N$, and any open set $O\subseteq \R^n$, adopting the usual convention $H^1(O,\mathbb{R}^k):=W^{1,2}(O,\mathbb{R}^k)$. Furthermore, we rely on the customary notations for the spaces of (generalized) functions of bounded variation and special bounded variation, respectively denoted by $(G)BV(O,\mathbb{R}^k)$ and $(G)SBV(O,\mathbb{R}^k)$; for a comprehensive treatment of these spaces and their fine properties, we refer to \cite{ambrosiofuscopallara}.

\subsection{The vectorial phase-field energies}
Our framework is built upon the structural potentials $l$, $Q$, $\omega$, and $\varphi$, alongside the elastic bulk density $\Phi$ and the damage norm $\gamma$. Following the scalar approach developed in \cite{ACF26}, we introduce the function
\begin{equation}\label{e:def f}
 f(t):=\Big(\frac{l(t)}{Q(1-t)}\Big)^{1/2},\qquad t\in [0,1),
\end{equation}
where:
\begin{itemize}
\item[\textup{(H1)}]
 \(l,Q\in C^0([0,1];[0,\infty))\),
 \(l(1)=1\),
 \(l^{-1}(\{0\})=Q^{-1}(\{0\})=\{0\}\);
 \(Q\) is nondecreasing in a right neighbourhood of \(0\), and \(f\) is
 nondecreasing in a right neighbourhood of \(0\) and in a left neighbourhood
 of \(1\).

\item[\textup{(H2)}]
 \(\omega\in C^0([0,1];[0,\infty))\) is nondecreasing,
 \(\omega^{-1}(\{0\})=\{0\}\), and
 \begin{equation}\label{eq:def-varsigma}
  \varsigma
  :=
  \lim_{t\downarrow0}
  \left(\frac{\omega(t)}{Q(t)}\right)^{1/2}
  \in[0,\infty].
 \end{equation}

\item[\textup{(H3)}]
 \(\varphi\) is continuous, nondecreasing, and bounded, with
 \begin{equation}\label{eq:def-varphi-infinity}
  \varphi(\infty)
  :=
  \lim_{t\to\infty}\varphi(t)\in(0,\infty).
 \end{equation}

\item[\textup{(H4)}]
 \(\varphi^{-1}(\{0\})=\{0\}\) and \(\varphi\) is right-differentiable at
 \(0\), with
 \begin{equation}\label{eq:def-a}
  \varphi'(0)\in(0,\infty).
 \end{equation}

\item[\textup{(H5)}]
 \(\Phi\in C^0(\R^{N\times n},[0,\infty))\) and there exists \(C_\Phi\geq1\)
 such that
 \begin{equation}\label{eq:growth-Phi}
  \left(\frac1{C_\Phi}\abs\xi^2-C_\Phi\right)\vee0
  \leq\Phi(\xi)
  \leq C_\Phi(1+\abs\xi^2)
  \qquad\text{for every }\xi\in\R^{N\times n}.
 \end{equation}
 Moreover, the limit
 \begin{equation}\label{eq:def-Phi-infty}
  \Phi_\infty(\xi)
  :=
  \lim_{t\to\infty}\frac{\Phi(t\xi)}{t^2}
 \end{equation}
 exists uniformly for \(\abs\xi=1\).
 The uniform convergence in \eqref{eq:def-Phi-infty} implies that
\(\Phi_\infty\) is continuous, positively two-homogeneous, and it gives
\begin{equation}\label{eq:growth-Phi-infty}
 \frac1{C_\Phi}\abs\xi^2
 \leq\Phi_\infty(\xi)
 \leq C_\Phi\abs\xi^2
 \qquad\text{for every }\xi\in\R^{N\times n}.
\end{equation}
Moreover, for every \(\delta>0\), there exists \(t_\delta>0\) such that
\begin{equation}\label{eq:uniform-recession-estimate}
 \abs{\Phi(\xi)-\Phi_\infty(\xi)}
 \leq\delta\abs\xi^2
 \qquad\text{whenever }\abs\xi\geq t_\delta.
\end{equation}

\item[\textup{(H6)}]
 \(\gamma\colon\R^n\to[0,\infty)\) is a norm.
\end{itemize}
Since all norms on a finite-dimensional vector space are equivalent, there
exist constants \(0<c_\gamma\leq C_\gamma<\infty\) such that
\begin{equation}\label{e:gamma lower euclidean suppenergy}
 c_\gamma\abs p
 \leq\gamma(p)
 \leq C_\gamma\abs p
 \qquad\text{for every }p\in\R^n.
\end{equation}

For \(\eps>0\) and \(A\subset\Omega\) open set, we define
\begin{equation}\label{eq:def-Feps-gamma}
 \Fm_\eps(u,v;A)
 :=
 \begin{cases}
 \displaystyle
 \int_A
 \left(
  \varphi(\eps f^2(v))\Phi(\nabla u)
  +\frac{\omega(1-v)}{4\eps}
  +\eps\gamma(\nabla v)^2
 \right)\,\diff x,
 &(u,v)\in H^1(\Omega,\R^N)\times H^1(\Omega,[0,1]),\\[1.2em]
 +\infty,
 &\text{otherwise in }L^1(\Omega;\R^{N+1}).
 \end{cases}
\end{equation}
We extend \(t\mapsto\varphi(\eps f^2(t))\) at \(t=1\) by the value
\(\varphi(\infty)\).

We consider the following bulk density 
\begin{equation}\label{eq:def-h-varsigma}
 h_\varsigma(\xi)
 :=
 \inf_{\tau\in[0,\infty)}
 \left\{
  \varphi\!\left(\frac1\tau\right)\Phi(\xi)
  +\frac{\varsigma^2}{4}\tau
 \right\},
\end{equation}
where \(\varphi(1/\tau):=\varphi(\infty)\) at \(\tau=0\) and we adopt the convention $0\cdot \infty=0$.
We denote by \(h_\varsigma^\qc\) its quasiconvex envelope and set
\begin{equation}\label{eq:def-h-recession}
 h_\varsigma^{\qcinfty}(\xi)
 :=
 \limsup_{t\to\infty}\frac{h_\varsigma^\qc(t\xi)}{t}.
\end{equation}

For every open set \(O\subset\R^n\), define the recession phase-field energy
\begin{equation}\label{eq:def-recession-energy}
 \Fm_\eps^{\infty}(u,v;O)
 :=
 \begin{cases}
 \displaystyle
 \int_O
 \left(
  \varphi'(0)\eps f^2(v)\Phi_\infty(\nabla u)
  +\frac{\omega(1-v)}{4\eps}
  +\eps\gamma(\nabla v)^2
 \right)\,\diff x,
 &(u,v)\in H^1(O,\R^N)\times H^1(O,[0,1]),\\[1.2em]
 +\infty,&\text{otherwise},
 \end{cases}
\end{equation}
still adopting the convention $0\cdot \infty=0$.

For \((\zeta,\nu)\in\R^N\times\Sf^{n-1}\), we define
\begin{equation}\label{eq:def-g-gamma}
 g(\zeta,\nu)
 :=
 \inf\Big\{
 \liminf_{j\to\infty}
 \Fm_{\eps_j}^{\infty}(u_j,v_j;Q^\nu):
 \eps_j\downarrow0,\\
 (u_j,v_j)\to(u_{\zeta,\nu},1)
 \text{ in }L^1(Q^\nu;\R^{N+1})
 \Big\}.
\end{equation}
The convergence of the second
component could be omitted from the definition: finite energy automatically
implies \(v_j\to1\) in \(L^1(Q^\nu)\). The last assertion follows directly from \textup{(H2)}.  Indeed, for every
\(\delta\in(0,1)\),
\[
 \Lm^n(\{v_j\leq1-\delta\})
 \frac{\min_{[\delta,1]}\omega}{4\eps_j}
 \leq
 \int_{Q^\nu}\frac{\omega(1-v_j)}{4\eps_j}\,\diff x.
\]
If the energies are bounded, the measure of the set on the left converges to
zero.  Since \(0\leq v_j\leq1\), convergence in measure to \(1\) implies
convergence in every \(L^p\)-space with $p\in (0,\infty)$.

The candidate \(\Gamma\)-limit is
\begin{equation}\label{eq:def-limit-functional}
 \Fm_\varsigma(u,v;A)
 :=
 \begin{cases}
 \displaystyle
 \int_A h_\varsigma^\qc(\nabla u)\,\diff x
 +\int_A h_\varsigma^{\qcinfty}
 \left(\frac{\diff D^c u}{\diff\abs{D^c u}}\right)
 \,\diff\abs{D^c u}
 +\int_{J_u\cap A}
 g([u],\nu_u)\,\diff\Hm^{n-1},
 &\begin{array}{l}
   u\in(GBV(\Omega))^N,\\[-0.2em]
   v=1\quad\Lm^n\text{-a.e. in }\Omega,
  \end{array}\\[1.4em]
 +\infty,&\text{otherwise}.
 \end{cases}
\end{equation}

We now state the main $\Gamma$-convergence theorem of the work, which characterizes the asymptotic behaviour of the phase-field functionals $\mathcal{F}_{\epsilon}$ in the linear cohesive regime $\varsigma \in (0,\infty)$.
\begin{theorem}\label{thm:general-anisotropic-gamma-convergence}
Assume \textup{(H1)--(H6)} and let
\(\varsigma\in(0,\infty)\).  Then
\begin{equation}\label{eq:general-gamma-convergence}
\Big( \Gamma\big(L^1\big)
 \text{-}\lim_{\eps\downarrow0}\Fm_\eps\Big)(u,v;\Omega)
 =\Fm_\varsigma(u,v;\Omega),
\end{equation}
where \(\Fm_\eps\) and
\(\Fm_\varsigma\) are defined by
\eqref{eq:def-Feps-gamma} and \eqref{eq:def-limit-functional},
respectively.
\end{theorem}

Furthermore, using exactly the same argument of \cite[Theorem 6.1]{CFI24} we obtain the following compactness result 
\begin{theorem}\label{t:comp}
	 If $(u_\eps,v_\eps)\in {H^1}(\Omega,\R^{N+1})$ is such that
    \begin{equation}\label{e:prima eq compattezza vector}
        \sup_\eps \left(\callF_{\eps}(u_\eps,v_\eps,\Omega)+\|u_\eps\|_{L^1(\Omega,\R^N)}\right)<\infty,
    \end{equation}
then, under the assumptions of Theorem~\ref{thm:general-anisotropic-gamma-convergence}, there exists a subsequence $(u_{\eps_j},v_{\eps_j})$ of $(u_\eps,v_\eps)$ and a function $u\in (GBV\cap L^1(\Omega))^N$ 
	such that $u_{\eps_j}\to u$ $\calL^n$-a.e. and $v_{\eps_j}\to 1$ in $L^1(\Omega)$.
\end{theorem}

\subsection{Some technical results}

We start by collecting some properties of the effective bulk density $h_{\sigma}$.
\begin{proposition}\label{p:proprieta hsigma}
Assume \textup{(H3)-(H4)}. Then $h_{\sigma_0} \in C^{0}(\mathbb R^{N\times n})$ (see \eqref{eq:def-h-varsigma}) for every $\sigma_0 \in (0,\infty)$, and it satisfies the bounds
\begin{equation}\label{e:la h al variare di sigma}
        h_{\sigma_0}(\xi)\leq h_{\sigma_1}(\xi)\leq (\sigma_1/\sigma_0)^2h_{\sigma_0}(\xi) \quad \quad \textup{for every $\xi\in \mathbb R^{N\times n}$},
    \end{equation}
for all pairs $\sigma_0,\sigma_1 \in (0,\infty)$ with $\sigma_0\leq \sigma_1$. 
\end{proposition}

\begin{proof}
We fix $\sigma_0 \in (0,\infty)$ and first verify the continuity of $h_{\sigma_0}$. Let $\xi_0 \in \mathbb R^{N \times n}$ and consider an arbitrary sequence $\xi_k \to \xi_0$ as $k \to \infty$. For each index $k$, we can select $\tau_k \in [0,\infty)$ such that
\begin{equation}\label{e:prima eq prop di h}
\varphi\left(\tfrac{1}{\tau_k}\right)\Phi(\xi_k) + \tfrac{\sigma_0^2}{4}\tau_k - \tfrac{1}{k} < h_{\sigma_0}(\xi_k).
\end{equation}
Since $\sup_{k \in \mathbb N} \Phi(\xi_k) < \infty$ and $\Phi$ is continuous, definition \eqref{eq:def-h-varsigma} ensures that the sequence $(\tau_k)$ is bounded, i.e., $\limsup_{k \to \infty} \tau_k < \infty$. Hence, we may extract a convergent subsequence $(\tau_{k_j})$ such that $\tau_{k_j} \to \tau_0$ for some $\tau_0 \in [0,\infty)$. Relying on the continuity of both $\varphi$ and $\Phi$, passing to the limit in \eqref{e:prima eq prop di h} yields
\begin{equation*}
h_{\sigma_0}(\xi_0) \leq \varphi\left(\frac{1}{\tau_0}\right)\Phi(\xi_0) + \frac{\sigma_0^2}{4}\tau_0
\leq \liminf_{j \to \infty} h_{\sigma_0}(\xi_{k_j}).
\end{equation*}
Since the point $\xi_0$ and the converging sequence $(\xi_k)$ were chosen arbitrarily, we conclude that $h_{\sigma_0}$ is lower semicontinuous. The upper semicontinuity is a direct consequence of the definition of $h_{\sigma_0}$ as an infimum of continuous functions and the continuity of $\Phi$.

Now consider $\sigma_0, \sigma_1 \in (0,\infty)$ with $\sigma_0 \leq \sigma_1$. The first inequality in \eqref{e:la h al variare di sigma} follows trivially from the definition \eqref{eq:def-h-varsigma}. To verify the second, fix any $\xi \in \mathbb R^{N \times n}$. By the continuity of $\varphi$, the infimum is attained, so there exists $\tau_{\sigma_0,\xi} \in [0,\infty)$ satisfying
\begin{equation*}
h_{\sigma_0}(\xi) = \varphi\left(\frac{1}{\tau_{\sigma_0,\xi}}\right)\Phi(\xi) + \frac{\sigma_0^2}{4}\tau_{\sigma_0,\xi}.
\end{equation*}
Exploiting the fact that $(\sigma_1/\sigma_0)^2 \geq 1$, we can easily deduce
\begin{equation*}
h_{\sigma_1}(\xi) \leq \varphi\left(\frac{1}{\tau_{\sigma_0,\xi}}\right)\Phi(\xi) + \frac{\sigma_1^2}{4}\tau_{\sigma_0,\xi}
\leq \left(\frac{\sigma_1}{\sigma_0}\right)^2 h_{\sigma_0}(\xi),
\end{equation*}
which validates the claimed inequality.
\end{proof}

The following result shows that the quadratic growth of $\Phi$ translates into linear bounds for the effective bulk density $h_{\varsigma}$, its quasiconvex envelope $h_{\varsigma}^{\qc}$, and the associated recession function $h_{\varsigma}^{\qc,\infty}$.
\begin{proposition}\label{prop:bulk-growth-section2}
Assume \textnormal{(H3)}-\textnormal{(H5)} and let $\sigma\in (0,\infty)$. Then there exists a positive constant $K_{\sigma}>0$ such that the following growth estimates hold:
\begin{equation}\label{e:bulk-growth-h}
 \Big(\frac1{K_{\sigma}}|\xi|-K_{\sigma}\Big)\vee 0
 \le h_{\sigma}^{\qc}(\xi)
 \le h_{\sigma}(\xi)
 \le K_{\sigma}(1+|\xi|)
 \qquad\text{for all }\xi\in \mathbb R^{N\times n},
\end{equation}
and correspondingly for the recession function,
\begin{equation}\label{e:bulk-growth-hrecession}
 \frac1{K_{\sigma}}|\xi|\le h_{\sigma}^{\qc,\infty}(\xi)\le K_{\sigma}|\xi|
 \qquad\text{for all }\xi\in \mathbb R^{N\times n}.
\end{equation}
\end{proposition}

\begin{proof}
Let us introduce the auxiliary scalar function defined for $t\in \mathbb R$ as
\begin{equation*}
\tilde h_{\textup{scal}}(t):=\inf_{\tau\in [0,\infty)}
\left\{\varphi\left(\frac1\tau\right)\frac{1}{C_\Phi}t^2 +\frac{\sigma^2}{4}\tau\right\},
\end{equation*}
where $C_\Phi$ is the coercivity constant appearing in \eqref{eq:growth-Phi}. From this definition, it immediately follows that
\begin{equation*}
\tilde h_{\textup{scal}}(|\xi|)\leq h_{\sigma}(\xi)+C_\Phi\,\varphi(\infty)
\qquad\text{for every }\xi\in \mathbb R^{N\times n}.
\end{equation*}
Invoking \cite[Proposition 2.1, item ii]{ACF26}, we can find a sufficiently large constant $K_{\sigma}$ such that
\begin{equation*}
    \Big(\frac1{K_{\sigma}}|\xi|-K_{\sigma}\Big)\lor 0
 \le  ({\tilde h}^{**}_{\textup{scal}}(|\xi|)-C_\Phi\varphi(\infty))\lor 0 \leq h_{\sigma}^{\qc}(\xi),
\end{equation*}
where $\tilde h^{**}_{\textup{scal}}$ denotes the convex envelope of $\tilde h_{\textup{scal}}$. In a similar way, we define another scalar function to govern the upper bound:
\begin{equation*}
\hat h_{\textup{scal}}(t):=\inf_{\tau\in[0,\infty)}
\left\{\varphi\left(\frac1\tau\right)C_\Phi t^2 +\frac{\sigma^2}{4}\tau\right\},
\qquad t\in\mathbb R.
\end{equation*}
Thanks to \eqref{eq:growth-Phi}, we trivially have $h_{\sigma}(\xi)\leq \hat h_{\textup{scal}}(|\xi|)+C\varphi(\infty)$ for every $\xi\in\mathbb R^{N\times n}$. Furthermore, applying again \cite[Proposition 2.1, item ii]{ACF26}, we infer that $\hat h_{\textup{scal}}(|\xi|)\leq K_\sigma\bigl(|\xi|+1\bigr)$ for all $\xi\in\mathbb R^{N\times n}$, slightly enlarging the constant $K_{\sigma}$ if necessary. This establishes the third inequality in \eqref{e:bulk-growth-h}. 

Finally, combining these growth conditions with the definition of the recession function $h^{\qc,\infty}_\sigma$ stated in \eqref{eq:def-h-recession}, we deduce \eqref{e:bulk-growth-hrecession}, completing the proof.
\end{proof}

For later use it will be useful the following notation

\begin{equation}\label{e:h norm generic}
 h_{\varsigma}(\xi)
 =
 h_{\varsigma}^{\textup{scal}}(\Phi(\xi)^{\frac12}),
 \qquad
h_{\varsigma}^{\textup{scal}}(t)
:=
\inf_{\tau\in(0,\infty)}
\left\{
\varphi\left(\frac1\tau\right)t^2
+
\frac{\varsigma^2}{4}\tau
\right\}.
\end{equation}
The following proposition provides a sufficient criterion ensuring that the effective bulk density $h_\varsigma$ is already convex, so that no quasiconvexification is needed (i.e., $h_\varsigma^{\qc}=h_\varsigma$), for the limit functional $\Fm_\varsigma$ in \eqref{eq:def-limit-functional}.
\begin{proposition}\label{prop:no-qc-isotropic}
Assume \textnormal{(H1)}--\textnormal{(H4)} with
$\varsigma\in(0,\infty)$, and suppose that $\Phi^{1/2}$ is convex.
Let
\begin{equation*}
r(\tau):=\varphi(1/\tau)
\quad\text{for }\tau>0,
\qquad
r(0):=\varphi(\infty),
\end{equation*}
and assume that $q:=\frac1r$ is concave, then $h_\varsigma$ is convex.
\end{proposition}

\begin{proof}
By definition \eqref{e:h norm generic} we have that $h_\varsigma^{\mathrm{scal}}$ is non-decreasing. Moreover, from the identity 
\begin{equation*}
    2zt-z^2q(\tau)= \frac{t^2}{q(\tau)}- q(\tau)\left(z-\frac{t}{q(\tau)}\right)^2,
\end{equation*}
and the positivity of $q$, we deduce the following
\[
r(\tau)t^2
=
\frac{t^2}{q(\tau)}
=
\sup_{z\in\mathbb R}
\bigl\{2zt-z^2q(\tau)\bigr\}
\]
showing the jointly convexity of $(t,\tau)\mapsto r(\tau)t^2$.
Thus the map $(t,\tau)\longmapsto
r(\tau)t^2+\frac{\varsigma^2}{4}\tau$ is jointly convex as well, and its partial infimum with respect
to $\tau\in (0,\infty)$ is convex. Finally, from the assumption on $\Phi^{\frac12}$ and the monotonicity of $h_\varsigma^{\mathrm{scal}}$ we obtain the convexity of $h_\varsigma$.
\end{proof}

\begin{remark}
The hyperbolic choice (cf. \cite{Wu17}) $\varphi(t)=\frac{t}{1+t}$ perfectly matches the setting of Proposition \ref{prop:no-qc-isotropic}, it corresponds to
\[
q(\tau):=\frac1{r(\tau)}=1+\tau.
\]
 More generally, let $q:[0,\infty)\to[1,\infty)$ be a continuous, strictly increasing and concave function such that 
\[
\lim_{\tau\to\infty}\frac{q(\tau)}{\tau}=:a
\in(0,\infty).
\]
Then, the associated degradation function $\varphi(t)=\frac{1}{q(1/t)}$, satisfies the corresponding properties
\textnormal{(H3)}--\textnormal{(H4)}, with
\[
\varphi(\infty)=(q(0))^{-1},
\qquad
\varphi'(0)=a^{-1},
\]
and therefore Proposition~\ref{prop:no-qc-isotropic} applies.
\end{remark}

Next, we state a truncation result following \cite[Lemma~4.4]{CFI24}, whose proof is entirely analogous to \cite[Lemma 2.3]{ACF26}.

\begin{lemma}\label{l:suppenergy}
Assume (\textup{H1}), (\textup{H2}) with $\varsigma\in (0,\infty)$. In addition, let
$\varphi:[0,\infty)\to[0,\infty)$ be non-decreasing and
$\varphi^{-1}(\{0\})=\{0\}$.
Then there exist two functions
$\zeta_1,\zeta_2:(0,1)\to (0,\infty)$ with
\begin{equation}\label{e:limittruncat}
\lim_{\delta \to 0}\zeta_1(\delta)=1
\qquad\textup{and}\qquad
\lim_{\delta \to 0}\zeta_2(\delta)=0
\end{equation}
satisfying the following property.
For every $\eta\in(0,1\wedge \varsigma)$ there exists
$\delta_\eta\in(0,1)$ such that, for every
$\delta\in(0,\delta_\eta)$,
$(u,v)\in H^1\cap L^{\infty}
(\Omega,\R^N\times[0,1])$, and
$A\in \mathcal{A}(\Omega)$, there is
$s_\delta\in(1-\delta,1-\delta^2)$ such that
\[
u^{\delta}:=u\,\chi_{\{v>s_\delta\}}
\in SBV(A,\R^N)
\]
(the previous quantities actually depend also on $\eta$ and $A$) and
\begin{equation}\label{e:stimaprincipal}
H_{\delta}(u^{\delta},A)
\leq
\Fm_\eps(u,v,A)
+
h_{\varsigma_\eta}(0)
\calL^n(A\cap\{v\leq\delta\}),
\end{equation}
where $\varsigma_\eta:=(\varsigma-\eta)(1-\eta)$ and
$H_{\delta}:L^{1}(\Omega,\R^N)\times
\mathcal{A}(\Omega)\to[0,\infty]$
is defined by
\begin{equation}\label{e:Haccagrand sigma finito}
H_{\delta}(w,A)
=
\zeta_1(\delta)
\int_A h^{\qc}_{\varsigma_\eta}(\nabla w)\,\diff x
+
\zeta_2(\delta)
\mathcal{H}^{n-1}(A\cap J_w)
\end{equation}
if $w\in SBV(A,\R^N)$, and $\infty$ otherwise.

Moreover, if $u_{\eps}\to u$ in
$L^1(\Omega,\R^N)$ and
$v_\eps\to1$ in measure on $\Omega$, then
$u_{\eps}^{\delta}\to u$ in
$L^1(A,\R^N)$ as $\eps\to0$ for every
$A\in\mathcal{A}(\Omega)$ and $\delta\in(0,1)$.
\end{lemma}

To reduce the analysis from $GBV$ to $BV \cap L^{\infty}$, we next prove that truncating the displacement field via the smooth operators $\mathcal{T}_k$ in~\eqref{e:Tk}
produces an arbitrarily small increase in the diffuse elastic energy.

\begin{lemma}\label{l:5.3}
Let $\eps>0$ and $S_\eps:H^1(A;\mathbb{R}^N) \times L^1(A;[0,1])\to [0,\infty]$ be the functional defined by
\begin{equation*}
    S_\eps(u,v) := \int_A \varphi(\eps f^2(v))\Phi(\nabla u)\,\mathrm{d}x.
\end{equation*}
where $A\in \mathcal{A}(\R^n)$. Then for every $M \in \mathbb{N}$ and $(u,v) \in H^1(A;\mathbb{R}^N) \times L^1(A;[0,1])$ there exists $k \in \{M + 1, \dots, 2M\}$ such that 
\begin{equation*}
    S_{\eps}(\mathcal{T}_k(u),v) \leq \left(1 + \frac{C_\Phi^2}{M}\right) S_{\eps}(u,v) + \varphi(\infty)(\Phi(0)+C_\Phi) \mathcal{L}^n(|u| > a_M).
\end{equation*}
\end{lemma}

\begin{proof}
Let us fix $M \in \mathbb{N}$ and $(u,v) \in H^1(A;\mathbb{R}^N) \times L^1(A;[0,1])$ with $S_\eps(u,v) < \infty$, otherwise the claim follows trivially. By averaging, there exists $k \in \{M + 1, \dots, 2M\}$ such that
\begin{equation}\label{e:5.11}
    \int_{A \cap \{a_k \leq |u| < a_{k+1}\}} \varphi(\eps f^2(v))\Phi(\nabla u)\,\mathrm{d}x \leq \frac{1}{M} S_\eps(u,v).
\end{equation}
By the properties of $GBV$ functions and the very definition of $\mathcal{T}_k$, we have that $\mathcal{T}_k(u)$ belongs to $BV \cap L^\infty(A;\mathbb{R}^N)$. Furthermore, being $\mathrm{Lip}(\mathcal{T}_k) \leq 1$, we can check for every $y,w \in \mathbb{R}^N$ with $|w| = 1$ that $|(\nabla \mathcal{T}_k(y))w| = |\partial_w \mathcal{T}_k(y)| \leq 1$ that provides $\|\nabla \mathcal{T}_k(y)\|_2 \leq 1$ for every $y \in \mathbb{R}^N$ (here $\|\cdot\|_2$ stands for the matrix norm on $\mathbb{R}^{N \times N}$ induced by $|\cdot|$ on $\mathbb{R}^N$) and consequently
\begin{equation}\label{e:5.12}
    |\nabla(\mathcal{T}_k(u))(x)| \leq |\nabla u(x)| \quad \text{for } \mathcal{L}^n\text{-a.e.\ } x \in A.
\end{equation}
In particular, in virtue of $\mathcal{T}_k(y)= y$ on $\{|y| < a_k\}$, $\mathcal{T}_k(y)= 0$ on $\{|y| \geq a_{k+1}\}$ and \eqref{eq:growth-Phi}, we obtain
\begin{align*}
    S_\eps(\mathcal{T}_k(u),v) &= \int_{A \cap \{|u| < a_k\}} \varphi(\eps f^2(v))\Phi(\nabla u)\,\mathrm{d}x 
    + \int_{A \cap \{a_k \leq |u| < a_{k+1}\}} \varphi(\eps f^2(v))\Phi(\nabla (\mathcal{T}_k(u)))\,\mathrm{d}x \\
    &\quad + \int_{A \cap \{|u| \geq a_{k+1}\}} \varphi(\eps f^2(v))\Phi(0)\,\mathrm{d}x \\
    &\leq S_\eps(u,v) + C_\Phi \int_{A \cap \{a_k \leq |u| < a_{k+1}\}} \varphi(\eps f^2(v))|\nabla(\mathcal{T}_k(u))|^2\,\mathrm{d}x + \varphi(\infty)(\Phi(0)+C_\Phi)\mathcal{L}^n(|u| \geq a_{k}) \\
    &\leq S_\eps(u,v) + C_\Phi\int_{A \cap \{a_k \leq |u| < a_{k+1}\}} \varphi(\eps f^2(v))|\nabla u|^2\,\mathrm{d}x + \varphi(\infty)(\Phi(0)+C_\Phi)\mathcal{L}^n(|u| \geq a_{k+1}) \\
    &\leq \left(1 + \frac{C_\Phi^2}{M}\right) S_\eps(u,v) + \varphi(\infty)(\Phi(0)+C_\Phi) \mathcal{L}^n(|u| > a_M),
\end{align*}
concluding the proof.
\end{proof}

\section{Properties of the surface energy density}\label{sec:surface-density}

In this section, inspired by \cite[Section 3]{CFI24}, we investigate the cohesive surface density $g$ introduced in \eqref{eq:def-g-gamma}. The first target is to reformulate the defining cell formula into alternative expressions that are more manageable to the derivation of the upper bound. Subsequently, we analyze the fundamental properties of $g$, including its growth bounds, subadditivity, continuity, its asymptotic behavior in the small-jump regime, as well as its scalar reduction.

We will frequently rely on the following comparison property.

\begin{lemma}\label{l:quasi monotonia funzione di degradamento}
Assume \textnormal{(H1)}, \textnormal{(H3)}, and \textnormal{(H4)}. Then, there exists a constant $K\ge 1$ such that, for every $\eps>0$ and every pair $t_1,t_2\in [0,1)$ satisfying $t_1\le t_2$, we have
\begin{equation}\label{eq:comparison-f-phi-surface}
 f^2(t_1)\le K f^2(t_2),
 \qquad
 \varphi\big(\eps f^2(t_1)\big)\le K\varphi\big(\eps f^2(t_2)\big).
\end{equation}
\end{lemma}

\begin{proof}
By assumption \textnormal{(H1)}, there exist $0<r_1<r_2<1$ such that $f$ is nondecreasing on both $[0,r_1]$ and $[r_2,1)$. Since $f$ is continuous and strictly positive on $[r_1,r_2]$, one can find a constant $c_0\ge 1$ guaranteeing $f^2(t_1)\le c_0 f^2(t_2)$ whenever $t_1\le t_2$.

We then introduce $\widehat\varphi(t):=\varphi(c_0t)/\varphi(t)$ for $t>0$. Owing to \textnormal{(H3)} and \textnormal{(H4)}, $\widehat\varphi$ is continuous on $(0,\infty)$ and exhibits the limits
\[
 \lim_{t\downarrow 0}\widehat\varphi(t)=c_0,
 \qquad
 \lim_{t\to\infty}\widehat\varphi(t)=1.
\]
Consequently, $c_1:=\sup_{t>0}\widehat\varphi(t)<\infty$, which implies $\varphi(c_0 t)\le c_1\varphi(t)$ for all $t\ge 0$. The second inequality in \eqref{eq:comparison-f-phi-surface} is thus an immediate consequence. 
\end{proof}

\subsection{Alternative formulations of $g$}\label{subsec:g-equivalent}

We begin by establishing a straightforward uniform upper bound for the surface density.

\begin{lemma}\label{l:g limitata}
Under the assumptions of Theorem~\ref{thm:general-anisotropic-gamma-convergence}, there exists a constant $L>0$ such that
\begin{equation}\label{eq:g-uniform-bound}
 g(\zeta,\nu)\le L
 \qquad\text{for all }(\zeta,\nu)\in \mathbb R^N\times \mathbb S^{n-1}.
\end{equation}
\end{lemma}

\begin{proof}
Fix $(\zeta,\nu)\in \mathbb R^N\times \mathbb S^{n-1}$ and let $\eps_j\downarrow 0$. We construct the test functions
\[
 u_j(x):=\widehat u_j(x\cdot\nu),
 \qquad
 v_j(x):=\widehat v_j(x\cdot\nu),
\]
where the one-dimensional profiles are defined as
\[
 \widehat u_j(t):=
 \begin{cases}
 0,& t\le -\eps_j,\\
 \zeta,& t\ge \eps_j,\\
 \textup{affine interpolation},& -\eps_j<t<\eps_j,
 \end{cases}
\qquad
 \widehat v_j(t):=
 \begin{cases}
 0,& |t|\le \eps_j,\\
 1,& |t|\ge 2\eps_j,\\
 \textup{affine interpolation},& \eps_j<|t|<2\eps_j.
 \end{cases}
\]
It is evident that $u_j\to u_{\zeta,\nu}$ in $L^1(Q^\nu;\mathbb R^N)$ and $0\le v_j\le 1$. Since the support of $\nabla u_j$ is strictly contained within the strip $\{|x\cdot\nu|<\eps_j\}$, where $v_j=0$, the bulk term in $\mathcal F^{\infty}_{\eps_j}(u_j,v_j;Q^\nu)$ vanishes identically. Hence, taking into account that $\omega$ is non-decreasing and the equivalence of $\gamma$ with the Euclidean norm (cf. \eqref{e:gamma lower euclidean suppenergy}), we obtain
\begin{align*}
 \mathcal F^{\infty}_{\varepsilon_j}(u_j,v_j;Q^\nu)
 &=\int_{Q^\nu\cap\{|x\cdot\nu|\le 2\varepsilon_j\}}\frac{\omega(1-v_j)}{4\varepsilon_j}\,dx
 +\int_{Q^\nu\cap\{\varepsilon_j\le |x\cdot\nu|\le 2\varepsilon_j\}}\varepsilon_j\gamma(\nabla v_j)^2\,dx\\
 &\le \omega(1)+C_\gamma^2\int_{Q^\nu\cap\{\varepsilon_j\le |x\cdot\nu|\le 2\varepsilon_j\}}\varepsilon_j|\nabla v_j|^2\,dx\leq \omega(1)+2C_\gamma^2.
\end{align*}
Passing to the lower limit in the definition of $g$ immediately yields the claim.
\end{proof}

The subsequent two propositions are dedicated to validating the $L^2$-periodic reduction of the class of competitors involved in the definition of $g$ (cf. \eqref{eq:def-g-gamma}).

To this end, we fix a standard mollifier $\phi_1\in C^\infty_c(B_1)$ such that $\int_{B_1}\phi_1\dx=1$, and denote the scaled kernel by $\phi_\eps(x):=\eps^{-n}\phi_1(x/\eps)\in C^{\infty}_c(B_\eps)$. 

\begin{proposition}\label{proplbboundary-gamma}
Let $(\zeta,\nu)\in\R^N\times\Sf^{n-1}$ and suppose that an
optimizing sequence in \eqref{eq:def-g-gamma} converges in
$L^2(Q^\nu,\R^{N+1})$. Then, under the hypotheses of
Theorem~\ref{thm:general-anisotropic-gamma-convergence}, there exist
a sequence $\eps_j\to0$ and a sequence
\[
 (\tilde u_j,\tilde v_j)\to(u_{\zeta,\nu},1)
 \qquad\text{in }L^2(Q^\nu,\R^{N+1}),
\]
with $\tilde v_j\in[0,1]$ for $\mathcal{L}^n$-a.e.\ in $Q^\nu$,
such that
\begin{equation*}
 \lim_{j\to\infty}
 \callF^{\infty}_{\eps_j}
 (\tilde u_j,\tilde v_j;Q^\nu)
 =g(\zeta,\nu),
\end{equation*}
where
\begin{equation}\label{eqpropboundrybv-gamma}
 \tilde u_j=(u_{\zeta,\nu})\ast\phi_{\eps_j},
 \qquad
 \tilde v_j=
 \mathbf 1_{\{|x\cdot\nu|\ge2\eps_j\}}\ast\phi_{\eps_j}
 \qquad\text{on }\partial Q^\nu.
\end{equation}
\end{proposition}

\begin{proof}
\textbf{Step 1. Construction of $\tilde u_j$ and $\tilde v_j$.}

By hypothesis, there exist a sequence $\eps_j\to0$ and functions
\[
 (u_j,v_j)\to(u_{\zeta,\nu},1)
 \qquad\text{in }L^2(Q^\nu,\R^{N+1}),
\]
with $0\le v_j\le1$ a.e.\ in $Q^\nu$, such that
\begin{equation}\label{e:successione-minimizzante-gamma-boundary}
 g(\zeta,\nu)
 =
 \lim_{j\to\infty}
 \callF^{\infty}_{\eps_j}
 (u_j,v_j;Q^\nu).
\end{equation}
For notational convenience, set
\begin{equation}\label{eqdefUjVj-gamma}
 U_j:=(u_{\zeta,\nu})\ast\phi_{\eps_j},
 \qquad
 V_j:=\mathbf 1_{\{|x\cdot\nu|\ge2\eps_j\}}\ast\phi_{\eps_j}.
\end{equation}
Then $\|u_j-U_j\|_{L^2(Q^\nu,\R^N)}\to0$.
Moreover, $U_j=u_{\zeta,\nu}$ on $\{|x\cdot\nu|\ge\eps_j\}$,
whereas $V_j=0$ on $\{|x\cdot\nu|\le\eps_j\}$ and $V_j=1$
on $\{|x\cdot\nu|\ge3\eps_j\}$.
In particular, $V_j=0$ on $\{\nabla U_j\ne0\}$.
Since $f(0)=0$ and $\Phi_\infty(0)=0$, it follows that
\begin{equation}\label{eq:canonical-energy-gamma}
 \callF^{\infty}_{\eps_j}(U_j,V_j;Q^\nu)
 =
 \callF^{\infty}_{\eps_j}(0,V_j;Q^\nu)\leq c,
\end{equation}
because $\gamma(\nabla V_j)
 \le C_\gamma|\nabla V_j|
 \le\frac{c}{\eps_j}$, by \eqref{e:gamma lower euclidean suppenergy}, and $\nabla V_j$ is supported in a strip of thickness of order $\eps_j$.
We now choose a sequence $\eta_j\downarrow0$ such that
\begin{equation}\label{eq:choice-eta-gamma}
 \frac{\eps_j}{\eta_j}\longrightarrow0,
 \qquad
 \frac{f^2(1-\eta_j)}{\eta_j}
 \|u_j-U_j\|_{L^2(Q^\nu)}^2\longrightarrow0,
\end{equation}
and we set
$K_j:=\left\lfloor\frac{\eta_j}{\eps_j}\right\rfloor$.
For $k=K_j+1,\ldots,2K_j$, define
\[
 R_k^j
 :=
 Q^\nu_{1-k\eps_j}
 \setminus
 Q^\nu_{1-(k+1)\eps_j},
 \qquad
 S_k^j
 :=
 Q^\nu_{1-k\eps_j}
 \setminus
 Q^\nu_{1-(k+3)\eps_j}.
\]
For $j$ sufficiently large, these layers are contained in $Q^\nu$.
The layers $R_k^j$ are pairwise disjoint, while the layers $S_k^j$
have overlap multiplicity at most three a.e.\ in $Q^\nu$.
By \eqref{e:successione-minimizzante-gamma-boundary},
\eqref{eq:canonical-energy-gamma}, and an averaging argument,
there exists $k_j\in\{K_j+1,\ldots,2K_j\}$,
such that, setting $R_j:=R_{k_j}^j$ and $S_j:=S_{k_j}^j$,
one has
\begin{align}
 &\callF^{\infty}_{\eps_j}(u_j,v_j;S_j)
 +\callF^{\infty}_{\eps_j}(U_j,V_j;S_j)
 +\frac{f^2(1-\eta_j)}{\eps_j}
   \|u_j-U_j\|_{L^2(R_j)}^2
 \le
 \frac{c}{K_j}
 +\frac{f^2(1-\eta_j)}{K_j\eps_j}
  \|u_j-U_j\|_{L^2(Q^\nu)}^2,
 \label{eqchoice1-gamma}
\end{align}
where $K_j\to\infty$ and $K_j\eps_j/\eta_j\to1$, so the
right-hand side converges to zero by \eqref{eq:choice-eta-gamma}.
Let $\theta_j\in C_c^1(Q^\nu_{1-k_j\eps_j},[0,1])$ be such that
\[
 \theta_j=1
 \quad\text{on }Q^\nu_{1-(k_j+1)\eps_j},
 \qquad
 |\nabla\theta_j|\le\frac{3}{\eps_j},
\]
and define $\tilde u_j
 :=
 \theta_j u_j+(1-\theta_j)U_j.$ 
 
 We next construct $\tilde v_j$. First, define
\begin{equation}\label{eqdefhatv-gamma}
 \hat v_j(x)
 :=
 \min\left\{
 1,\,
 1-\eta_j+\frac{\dist(x,R_j)}{\eps_j}
 \right\}.
\end{equation}
Thus,
\[
 \hat v_j=1-\eta_j\quad\text{on }R_j,
 \qquad
 \hat v_j=1
 \quad\text{if }\dist(x,R_j)\ge\eta_j\eps_j.
\]
Next, set
\begin{equation}\label{eqdefhatV-gamma}
 \hat V_j(x)
 :=
 \min\left\{
 1,\,
 V_j(x)
 +\frac{
   \dist\bigl(
    x,Q^\nu\setminus Q^\nu_{1-(k_j+1)\eps_j}
   \bigr)
  }{\eps_j}
 \right\}.
\end{equation}
Then $\hat V_j=V_j$ outside
$Q^\nu_{1-(k_j+1)\eps_j}$, while $\hat V_j=1$ sufficiently far
inside this cube, and $\hat V_j=1$ on
$\{|x\cdot\nu|\ge3\eps_j\}$.
Finally, define
\begin{equation*}
 v_j^\ast(x)
 :=
 \min\left\{
 1,\,
 v_j(x)
 +\frac{2}{k_j\eps_j}
  \dist\bigl(x,Q^\nu_{1-k_j\eps_j}\bigr)
 \right\}
\end{equation*}
and set
\begin{equation}\label{eq:def-tilde-v-gamma}
 \tilde v_j:=v_j^\ast\wedge\hat V_j\wedge\hat v_j.
\end{equation}
By construction $0\le\tilde v_j\le1$.
Moreover, on $\partial Q^\nu$, both $v_j^\ast$ and $\hat v_j$
are equal to $1$ for $j$ sufficiently large, consequently
$\tilde v_j=\hat V_j=V_j$ on $\partial Q^\nu$.
Together with the definition of $\tilde u_j$, this proves
the boundary condition \eqref{eqpropboundrybv-gamma}.

\medskip
\noindent
\textbf{Step 2. Estimate of the bulk term.}

On $R_j$ we have
\[
 \nabla\tilde u_j
 =
 \theta_j\nabla u_j
 +(1-\theta_j)\nabla U_j
 +(u_j-U_j)\otimes\nabla\theta_j,
\]
therefore,
\[
 |\nabla\tilde u_j|
 \le
 |\nabla u_j|+|\nabla U_j|
 +\frac{3}{\eps_j}|u_j-U_j|.
\]
By the quadratic growth of $\Phi_\infty$
(cf.\ \eqref{eq:growth-Phi-infty}),
\begin{equation}\label{eq:bulk-growth-tilde-u-gamma}
 \Phi_\infty(\nabla\tilde u_j)
 \le
 c\Phi_\infty(\nabla u_j)
 +c\Phi_\infty(\nabla U_j)
 +\frac{c}{\eps_j^2}|u_j-U_j|^2
 \qquad\text{on }R_j.
\end{equation}
Moreover on $R_j$, $\tilde v_j\le v_j$, $\tilde v_j\le V_j$
and $\tilde v_j\le1-\eta_j$, thus
Lemma~\ref{l:quasi monotonia funzione di degradamento} gives
\[
 f^2(\tilde v_j)\le c f^2(v_j),
 \qquad
 f^2(\tilde v_j)\le c f^2(1-\eta_j).
\]
Furthermore,
\[
 \tilde v_j\le V_j=0
 \qquad\text{on }\{\nabla U_j\ne0\}\cap R_j,
\]
and hence
\[
 f^2(\tilde v_j)\Phi_\infty(\nabla U_j)=0
 \qquad\text{a.e.\ on }R_j.
\]
Multiplying \eqref{eq:bulk-growth-tilde-u-gamma} by
$\eps_jf^2(\tilde v_j)$ yields
\begin{align}
 \eps_j f^2(\tilde v_j)
 \Phi_\infty(\nabla\tilde u_j)
 \le
 c\eps_j f^2(v_j)\Phi_\infty(\nabla u_j)
 +c\frac{f^2(1-\eta_j)}{\eps_j}|u_j-U_j|^2
 \qquad\text{on }R_j.
 \label{eq:bulk-transition-gamma}
\end{align}
Integrating and using \eqref{eqchoice1-gamma}, we obtain
\begin{align*}
 \int_{R_j}
 \varphi'(0)\eps_j f^2(\tilde v_j)
 \Phi_\infty(\nabla\tilde u_j)\,\diff x
 \le
 \frac{c}{K_j}
 +c\frac{f^2(1-\eta_j)}{K_j\eps_j}
 \|u_j-U_j\|_{L^2(Q^\nu)}^2,
\end{align*}
and hence by \eqref{eq:choice-eta-gamma},
\begin{equation}\label{eq:bulk-transition-vanishes-gamma}
 \limsup_{j\to\infty}
 \int_{R_j}
 \varphi'(0)\eps_j f^2(\tilde v_j)
 \Phi_\infty(\nabla\tilde u_j)\,\dx
 =0.
\end{equation}
On $S_j\setminus R_j$, we have $\tilde u_j=u_j$ and
$\tilde v_j\le v_j$. Hence, by
Lemma~\ref{l:quasi monotonia funzione di degradamento} and
\eqref{eqchoice1-gamma},
\[
 \int_{S_j\setminus R_j}
 \varphi'(0)\eps_j f^2(\tilde v_j)
 \Phi_\infty(\nabla\tilde u_j)\,\dx
 \le
 c\callF^{\infty}_{\eps_j}(u_j,v_j;S_j)
 \longrightarrow0.
\]
Moreover, on $Q^\nu_{1-(k_j+3)\eps_j}$ we have
$\hat V_j=\hat v_j=1$ for $j$ sufficiently large, and therefore
\[
 \tilde u_j=u_j,
 \qquad
 \tilde v_j=v_j.
\]
On the exterior region $\tilde u_j=U_j$, $\tilde v_j\le V_j$
and the supports of $\nabla U_j$ and $V_j$ are disjoint, thus
\[
 \int_{Q^\nu\setminus Q^\nu_{1-k_j\eps_j}}
 \varphi'(0)\eps_jf^2(\tilde v_j)
 \Phi_\infty(\nabla\tilde u_j)\,\dx=0.
\]
Combining these estimates, we obtain
\begin{equation}\label{eqfinelastcu-gamma}
 \int_{Q^\nu}
 \varphi'(0)\eps_j f^2(\tilde v_j)
 \Phi_\infty(\nabla\tilde u_j)\,\dx
 \le
 \int_{Q^\nu}
 \varphi'(0)\eps_j f^2(v_j)
 \Phi_\infty(\nabla u_j)\,\dx+o(1).
\end{equation}

\medskip
\noindent
\textbf{Step 3. Estimate of the anisotropic phase-field energy.}

From the monotonicity of $\PotDann$ we obtain
\begin{align}
 \callF^{\infty}_{\eps_j}
 (0,\tilde v_j;Q^\nu)
 \le
 \callF^{\infty}_{\eps_j}
 (0,v_j^\ast;Q^\nu)
 +
 \callF^{\infty}_{\eps_j}
 (0,\hat V_j;Q^\nu)
 +\callF^{\infty}_{\eps_j}
 (0,\hat v_j;Q^\nu).
 \label{eqcalFst0-gamma}
\end{align}
We start with $\hat v_j$. By \eqref{eqdefhatv-gamma}, we have $0\le1-\hat v_j\le\eta_j$ and $|\nabla\hat v_j|\le\frac1{\eps_j}$, with
\[
 \calL^n(\{\hat v_j\ne1\})\le c\eps_j,
 \qquad
 \calL^n(\{\nabla\hat v_j\ne0\})
 \le c\eta_j\eps_j,
 \qquad
 \gamma(\nabla\hat v_j)
 \le\frac{C_\gamma}{\eps_j}.
\]
Setting $\tau_j:=\max_{t\in[0,\eta_j]}\PotDann(t)$, we have
\begin{align}
 \callF^{\infty}_{\eps_j}
 (0,\hat v_j;Q^\nu)
 =\int_{Q^\nu}
 \left(
  \frac{\PotDann(1-\hat v_j)}{4\eps_j}
  +\eps_j\gamma(\nabla\hat v_j)^2
 \right)\,\dx
 \le c\tau_j+c\eta_j.
 \label{eq:hat-v-energy-gamma}
\end{align}
Likewise, the definition of $\hat V_j$ gives
\[
 |\nabla\hat V_j|
 \le|\nabla V_j|+\frac1{\eps_j},
 \qquad
 \gamma(\nabla\hat V_j)\le\frac{c}{\eps_j},
 \qquad
 \calL^n(\{\hat V_j\ne1\})\le c\eta_j\eps_j,
\]
and consequently,
\begin{equation}\label{eq:hat-V-energy-gamma}
 \callF^{\infty}_{\eps_j}
 (0,\hat V_j;Q^\nu)
 \le c\eta_j.
\end{equation}
It remains to estimate $v_j^\ast$.
Set $E_j:=Q^\nu\setminus Q^\nu_{1-k_j\eps_j}$ and observe that $v_j^\ast=v_j$ on $Q^\nu_{1-k_j\eps_j}$ and $v_j^\ast\ge v_j$ on $Q^\nu$. Thus, by the monotonicity of $\PotDann$,
$\PotDann(1-v_j^\ast)\le\PotDann(1-v_j)$, and being the
distance function 1-Lipschitz,
\[
 \gamma(\nabla v_j^\ast)^2
 \le
 \gamma(\nabla v_j)^2
 +\frac{c}{k_j\eps_j}\gamma(\nabla v_j)
 +\frac{c}{k_j^2\eps_j^2}
 \qquad\text{on }E_j.
\]
Therefore,
\begin{align}
 \callF^{\infty}_{\eps_j}
 (0,v_j^\ast;Q^\nu)
 \le
 \callF^{\infty}_{\eps_j}
 (0,v_j;Q^\nu)
 +\frac{c\calL^n(E_j)}{k_j^2\eps_j}
 +
 \frac{c}{k_j}
 \int_{E_j}\gamma(\nabla v_j)\,\dx,
 \label{eq:vstar-preliminary-gamma}
\end{align}
where, by the Cauchy-Schwarz inequality,
\begin{align*}
 \frac1{k_j}\int_{E_j}\gamma(\nabla v_j)\,\dx
 \le
 \frac{\calL^n(E_j)^{\frac12}}{k_j\eps_j^{\frac12}}
 \left(
  \eps_j
  \int_{E_j}\gamma(\nabla v_j)^2\,\dx
 \right)^{\frac12}
 \le
 \frac{\calL^n(E_j)^{\frac12}}{k_j\eps_j^{\frac12}}
 \callF^{\infty}_{\eps_j}
 (0,v_j;Q^\nu)^{\frac12}.
\end{align*}
Since $E_j$ is a boundary layer of thickness of order
$k_j\eps_j$, we have $\calL^n(E_j)\le ck_j\eps_j$.
The energies
$\callF^{\infty}_{\eps_j}(u_j,v_j;Q^\nu)$
are uniformly bounded by
\eqref{e:successione-minimizzante-gamma-boundary}, hence
\eqref{eq:vstar-preliminary-gamma} yields
\begin{equation}\label{eq:vstar-final-gamma}
 \callF^{\infty}_{\eps_j}
 (0,v_j^\ast;Q^\nu)
 \le
 \callF^{\infty}_{\eps_j}
 (0,v_j;Q^\nu)
 +\frac{c}{k_j}
 +\frac{c}{k_j^{1/2}},
\end{equation}
and, since $k_j\ge K_j+1\to\infty$,
\begin{equation}\label{eq:vstar-limsup-gamma}
 \limsup_{j\to\infty}
 \callF^{\infty}_{\eps_j}
 (0,v_j^\ast;Q^\nu)
 \le
 \limsup_{j\to\infty}
 \callF^{\infty}_{\eps_j}
 (0,v_j;Q^\nu).
\end{equation}
Combining \eqref{eqcalFst0-gamma},
\eqref{eq:hat-v-energy-gamma},
\eqref{eq:hat-V-energy-gamma}, and
\eqref{eq:vstar-final-gamma}, and using
$\eta_j,\tau_j\to0$, we conclude that
\begin{equation}\label{eq:phase-limsup-gamma}
 \callF^{\infty}_{\eps_j}
 (0,\tilde v_j;Q^\nu)
 \le
 \callF^{\infty}_{\eps_j}
 (0,v_j;Q^\nu)+o(1).
\end{equation}
Finally, \eqref{eqfinelastcu-gamma} and
\eqref{eq:phase-limsup-gamma} imply
\[
 \limsup_{j\to\infty}
 \callF^{\infty}_{\eps_j}
 (\tilde u_j,\tilde v_j;Q^\nu)
 \le g(\zeta,\nu).
\]
By construction, $\tilde u_j\to u_{\zeta,\nu}$ in
$L^2(Q^\nu,\R^N)$.
Since $\tilde v_j=v_j$ on
$Q^\nu_{1-(k_j+3)\eps_j}$, $0\le\tilde v_j\le1$, and
$(k_j+3)\eps_j\to0$, we also have $\tilde v_j\to1$ in
$L^2(Q^\nu)$.
The reverse inequality follows from
\eqref{eq:def-g-gamma}, completing the proof.
\end{proof}
We now present the periodic reduction on the class of test sequences defining $g(\cdot,\nu)$ in \eqref{eq:def-g-gamma}. Using a De Giorgi’s averaging argument, the family $\mathcal{T}_k$, defined in \eqref{e:Tk}, replaces sequences merely converging in $L^1$ with sequences sharing the identical $L^1$ limit but enjoying uniform bounds in $L^\infty$. This modification entails an arbitrarily small increase in energy.

\begin{proposition}\label{periodicity}
Let $(\zeta,\nu)\in\R^N\times\Sf^{n-1}$ and assume that the
hypotheses of Theorem~\ref{thm:general-anisotropic-gamma-convergence} hold. Then, for every sequence
$\tilde\eps_j\downarrow0$, there exists a sequence
$(\tilde u_j,\tilde v_j)$ converging to $(u_{\zeta,\nu},1)$ in
$L^2(Q^\nu;\R^{N+1})$. Furthermore,
$\tilde v_j(x)\in[0,1]$ for $\calL^n$-a.e.\ $x\in Q^\nu$, and
\begin{equation}\label{e:g periodic sequence}
 \lim_{j\to\infty}
 \callF^{\infty}_{\tilde\eps_j}
 (\tilde u_j,\tilde v_j;Q^\nu)
 =
 g(\zeta,\nu),
\end{equation}
with the boundary conditions
\begin{equation*}
 \tilde u_j=(u_{\zeta,\nu})\ast\phi_{\tilde\eps_j},
 \qquad
 \tilde v_j=
 \mathbf 1_{\{|x\cdot\nu|\ge2\tilde\eps_j\}}
 \ast\phi_{\tilde\eps_j}
 \qquad\text{on }\partial Q^\nu.
\end{equation*}
\end{proposition}

\begin{proof}
\textbf{Step 1. Reduction to an optimal sequence converging in
$L^2$.} Consider a vanishing sequence $\eps_j\to0$ and a pair $(u_j,v_j)$
such that $(u_j,v_j)\to (u_{\zeta,\nu},1)
 \text{ in }L^1(Q^\nu;\R^{N+1})$ and
\begin{equation}\label{eq:optimal-L1-periodicity-gamma}
 g(\zeta,\nu)
 = \lim_{j\to\infty}
 \callF^{\infty}_{\eps_j}(u_j,v_j;Q^\nu).
\end{equation}
Since $0\le v_j\le1$, we have that $v_j\to 1$ strongly in $L^2(Q^{\nu})$. We assert that, for every $j,M\in\N$, there exists $ k_{M,j}\in\{M+1,\ldots,2M\}$ such that
\begin{equation}\label{e:truncation g}
 \callF^{\infty}_{\eps_j}
 \bigl(\mathcal T_{k_{M,j}}(u_j),v_j;Q^\nu\bigr)
 \le
 \left(1+\frac{C^2}{M}\right)
 \callF^{\infty}_{\eps_j}
 (u_j,v_j;Q^\nu),
\end{equation}
where the functions $\mathcal{T}_k$ are defined in \eqref{e:Tk}, and we assume that
\[
 a_M>1+|\zeta|
 =1+\|u_{\zeta,\nu}\|_{L^\infty(Q^\nu)}.
\]
For fixed
$M$, the functions
$\mathcal T_{k_{M,j}}(u_j)$ are uniformly bounded and converge
to $u_{\zeta,\nu}$ in $L^2(Q^\nu,\R^N)$. Applying \eqref{e:truncation g}, we obtain
\[
 \limsup_{j\to\infty}
 \callF^{\infty}_{\eps_j}
 \bigl(\mathcal T_{k_{M,j}}(u_j),v_j;Q^\nu\bigr)
 \le
 \left(1+\frac{C^2}{M}\right)g(\zeta,\nu).
\]
Letting $M\to\infty$ and using a diagonal argument, we deduce
\begin{align}
 g(\zeta,\nu)
 =\inf\Bigl\{
 \liminf_{j\to\infty}
 \callF^{\infty}_{\eps_j}(u_j,v_j;Q^\nu):
 \eps_j\to0,\,
 \|u_j-u_{\zeta,\nu}\|_{L^2(Q^\nu)}\to0,\quad
 \|v_j-1\|_{L^2(Q^\nu)}\to0
 \Bigr\}.
 \label{eq:g-characterization-L2-gamma}
\end{align}
We now justify \eqref{e:truncation g}. For every
$k\in\{M+1,\ldots,2M\}$, since $\Phi_{\infty}(0)=0$,
\begin{align}
 &\callF^{\infty}_{\eps_j}
 \bigl(\mathcal T_k(u_j),v_j;Q^\nu\bigr)
 =
 \callF^{\infty}_{\eps_j}
 \bigl(u_j,v_j;\{|u_j|\le a_k\}\bigr)
 \notag\\
 &+
 \callF^{\infty}_{\eps_j}
 \bigl(\mathcal T_k(u_j),v_j;
       \{a_k<|u_j|<a_{k+1}\}\bigr)+
 \callF^{\infty}_{\eps_j}
 \bigl(0,v_j;\{|u_j|\ge a_{k+1}\}\bigr).
 \label{e:Feps-Tk-gamma}
\end{align}
On the intermediate region, the quadratic growth of
$\Phi_\infty$ (cf. \eqref{eq:growth-Phi-infty}), together with
$\operatorname{Lip}(\mathcal T_k)\le1$, give
\begin{align}
 \callF^{\infty}_{\eps_j}
 \bigl(\mathcal T_k(u_j),v_j;
       \{a_k<|u_j|<a_{k+1}\}\bigr)
&\le
 C^2
 \int_{\{a_k<|u_j|<a_{k+1}\}}
 \varphi'(0)\eps_j f^2(v_j)
 \Phi_\infty(\nabla u_j)\,\dx
 \notag\\
 &+
 \callF^{\infty}_{\eps_j}
 \bigl(0,v_j;\{a_k<|u_j|<a_{k+1}\}\bigr).
 \label{e:Feps-Tk-2-gamma}
\end{align}
Combining \eqref{e:Feps-Tk-gamma} and
\eqref{e:Feps-Tk-2-gamma}, we find
\begin{align*}
 \callF^{\infty}_{\eps_j}
 \bigl(\mathcal T_k(u_j),v_j;Q^\nu\bigr)
\le\callF^{\infty}_{\eps_j}
 (u_j,v_j;Q^\nu)
 +C^2\int_{\{a_k<|u_j|<a_{k+1}\}}
 \varphi'(0)\eps_j f^2(v_j)
 \Phi_\infty(\nabla u_j)\,\dx.
\end{align*}
Therefore choosing $k_{j,M}\in \{M+1,\dots,2M\}$ such that 
\begin{equation*}
    \int_{\{a_k<|u_j|<a_{k+1}\}}
 \varphi'(0)\eps_j f^2(v_j)
 \Phi_\infty(\nabla u_j)\,\dx\leq \frac1M\callF^{\infty}_{\eps_j}
 (u_j,v_j;Q^\nu), 
\end{equation*}
we obtain \eqref{e:truncation g}.

\textbf{Step 2. Conclusion.} By Step 1, we can assume that a minimizing sequence for
$g(\zeta,\nu)$ converges strongly in
$L^2(Q^\nu;\R^{N+1})$. Let $(\eps_k,u_k,v_k)$ be the sequence provided by
Proposition~\ref{proplbboundary-gamma}.
Since, for every fixed $k$, $\frac{\tilde\eps_j}{\eps_k}\rightarrow0$ as $j\to \infty$, we can choose a monotonically increasing sequence $k(j)\to\infty$
such that
\begin{equation}\label{eq:def-lambda-periodicity-gamma}
 \lambda_j:=
 \frac{\tilde\eps_j}{\eps_{k(j)}}\longrightarrow0.
\end{equation}
For $r>0$, set $ \widetilde Q^\nu_r:=\Pi_\nu\cap Q^\nu_r,
$ where $\Pi_\nu$ is the hyperplane orthogonal to $\nu$. Choose points
\[
 x_1,\ldots,x_{I_j}\in\widetilde Q^\nu,
 \qquad
 I_j:=\left\lfloor\frac1{\lambda_j}\right\rfloor^{n-1},
\]
such that the tangential cubes
\[
 x_i+\widetilde Q^\nu_{\lambda_j},
 \qquad i=1,\ldots,I_j,
\]
are pairwise disjoint and contained in $\widetilde Q^\nu$. They may
be chosen so that $I_j\lambda_j^{n-1}\rightarrow1$ and
\begin{equation}\label{eq:tangential-remainder-measure}
 \calH^{n-1}\left(
 \widetilde Q^\nu
 \setminus
 \bigcup_{i=1}^{I_j}
 (x_i+\widetilde Q^\nu_{\lambda_j})
 \right)\longrightarrow0,
\end{equation}
as $j\to \infty$. Set
\[
 \widetilde U_j:=
 u_{\zeta,\nu}\ast\phi_{\tilde\eps_j},
 \qquad
 \widetilde V_j:=
 \chi_{\{|x\cdot\nu|\ge2\tilde\eps_j\}}
 \ast\phi_{\tilde\eps_j}.
\]
We define
\[
 \tilde u_j(x):=
 \begin{cases}
  u_{k(j)}
  \left(\dfrac{x-x_i}{\lambda_j}\right),
  &
  \text{if }x-x_i\in Q^\nu_{\lambda_j}
  \text{ for some }i,
  \\[6pt]
  \widetilde U_j(x),
  &
  \text{otherwise in }Q^\nu,
 \end{cases}
\]
and
\[
 \tilde v_j(x):=
 \begin{cases}
  v_{k(j)}
  \left(\dfrac{x-x_i}{\lambda_j}\right),
  &
  \text{if }x-x_i\in Q^\nu_{\lambda_j}
  \text{ for some }i,
  \\[6pt]
  \widetilde V_j(x),
  &
  \text{otherwise in }Q^\nu.
 \end{cases}
\]
Since $x_i\cdot\nu=0$ and
$\tilde\eps_j=\lambda_j\eps_{k(j)}$, the scaling properties of the
mollifier give
\begin{equation}\label{eq:canonical-scaling-periodicity}
 \widetilde U_j(x)
 =
 U_{k(j)}
 \left(\frac{x-x_i}{\lambda_j}\right),
 \qquad
 \widetilde V_j(x)
 =
 V_{k(j)}
 \left(\frac{x-x_i}{\lambda_j}\right).
\end{equation}
Consequently, the traces of the two definitions coincide on the
boundaries of the cubes $x_i+Q^\nu_{\lambda_j}$. By the boundary
conditions furnished by
Proposition~\ref{proplbboundary-gamma}, we therefore have $ (\tilde u_j,\tilde v_j)
 \in W^{1,2}(Q^\nu;\R^{N+1})$, with
\[
 \tilde u_j=\widetilde U_j,
 \qquad
 \tilde v_j=\widetilde V_j
 \qquad\text{on }\partial Q^\nu.
\]
We now estimate the energy. On each cube
$x_i+Q^\nu_{\lambda_j}$, using the change of variables $x=x_i+\lambda_j y$, we have
\[ \nabla\tilde u_j(x)
 =\frac1{\lambda_j}\nabla u_{k(j)}(y),
 \qquad
 \nabla\tilde v_j(x)
 =\frac1{\lambda_j}\nabla v_{k(j)}(y).
\]
Since $\Phi_\infty$ and $\gamma^2$ are positively two-homogeneous, using  $\tilde\eps_j=\lambda_j\eps_{k(j)}$, we obtain the 
scaling identity
\begin{align}
 \callF^{\infty}_{\tilde\eps_j}
 \bigl(\tilde u_j,\tilde v_j;
       x_i+Q^\nu_{\lambda_j}\bigr)
 =
 \lambda_j^{n-1}
 \callF^{\infty}_{\eps_{k(j)}}
 (u_{k(j)},v_{k(j)};Q^\nu).
 \label{eq:cell-scaling-gamma}
\end{align}
On the complementary region the pair coincides with
$(\widetilde U_j,\widetilde V_j)$. By construction,
\[
 \widetilde U_j=u_{\zeta,\nu}
 \quad\text{on }\{|x\cdot\nu|\ge\tilde\eps_j\},
\]
whereas
\[
 \widetilde V_j=0
 \quad\text{on }\{|x\cdot\nu|\le\tilde\eps_j\},
 \qquad
 \widetilde V_j=1
 \quad\text{on }\{|x\cdot\nu|\ge3\tilde\eps_j\},
\]
therefore
\[
 f^2(\widetilde V_j)
 \Phi_\infty(\nabla\widetilde U_j)=0
 \qquad\text{a.e. on  $Q^\nu$.}
\]
It follows that
\begin{align}
 \callF^{\infty}_{\tilde\eps_j}
 (\tilde u_j,\tilde v_j;Q^\nu)
 \le
 I_j\lambda_j^{n-1}
 \callF^{\infty}_{\eps_{k(j)}}
 (u_{k(j)},v_{k(j)};Q^\nu)
 +
 c\calH^{n-1}\left(
 \widetilde Q^\nu
 \setminus
 \bigcup_{i=1}^{I_j}
 (x_i+\widetilde Q^\nu_{\lambda_j})
 \right),
 \label{eq:periodicity-upper-gamma}
\end{align}
where $c>0$ depends only on $\PotDann$, the mollifier kernel $\phi$, and the
equivalence constants of $\gamma$ (cf. \eqref{e:gamma lower euclidean suppenergy}). By optimality of the sequence $(\eps_k,u_k,v_k)$, \eqref{eq:tangential-remainder-measure} and $I_j \lambda_j^{n-1}\to 1$ as $j\to \infty$, we obtain that
\[
 \limsup_{j\to\infty}
 \callF^{\infty}_{\tilde\eps_j}
 (\tilde u_j,\tilde v_j;Q^\nu)
 \le g(\zeta,\nu),
\]
concluding the proof, by definition of $g(\cdot,\cdot)$ in \eqref{eq:def-g-gamma}.
\end{proof}

We proceed to formulate an equivalent large-cell expression for the surface energy $g$, analogous in structure to the one-dimensional formulations (cf. \cite[Proposition 2.5 item (v)]{ACF26}).

\begin{proposition}\label{p:chargT}
Under the assumptions of Theorem~\ref{thm:general-anisotropic-gamma-convergence},
for every $(\zeta,\nu)\in\R^N\times\Sf^{n-1}$ one has
\begin{equation}\label{e:characg}
 g(\zeta,\nu)
 =
 \lim_{T\to\infty}
 \inf_{(u,v)\in\calU_{\zeta,\nu}^T}
 \frac{1}{T^{n-1}}
 \callF^{\infty}_{1}(u,v;Q^\nu_T),
\end{equation}
with
\begin{multline}\label{e:spazio funzionale per la g vettoriale alternativa}
 \calU^T_{\zeta,\nu}
 :=
 \Big\{
 (u,v)\in
 H^1(Q^\nu_T;\R^N)\times H^1(Q^\nu_T;[0,1])\colon
 \\
 v=\mathbf 1_{\{|x\cdot\nu|\ge2\}}\ast\phi_1
 \ \text{ and }\
 u=u_{\zeta,\nu}\ast\phi_1
 \quad\text{on }\partial Q^\nu_T
 \Big\},
\end{multline}
where the boundary equalities are understood in the sense of traces.
\end{proposition}

\begin{proof}
For every $(\zeta,\nu)\in\R^N\times\Sf^{n-1}$ and $T>0$, define
\begin{equation}\label{e:def di gT}
 g_T(\zeta,\nu)
 :=
 \inf_{(u,v)\in\calU_{\zeta,\nu}^T}
 \frac{1}{T^{n-1}}
 \callF^{\infty}_{1}(u,v;Q^\nu_T).
\end{equation}
We first prove that
\begin{equation}\label{e:limsup leq g}
 \limsup_{T\to\infty}g_T(\zeta,\nu)
 \le g(\zeta,\nu)
 \qquad
 \text{for every }(\zeta,\nu)\in\R^N\times\Sf^{n-1}.
\end{equation}
Fix $(\zeta,\nu)\in\R^N\times\Sf^{n-1}$ and let
$T_j\uparrow\infty$ be a sequence realizing the upper limit on the
left-hand side. By Proposition~\ref{periodicity}, applied to the
prescribed sequence $1/T_j\downarrow0$, there exist $(u_j,v_j)\in
 H^1(Q^\nu;\R^N)\times H^1(Q^\nu;[0,1])
$ such that
\[
 (u_j,v_j)\to(u_{\zeta,\nu},1)
 \qquad\text{in }L^2(Q^\nu;\R^{N+1}),
\]
and
\begin{equation}\label{e:bound}
 u_j=u_{\zeta,\nu}\ast\phi_{\frac1{T_j}},
 \qquad
 v_j=
 \chi_{\{|x\cdot\nu|\ge\frac2{T_j}\}}
 \ast\phi_{\frac1{T_j}}
 \qquad\text{on }\partial Q^\nu,
\end{equation}
with
\begin{equation}\label{e:estf}
 \lim_{j\to\infty}
 \callF^{\infty}_{\frac1{T_j}}
 (u_j,v_j;Q^\nu)
 =
 g(\zeta,\nu).
\end{equation}
Define
\[
 \tilde u_j(y):=u_j\left(\frac{y}{T_j}\right),
 \qquad
 \tilde v_j(y):=v_j\left(\frac{y}{T_j}\right),
 \qquad y\in Q^\nu_{T_j}.
\]
Since $\Phi_\infty$ and $\gamma^2$ are
positively two-homogeneous, a change of variables gives
\begin{equation}\label{eq:large-cell-scaling-upper-gamma}
 \frac{1}{T_j^{n-1}}
 \callF^{\infty}_{1}
 (\tilde u_j,\tilde v_j;Q^\nu_{T_j})
 =
 \callF^{\infty}_{\frac1{T_j}}
 (u_j,v_j;Q^\nu).
\end{equation}
The scaling properties of $u_{\zeta,\nu}$ and of the mollifier imply,
by \eqref{e:bound}, that $(\tilde u_j,\tilde v_j)\in\calU^{T_j}_{\zeta,\nu}$, obtaining \eqref{e:limsup leq g}.

We now prove the converse inequality
\begin{equation}\label{e:liminf geq g}
 \liminf_{T\to\infty}g_T(\zeta,\nu)
 \ge g(\zeta,\nu).
\end{equation}
Fix $\eta>0$ and choose $T$, $(u_T,v_T)\in \calU^T_{\zeta,\nu}$ such that
\begin{equation}\label{e:estf2}
 \frac{1}{T^{n-1}}
 \callF^{\infty}_{1}
 (u_T,v_T;Q^\nu_T)
 \le
 \liminf_{S\to\infty}g_S(\zeta,\nu)+\eta.
\end{equation}
Let $\eps_j\downarrow0$ be such that
\[
 m_j:=\frac{1}{\eps_jT}\in\N
 \qquad\text{for every }j.
\]
Let $I_{\eps_j}^\nu\subset\Pi_\nu$ be a finite set such that the
oriented cubes
\[
 Q^\nu_{\eps_jT}(\eps_jd)
 :=
 \eps_jd+Q^\nu_{\eps_jT},
 \qquad d\in I_{\eps_j}^\nu,
\]
are pairwise disjoint, are contained in $Q^\nu$, and satisfy
\begin{equation}\label{eq:tangential-tiling-large-cell}
 \calH^{n-1}\left(
 (Q^\nu\cap\Pi_\nu)
 \setminus
 \bigcup_{d\in I_{\eps_j}^\nu}
 \bigl(Q^\nu_{\eps_jT}(\eps_jd)\cap\Pi_\nu\bigr)
 \right)=0.
\end{equation}
Define the functions
\[
 \widetilde U_j(y)
 :=u_{\zeta,\nu}\ast\phi_{\eps_j}(y), \qquad \widetilde V_j(y)
 := \mathbf 1_{\{|x\cdot\nu|\ge2\eps_j\}}
 \ast\phi_{\eps_j}(y),
\]
and set
\[
 (u_j(y),v_j(y))
 :=
 \begin{cases}
 \displaystyle
 \left(
 u_T\left(\frac{y}{\eps_j}-d\right),
 v_T\left(\frac{y}{\eps_j}-d\right)
 \right),
 &
 \begin{aligned}
 &\text{if }y\in
 Q^\nu_{\eps_jT}(\eps_jd)\\[-2pt]
 &\text{for some }d\in I_{\eps_j}^\nu,
 \end{aligned}
 \\[10pt]
 \bigl(\widetilde U_j(y),\widetilde V_j(y)\bigr),
 &
 \text{otherwise in }Q^\nu.
 \end{cases}
\]

Since $d\in\Pi_\nu$, translations by $d$ leave the canonical boundary
data invariant, consequently we have $(u_j,v_j)\in
 H^1(Q^\nu;\R^N)\times H^1(Q^\nu;[0,1]).
$ We next verify convergence: since $d\cdot\nu=0$ and $\eps_j>0$,
\[
 u_{\zeta,\nu}\left(\frac{y}{\eps_j}-d\right)
 =
 u_{\zeta,\nu}(y).
\]
A change of variables in each cube gives
\begin{align*}
 \int_{\bigcup_{d\in I_{\eps_j}^\nu}
 Q^\nu_{\eps_jT}(\eps_jd)}
 |u_j-u_{\zeta,\nu}|\,\dx=
 \eps_j^n(\#I_{\eps_j}^\nu)
 \int_{Q^\nu_T}|u_T-u_{\zeta,\nu}|\,\dx=
 \eps_jT^{1-n}
 \int_{Q^\nu_T}|u_T-u_{\zeta,\nu}|\,\dx
 \longrightarrow0,
\end{align*}
and similarly,
\[
 \int_{\bigcup_{d\in I_{\eps_j}^\nu}
 Q^\nu_{\eps_jT}(\eps_jd)}
 |v_j-1|\,\dx
 =
 \eps_jT^{1-n}
 \int_{Q^\nu_T}|v_T-1|\,\dx
 \longrightarrow0.
\]
Outside the union of the cubes, the pair is equal to
$(\widetilde U_j,\widetilde V_j)$, which converges to
$(u_{\zeta,\nu},1)$ in $L^1(Q^\nu,\R^{N+1})$. 

It remains to compute the energy. By the two-homogeneity of $\Phi_\infty$ and $\gamma^2$, the change of variable $x=\frac{y}{\eps_j}-d$ yields
\begin{equation}\label{eq:small-cell-scaling-gamma}
 \callF^{\infty}_{\eps_j}
 \bigl(u_j,v_j;
 Q^\nu_{\eps_jT}(\eps_jd)\bigr)
 =
 \eps_j^{n-1}
 \callF^{\infty}_{1}
 (u_T,v_T;Q^\nu_T).
\end{equation}

Moreover,
\[
 \widetilde U_j=u_{\zeta,\nu}
 \quad\text{on }\{|x\cdot\nu|\ge\eps_j\},
\]
whereas
\[
 \widetilde V_j=0
 \quad\text{on }\{|x\cdot\nu|\le\eps_j\},
 \qquad
 \widetilde V_j=1
 \quad\text{on }\{|x\cdot\nu|\ge3\eps_j\}.
\]
Taking $T>0$ large enough, we may assume the support of the energy of
$(\widetilde U_j,\widetilde V_j)$ is contained in the central slab
covered by the cubes
$Q^\nu_{\eps_jT}(\eps_jd)$. Therefore, using
$\Phi_\infty(0)=0$ and $f(0)=0$, the energy vanishes outside their
union, and from \eqref{eq:small-cell-scaling-gamma} we obtain
\begin{align*}
 \callF^{\infty}_{\eps_j}
 (u_j,v_j;Q^\nu)=
 \sum_{d\in I_{\eps_j}^\nu}
 \callF^{\infty}_{\eps_j}
 \bigl(u_j,v_j;
 Q^\nu_{\eps_jT}(\eps_jd)\bigr)=
 \eps_j^{n-1}\#I_{\eps_j}^\nu
 \callF^{\infty}_{1}
 (u_T,v_T;Q^\nu_T)=
 \frac{1}{T^{n-1}}
 \callF^{\infty}_{1}
 (u_T,v_T;Q^\nu_T).
\end{align*}
In particular, by definition of
$g(\cdot,\cdot)$, we conclude that
\begin{align*}
 g(\zeta,\nu)
\le
 \liminf_{j\to\infty}
 \callF^{\infty}_{\eps_j}
 (u_j,v_j;Q^\nu)=
 \frac{1}{T^{n-1}}
 \callF^{\infty}_{1}
 (u_T,v_T;Q^\nu_T)\le
 \liminf_{S\to\infty}g_S(\zeta,\nu)+\eta.
\end{align*}
Letting $\eta\downarrow0$ proves \eqref{e:liminf geq g}, concluding the proof.
\end{proof}

Building on this representation of $g$, we derive a modified version of Proposition~\ref{periodicity}, which includes an additional regularization term $\eta_\eps \int \Phi(\nabla u)\,dx$. This term, which ensures the coercivity of the functionals $\Fm_\eps$, is crucial for the analysis in Section~\ref{s:upperbound}.

\begin{proposition}\label{propuvjump}
For any sequences $\eps_j\downarrow0$ and $\eta_j\downarrow0$
satisfying $\eta_j/\eps_j\to0$, and for every
$(\zeta,\nu)\in \R^N\times \Sf^{n-1}$, under the assumptions of
Theorem~\ref{thm:general-anisotropic-gamma-convergence}, there exist functions
\begin{equation*}
(u_j,v_j)\to (u_{\zeta,\nu},1)
\qquad\text{in }L^2(Q^\nu;\R^{N+1}),
\end{equation*}
such that
\begin{equation*}
\lim_{j\to\infty}
\callF^{\infty}_{\eps_j}(u_j,v_j;Q^\nu)=
g(\zeta,\nu) \qquad \textup{and} \qquad \lim_{j\to\infty}
\eta_j\int_{Q^\nu}|\nabla u_j|^2\dx=0,
\end{equation*}
with
\begin{equation*}
u_j=u_{\zeta,\nu}\ast\phi_{\eps_j},
\qquad
v_j=
\mathbf 1_{{|x\cdot\nu|\ge2\eps_j}}\ast\phi_{\eps_j}
\qquad
\text{on }\partial Q^\nu.
\end{equation*}
\end{proposition}

\begin{proof}
We use the construction developed in the second part of Proposition~\ref{p:chargT}. For every sufficiently large $T>0$, and $j\in \N$ we construct a sequence $(u_{j,T},v_{j,T})\in H^1(Q^\nu,\R^N\times [0,1])$, with
\begin{equation}\label{e:boundary coercive prop} u_{j,T}=u_{\zeta,\nu}\ast\phi_{\eps_j}, \qquad v_{j,T} = \mathbf 1_{\{|x\cdot\nu|\ge2\eps_j\}}\ast\phi_{\eps_j} \qquad \text{on }\partial Q^\nu,
\end{equation} 
such that, for every fixed $T$, 
\begin{equation}\label{e:L2 fixed T coercive} (u_{j,T},v_{j,T}) \longrightarrow (u_{\zeta,\nu},1) \qquad \text{in }L^2(Q^\nu;\R^{N+1}) \end{equation} 
as $j\to\infty$, and 
\begin{equation}\label{e:energy fixed T coercive} \limsup_{j\to\infty} \callF^\infty_{\eps_j} (u_{j,T},v_{j,T};Q^\nu) \leq g_T(\zeta,\nu), 
\end{equation} 
where $g_T$ is defined in \eqref{e:def di gT}. 

It remains to show that the parameter $T$ can be chosen depending on $j$ in such a way that, in addition to \eqref{e:L2 fixed T coercive}--\eqref{e:energy fixed T coercive}, the additional coercive term tends to zero. By the very construction of $u_j$, we have
\begin{align*}
\|\nabla u_{j,T_j}\|_{L^2(Q^\nu)}^2
&\leq
\sum_{d\in I_{\eps_j}}
\|\nabla u_{j,T_j}\|_{L^2(Q^\nu_{\eps_jT_j}(\eps_jd))}^2
+
\frac{c}{\eps_j^2}
\calL^n
\bigl(
Q^\nu\cap\{|x\cdot\nu|\leq\eps_j\}
\bigr)
\\
&=
\eps_j^{n-1}(\# I_{\eps_j})
\|\nabla u_{T_j}\|_{L^2(Q_{T_j}^\nu)}^2
+
\frac{c}{\eps_j}=
\frac{\|\nabla u_{T_j}\|_{L^2(Q_{T_j}^\nu)}^2}{T_j^{n-1}}
+
\frac{c}{\eps_j}
\leq
\frac{c_{T_j}}{\eps_j}.
\end{align*}
To finish the proof, it is enough to choose $T_j$, such that $\frac{\eta_j c_{T_j}}{\eps_j}\to0$.
\end{proof}

We conclude the analysis of alternative formulations by proving that, when the recession function $\Phi_{\infty}$ is given by the squared operator norm $\|\cdot\|_{\rho,\gamma}^2$ (cf. \eqref{eq:rho-gamma-tensor-norm}), the multidimensional cell problem for $g$ explicitly factorizes into the directional norm $\gamma(\nu)$ and the one-dimensional cohesive density $g_{\scal}(\rho(\zeta))$.

\begin{proposition}\label{prop:rho-gamma-identification}
Let $\rho$ be norm on $\R^N$ and assume that
\begin{equation}\label{eq:Phi-rho-gamma}
 \Phi_\infty(\xi)=\|\xi\|_{\rho,\gamma}^2
 \qquad
 \text{for every }\xi\in\R^{N\times n},
\end{equation}
as in \eqref{eq:rho-gamma-tensor-norm}. Then, under the assumptions of
Theorem~\ref{thm:general-anisotropic-gamma-convergence}, one has
\begin{equation}\label{eq:g-rho-gamma-identification}
 g(\zeta,\nu)
 =
 \gamma(\nu)\,
 g_\scal\bigl(\rho(\zeta)\bigr)
\end{equation}
for every $(\zeta,\nu)\in\R^N\times\Sf^{n-1}$, $g_\scal$ is defined in \eqref{e:definizione di g scal}.
\end{proposition}

\begin{proof}
We divide the proof into three steps.

\medskip
\noindent
\textbf{Step 1. Preliminaries.}

We first record the two properties of
$\|\cdot\|_{\rho,\gamma}$ that will be used in the proof. Let $\xi\in\R^{N\times n}$ and let $a\in\R^N$ satisfy
$\rho^\circ(a)\le1$, where $\rho^\circ$ is the dual norm of $\rho$ (cf. \eqref{eq:def-polar-norm}). By definition \eqref{eq:rho-gamma-tensor-norm},
\begin{equation}\label{eq:tensor-projection-inequality}
 \gamma(\xi^{\mathsf T}a)
 \le
 \|\xi\|_{\rho,\gamma},
\end{equation}
and consequently, by \eqref{eq:Phi-rho-gamma},
\begin{equation}\label{eq:Phi-projection-inequality}
 \Phi_\infty(\xi)
 \ge
 \gamma(\xi^{\mathsf T}a)^2.
\end{equation}
We next show that, for every $z\in\R^N$ and $p\in\R^n$,
\begin{equation}\label{eq:tensor-rank-one-identity}
 \|z\otimes p\|_{\rho,\gamma}
 =
 \rho(z)\gamma(p).
\end{equation}
Indeed, for every $a\in\R^N$, $ (z\otimes p)^{\mathsf T}a=(z\cdot a)p.$ Therefore, using the positive homogeneity of $\gamma$,
\begin{align*}
 \|z\otimes p\|_{\rho,\gamma}
 =
 \sup_{\rho^\circ(a)\le1}
 \gamma\bigl((z\cdot a)p\bigr)=
 \gamma(p)
 \sup_{\rho^\circ(a)\le1}|z\cdot a|
 =
 \rho(z)\gamma(p).
\end{align*}
Thus, by \eqref{eq:Phi-rho-gamma},
\begin{equation}\label{eq:Phi-rank-one-identity}
 \Phi_\infty(z\otimes p)
 =
 \rho(z)^2\gamma(p)^2.
\end{equation}
For $L>0$ and $s\ge0$, set
\begin{align}
 m_L(s)
 :=
 \inf_{(\alpha,\beta)\in
 \mathcal{U}_s(-\frac L2,\frac L2)}
 \int_{-\frac L2}^{\frac L2}
 \bigg(
 &\varphi'(0)f^2(\beta)|\alpha'|^2
 +
 \frac{\PotDann(1-\beta)}{4}
 +
 |\beta'|^2
 \bigg)\,\textup{d}t
 \label{eq:def-mL-rho-gamma}
\end{align}
where $\mathcal{U}_s$ is defined in \eqref{e:glispaziu}. In particular, thanks to \cite[Proposition 2.5 item (v)]{ACF26}, we have
\begin{equation}\label{eq:gscalar-large-interval}
 g_\scal(s):=\lim_{L\to\infty}m_L(s).
\end{equation}

\medskip
\noindent
\textbf{Step 2. The lower bound.}

We prove that
\begin{equation}\label{eq:g-rho-gamma-lower}
 g(\zeta,\nu)
 \ge
 \gamma(\nu)g_\scal\bigl(\rho(\zeta)\bigr).
\end{equation}
If $\zeta=0$, the right-hand side is zero, since
$g_\scal(0)=0$, and the inequality follows from the non-negativity
of $g$. We may therefore assume that $\zeta\ne0$ and set $ s:=\rho(\zeta)>0$, $ c_\nu:=\gamma(\nu)>0$. By \eqref{eq:def-polar-norm}, there exists
$a_\zeta\in\R^N$ such that
\begin{equation}\label{eq:rho-calibrating-vector}
 \rho^\circ(a_\zeta)=1,
 \qquad
 a_\zeta\cdot\zeta=\rho(\zeta)=s,
\end{equation}
and $q_\nu\in\R^n$ such that
\begin{equation}\label{eq:gamma-calibrating-vector}
 \gamma^\circ(q_\nu)=1,
 \qquad
 q_\nu\cdot\nu=\gamma(\nu)=c_\nu.
\end{equation}
Decompose $q_\nu$ as
\begin{equation}\label{eq:q-normal-tangential-decomposition}
 q_\nu=q_\nu^\tau+c_\nu\nu,
 \qquad
 q_\nu^\tau\in\Pi_\nu.
\end{equation}
Fix $T>0$ and let $ (u,v)\in\calU_{\zeta,\nu}^T
$. Defining
the scalar function $ \alpha:=a_\zeta\cdot u$, then $ \nabla\alpha=(\nabla u)^{\mathsf T}a_\zeta$.
Since $\rho^\circ(a_\zeta)=1$, inequality
\eqref{eq:Phi-projection-inequality} gives
\begin{equation}\label{eq:bulk-projection-estimate}
 \Phi_\infty(\nabla u)
 \ge
 \gamma(\nabla\alpha)^2
 \qquad\text{a.e.\ in }Q^\nu_T.
\end{equation}
Moreover, by linearity of convolution and
$a_\zeta\cdot\zeta=s$, the boundary condition for $u$ implies
\[
 \alpha=u_{s,\nu}\ast\phi_1
 \qquad\text{on }\partial Q^\nu_T,
\]
where
\[
 u_{s,\nu}(x)
 :=
 \begin{cases}
  s,&x\cdot\nu>0,\\
  0,&x\cdot\nu\le0.
 \end{cases}
\]
Therefore, using \eqref{eq:bulk-projection-estimate},
\begin{align}
 \callF^\infty_1(u,v;Q^\nu_T)
 \ge
 \int_{Q^\nu_T}
 \left(
 \varphi'(0)f(v)^2\gamma(\nabla\alpha)^2
 +
 \frac{\PotDann(1-v)}{4}
 +
 \gamma(\nabla v)^2
 \right)\,\dx.
 \label{eq:vector-to-scalar-energy-lower}
\end{align}
We now reduce the right-hand side of
\eqref{eq:vector-to-scalar-energy-lower} to one dimension. As in Proposition~\ref{p:chargT}, set $ \widetilde Q^\nu_T:=Q^\nu_T\cap\Pi_\nu.$ Choose $R_\nu
 :=
 1+\frac{3|q_\nu^\tau|}{c_\nu}
$ and define the interior part of the tangential section by
\begin{equation}\label{eq:interior-tangential-section-direct}
 \widetilde Q^{\nu,\mathrm{int}}_T
 :=
 \left\{
 y\in\widetilde Q^\nu_T\colon
 \dist\bigl(y,\partial\widetilde Q^\nu_T\bigr)>R_\nu
 \right\}.
\end{equation}
Since $R_\nu$ is independent of $T$, we have
\begin{equation}\label{eq:interior-section-asymptotic-measure}
 \frac{
 \calH^{n-1}(\widetilde Q^{\nu,\mathrm{int}}_T)
 }{T^{n-1}}
 \longrightarrow1
 \qquad\text{as }T\to\infty.
\end{equation}
For every $y\in\widetilde Q^{\nu,\mathrm{int}}_T$, let
\begin{equation}\label{eq:def-line-interval-Iy}
 I_y
 :=
 \left\{
 t\in\R\colon y+tq_\nu\in Q^\nu_T
 \right\}.
\end{equation}
Since $Q^\nu_T$ is open and convex, $I_y$ is an open interval.
We write $I_y=(a_y,b_y)$. For $T>6$, the choice of
$\widetilde Q^{\nu,\mathrm{int}}_T$ implies that
\begin{equation}\label{eq:line-contains-fixed-interval}
 \textstyle \left[-\frac{3}{c_\nu},\frac{3}{c_\nu}\right]
 \subset I_y.
\end{equation}
Indeed, if $|t|\le3/c_\nu$, then $ (y+tq_\nu)\cdot\nu= t(q_\nu\cdot\nu)
 =tc_\nu,$ and hence $|(y+tq_\nu)\cdot\nu|\le3.$ Moreover, by definition of $R_\nu$,
the tangential displacement satisfies $|tq_\nu^\tau| \le\frac{3|q_\nu^\tau|}{c_\nu}< R_\nu,$ proving \eqref{eq:line-contains-fixed-interval}. Define now
\begin{equation}\label{eq:def-direct-slicing-region}
 E_T
 :=
 \left\{
 y+tq_\nu\colon
 y\in\widetilde Q^{\nu,\mathrm{int}}_T,\ 
 t\in I_y
 \right\}.
\end{equation}
By construction $ E_T\subset Q^\nu_T$ and the map
\[
 (y,t)\longmapsto y+tq_\nu,
 \qquad
 y\in\Pi_\nu,\quad t\in\R,
\]
is injective. Indeed, if $y_1+t_1q_\nu=y_2+t_2q_\nu$, then, taking the scalar product with $\nu$ and using
$y_1,y_2\in\Pi_\nu$, we obtain $(t_1-t_2)q_\nu\cdot\nu
 =(t_1-t_2)c_\nu=0.$ Since $c_\nu>0$, it follows that $t_1=t_2$, and hence $y_1=y_2$. 
Consequently, for every non-negative measurable function
$h\colon E_T\to[0,+\infty]$, one has
\begin{equation}\label{eq:direct-line-change-variables}
 \int_{E_T}h(x)\,\dx
 =
 c_\nu
 \int_{\widetilde Q^{\nu,\mathrm{int}}_T}
 \int_{I_y}
 h(y+tq_\nu)\,\textup{d}t\,\textup{d}\calH^{n-1}(y).
\end{equation}
For $\calH^{n-1}$-a.e.\
$y\in\widetilde Q^{\nu,\mathrm{int}}_T$, define
\begin{equation}\label{eq:def-direct-slices-alpha-v}
 \alpha_y(t):=\alpha(y+tq_\nu),
 \qquad
 v_y(t):=v(y+tq_\nu),
 \qquad t\in I_y.
\end{equation}
By the Sobolev slicing theorem, $\alpha_y,v_y\in H^1(I_y)$ and the chain rule yields
\begin{equation}\label{eq:direct-slice-derivatives}
 \alpha_y'(t)
 =
 \nabla\alpha(y+tq_\nu)\cdot q_\nu,
 \qquad
 v_y'(t)
 =
 \nabla v(y+tq_\nu)\cdot q_\nu
\end{equation}
for a.e.\ $t\in I_y$. By the dual inequality for $\gamma$ and $\gamma^\circ$ in \eqref{e:polar inequality},
\[
 |p\cdot q_\nu|
 \le
 \gamma(p)\gamma^\circ(q_\nu)
 =
 \gamma(p)
 \qquad
 \text{for every }p\in\R^n.
\]
It follows from \eqref{eq:direct-slice-derivatives} that
\begin{equation}\label{eq:gamma-controls-direct-alpha-slice}
 |\alpha_y'(t)|
 \le
 \gamma\bigl(\nabla\alpha(y+tq_\nu)\bigr) \qquad \textup{and} \qquad |v_y'(t)|
 \le
 \gamma\bigl(\nabla v(y+tq_\nu)\bigr)
\end{equation}
for a.e.\ $t\in I_y$. We next verify the endpoint conditions of the slices. Since
\[
 y+a_yq_\nu,\ y+b_yq_\nu\in\partial Q^\nu_T
\]
and,
\[
 (y+a_yq_\nu)\cdot\nu=c_\nu a_y\le-3, \qquad (y+b_yq_\nu)\cdot\nu=c_\nu b_y\ge3
\]
from \eqref{eq:line-contains-fixed-interval},
the canonical boundary conditions defining
$\calU^T_{\zeta,\nu}$ imply
\[
 \alpha_y(a_y)=0,
 \qquad
 v_y(a_y)=1, \qquad \alpha_y(b_y)=s,
 \qquad
 v_y(b_y)=1.
\]
Extend $\alpha_y$ and $v_y$ to the whole real line by setting
\[
 \overline\alpha_y(t)
 :=
 \begin{cases}
  0,&t\le a_y,\\
  \alpha_y(t),&a_y<t<b_y,\\
  s,&t\ge b_y,
 \end{cases} \qquad \textup{and} \qquad \overline v_y(t)
 :=
 \begin{cases}
  1,&t\le a_y,\\
  v_y(t),&a_y<t<b_y,\\
  1,&t\ge b_y.
 \end{cases}
\]
Therefore, by definition of
$g_\scal(s)$,
\begin{align}
 &\int_{I_y}
 \left(
 \varphi'(0)f^2(v_y)|\alpha_y'|^2
 +
 \frac{\PotDann(1-v_y)}{4}
 +
 |v_y'|^2
 \right)\,\textup{d}t
 \ge
 g_\scal(s).
 \label{eq:one-dimensional-slice-lower-direct}
\end{align}
We now use the non-negativity of the energy density, the inclusion
$E_T\subset Q^\nu_T$, and
\eqref{eq:vector-to-scalar-energy-lower}. We obtain
\begin{align}
 \callF^\infty_1(u,v;Q^\nu_T)
 &\ge
 \int_{E_T}
 \left(
 \varphi'(0)f^2(v)\gamma(\nabla\alpha)^2
 +
 \frac{\PotDann(1-v)}{4}
 +
 \gamma(\nabla v)^2
 \right)\,\dx.
 \label{eq:restriction-to-ET}
\end{align}
Applying \eqref{eq:direct-line-change-variables} and
\eqref{eq:gamma-controls-direct-alpha-slice}, we find
\begin{align}
 \callF^\infty_1(u,v;Q^\nu_T)
 \ge
 c_\nu
 \int_{\widetilde Q^{\nu,\mathrm{int}}_T}
 \int_{I_y}
 \left(
 \varphi'(0)f^2(v_y)|\alpha_y'|^2
 \right.
  +
 \frac{\PotDann(1-v_y)}{4}
 +
 |v_y'|^2)
 \,\textup{d}t\,\textup{d}\calH^{n-1}(y).
 \label{eq:direct-slicing-energy-lower}
\end{align}
Using \eqref{eq:one-dimensional-slice-lower-direct}, we conclude that
\begin{equation}\label{eq:direct-lower-before-normalization}
 \callF^\infty_1(u,v;Q^\nu_T)
 \ge
 c_\nu
 \calH^{n-1}
 \bigl(\widetilde Q^{\nu,\mathrm{int}}_T\bigr)
 g_\scal(s).
\end{equation}

After division by $T^{n-1}$, this yields
\begin{align}
 \frac{1}{T^{n-1}}
 \callF^\infty_1(u,v;Q^\nu_T)
 \ge
 c_\nu
 \frac{
 \calH^{n-1}(\widetilde Q^{\nu,\mathrm{int}}_T)
 }{T^{n-1}}
 g_\scal(s).
 \label{eq:direct-lower-normalized}
\end{align}
Taking the infimum over
$(u,v)\in\calU^T_{\zeta,\nu}$ and using
Proposition~\ref{p:chargT}, from \eqref{eq:interior-section-asymptotic-measure} we obtain
\begin{align*}
 g(\zeta,\nu)
 =
 \lim_{T\to\infty}
 \inf_{(u,v)\in\calU^T_{\zeta,\nu}}
 \frac{1}{T^{n-1}}
 \callF^\infty_1(u,v;Q^\nu_T)
 \ge
 c_\nu
 \lim_{T\to\infty}
 \frac{
 \calH^{n-1}(\widetilde Q^{\nu,\mathrm{int}}_T)
 }{T^{n-1}}
 g_\scal(s)= c_\nu g_\scal(s).
\end{align*}
Recalling that $ c_\nu=\gamma(\nu)$ and $s=\rho(\zeta)$ we finally obtain
\[
 g(\zeta,\nu)
 \ge
 \gamma(\nu)g_\scal\bigl(\rho(\zeta)\bigr).
\]
\medskip
\noindent
\textbf{Step 3. The upper bound.}

We now prove that
\begin{equation}\label{eq:g-rho-gamma-upper}
 g(\zeta,\nu)
 \le
 \gamma(\nu)g_\scal\bigl(\rho(\zeta)\bigr).
\end{equation}
Assume that $\zeta\ne0$, otherwise the claim is trivial, and set again $s:=\rho(\zeta)$, $c_\nu:=\gamma(\nu)$. Fix $L>0$ and $(\alpha,\beta)
 \in
 \mathcal{U}_s(-L/2,L/2).$ Extending the functions by constants outside
$(-L/2,L/2)$ and truncating, we may assume that
\[
 \alpha(t)=0,\quad \beta(t)=1
 \qquad\text{for }t\le-L/2, \qquad \textup{and} \qquad  \alpha(t)=s,\quad \beta(t)=1
 \qquad\text{for }t\ge L/2,
\]
with $0\le\alpha\le s$, $0\le\beta\le1$. Let $\eps_j\downarrow0$ be any sequence such that $\eps_jc_\nu L<1$ for every $j$. Define, for $x\in Q^\nu$,
\begin{equation}\label{eq:direct-recovery-u-rho-gamma}
 u_j(x)
 :=
 \frac{\zeta}{s}
 \alpha\left(
 \frac{x\cdot\nu}{\eps_jc_\nu}
 \right) \qquad \textup{and} \qquad v_j(x)
 :=
 \beta\left(
 \frac{x\cdot\nu}{\eps_jc_\nu}
 \right).
\end{equation}
Since $\alpha$ and $\beta$ are constant outside
$(-L/2,L/2)$, the pair $(u_j,v_j)$ differs from
$(u_{\zeta,\nu},1)$ only in 
\[
 \left\{
 x\in Q^\nu\colon
 |x\cdot\nu|
 <
 \frac{\eps_jc_\nu L}{2}
 \right\}.
\]
The measure of this region is $\eps_jc_\nu L$. Since $u_j$ and $v_j$
are uniformly bounded, it follows that
\begin{equation}\label{eq:direct-recovery-convergence-rho-gamma}
 (u_j,v_j)\to(u_{\zeta,\nu},1)
 \qquad\text{in }L^2(Q^\nu,\R^{N+1}).
\end{equation}
We now compute the energy: for a.e.\ $x\in Q^\nu$,
\[
 \nabla u_j(x)
 =
 \frac{\zeta}{s}\otimes
 \frac{\nu}{\eps_jc_\nu}
 \alpha'\left(
 \frac{x\cdot\nu}{\eps_jc_\nu}
 \right),
\]
and by the rank-one identity \eqref{eq:Phi-rank-one-identity},
\begin{align}
 \Phi_\infty(\nabla u_j(x))
 =
 \rho\left(\frac{\zeta}{s}\right)^2
 \gamma\left(\frac{\nu}{\eps_jc_\nu}\right)^2
 \left|
 \alpha'\left(
 \frac{x\cdot\nu}{\eps_jc_\nu}
 \right)
 \right|^2
=
 \frac{1}{\eps_j^2}
 \left|
 \alpha'\left(
 \frac{x\cdot\nu}{\eps_jc_\nu}
 \right)
 \right|^2,
 \label{eq:direct-recovery-bulk-computation}
\end{align}
because $\rho\left(\frac{\zeta}{s}\right)=1$ and $\gamma\left(\frac{\nu}{c_\nu}\right)=1$. Similarly,
\begin{equation}\label{eq:direct-recovery-phase-gradient}
 \gamma(\nabla v_j(x))^2
 =
 \frac{1}{\eps_j^2}
 \left|
 \beta'\left(
 \frac{x\cdot\nu}{\eps_jc_\nu}
 \right)
 \right|^2.
\end{equation}
Since the tangential section
$Q^\nu\cap\Pi_\nu$ has $\calH^{n-1}$-measure one, Fubini's theorem
and the change of variables $ t=\frac{x\cdot\nu}{\eps_jc_\nu}$ gives
\begin{align}
 \callF^\infty_{\eps_j}(u_j,v_j;Q^\nu)
 &=
 c_\nu
 \int_{-\frac L2}^{\frac L2}
 \bigg(
 \varphi'(0)f^2(\beta)|\alpha'|^2
 +
 \frac{\PotDann(1-\beta)}{4}
 +
 |\beta'|^2
 \bigg)\,\textup{d}t.
 \label{eq:direct-recovery-energy-identity}
\end{align}
By \eqref{eq:direct-recovery-convergence-rho-gamma} and the
definition of $g$ in \eqref{eq:def-g-gamma}, we infer that
\begin{align}
 g(\zeta,\nu)
 \le
 \liminf_{j\to\infty}
 \callF^\infty_{\eps_j}(u_j,v_j;Q^\nu)
 =
 c_\nu
 \int_{-\frac L2}^{\frac L2}
 \bigg(
 \varphi'(0)f^2(\beta)|\alpha'|^2
 +
 \frac{\PotDann(1-\beta)}{4}
 +
 |\beta'|^2
 \bigg)\,\textup{d}t.
 \label{eq:upper-before-infimum}
\end{align}
Taking the infimum over
$(\alpha,\beta)\in\mathcal{U}_s(-L/2,L/2)$ gives
\[
 g(\zeta,\nu)\le c_\nu m_L(s).
\]
Letting $L\to\infty$ and using
\eqref{eq:gscalar-large-interval}, we conclude the proof.
\end{proof}

The next two Theorems are direct consequence of Theorem \ref{thm:general-anisotropic-gamma-convergence}, Proposition \ref{prop:rho-gamma-identification} and \cite[Theorems 3.1 and 3.2]{alessi2025phasefieldpart2}. In the following we reconstruct a special class of anisotropic cohesive law, by prescribing the degradation function. 

\begin{theorem}\label{thm:ricostruzione fissata degrad}
 Let $g_0:\R^N\times \Sf^{n-1}\to [0,\infty)$ be such that 
 \begin{equation}
     g_0(\zeta,\nu)=\gamma_0(\nu)\widehat g_0(|\zeta|) \qquad \textup{for every $(\zeta,\nu)\in \R^N\times \Sf^{n-1}$},
 \end{equation}
 where $\gamma_0$ is a norm on $\R^n$ and $\widehat g_0$ is a function satisfying \eqref{e:ipotesi legge coesiva} with $(\widehat g_0)'(0)=\varsigma_0$, and let $\theta_0:[0,1]\to [0,\infty)$ be such that $\theta_0^{\frac12}\in C^1([0,1])$, $\theta_0(0)=0$, $\theta_0(1)=\varsigma_0^2$, with $\theta_0'(t)>0$ for every $t\in (0,1)$, and concave in a left neighbourhood of $t=1$. Then there exists $\omega_0\in C^0([0,1])$, such that $\ell(t):=\varsigma_0^{-2}\theta_0(t)$, $\varsigma_0^2 Q(t)=\omega(t)=\omega_0(t)$ on $[0,1]$ satisfy (\textup{H1})-(\textup{H2})and for any $\varphi$ satisfying (\textup{H3})-(\textup{H4}) with $\varphi'(0)=1$ the corresponding functionals $\Fm_\eps$ in \eqref{eq:def-Feps-gamma}, with $\gamma:=\gamma_0$ and $\Phi(\xi):=\|\xi\|^2_{2,\gamma}$ as in \eqref{eq:def-compatible-matrix-norm-gamma}, $\Gamma$-converge to $\Fm_{\varsigma_0}$ in \eqref{eq:def-limit-functional} with $g=g_0$.
\end{theorem} 

Here we reconstruct the cohesive law, by prescribing the damage potential. To this end, we introduce the function $\Psi : [0,1] \to \mathbb{R}$ defined by
\begin{equation}\label{e:definizione di Psi la primitiva}
    \Psi(x) := \int_0^x \PotDann^{1/2}(1 - t) \, \mathrm{d}t.
\end{equation}
\begin{theorem}\label{thm:ricostruzione fissata damage pot}
 Let $g_0:\R^N\times \Sf^{n-1}\to [0,\infty)$ be such that 
 \begin{equation}
     g_0(\zeta,\nu)=\gamma_0(\nu)\widehat g_0(|\zeta|) \qquad \textup{for every $(\zeta,\nu)\in \R^N\times \Sf^{n-1}$},
 \end{equation}
 where $\gamma_0$ is a norm on $\R^n$ and $\widehat g_0$ is a function satisfying \eqref{e:ipotesi legge coesiva} with $(\widehat g_0)'(0)=\varsigma_0$, and let $\omega_0\in C^0([0,1],[0,\infty))$ with $\omega_0^{-1}(\{0\})=\{0\}$, $\omega_0$ increasing and such that $2\Psi_0(1)=\widehat g_0(\infty)$, where $\Psi_0$ is defined as in \eqref{e:definizione di Psi la primitiva} integrating $\omega_0$. Then, there exists $\ell_0\in C^0([0,1])\cap C^1((0,1))$ with $\ell'_0>0$ on $(0,1)$, such that $\ell=\ell_0$, $\varsigma_0^2Q=\omega=\omega_0$ satisfy (\textup{H1})-(\textup{H2}), and for any $\varphi$ satisfying (\textup{H3})-(\textup{H4}) with $\varphi'(0)=1$ the corresponding functionals $\Fm_\eps$ in \eqref{eq:def-Feps-gamma},  with $\gamma:=\gamma_0$ and $\Phi(\xi):=\|\xi\|^2_{2,\gamma}$ as in \eqref{eq:def-compatible-matrix-norm-gamma}, $\Gamma$-converge to $\Fm_{\varsigma_0}$ in \eqref{eq:def-limit-functional} with $g=g_0$.
\end{theorem}

\subsubsection{Structural properties of $g$}
We now characterize the structural behaviour of $g$, demonstrating its linear response for small crack openings.

\begin{proposition}\label{p:crescita della gvettoriale}
Assume that the hypotheses of
Theorem~\ref{thm:general-anisotropic-gamma-convergence} hold.
Then there exists a constant $K\ge1$ such that
\begin{equation}\label{e:crescita della g vettoriale}
 \frac{1}{K}\bigl(1\wedge|\zeta|\bigr)
 \le
 g(\zeta,\nu)
 \le
 K\bigl(1\wedge|\zeta|\bigr)
\end{equation}
for every $(\zeta,\nu)\in\R^N\times\Sf^{n-1}$.
\end{proposition}

\begin{proof}
We first construct two comparison integrands that are compatible
with the norm $\gamma$ and therefore satisfy the assumptions of
Proposition~\ref{prop:rho-gamma-identification}. For $\xi\in\R^{N\times n}$, consider
\begin{equation}\label{eq:def-compatible-matrix-norm-gamma}
 \|\xi\|_{2,\gamma}
 :=
 \sup\left\{
 \gamma(\xi^{\mathsf T}a)\colon
 a\in\R^N,\ |a|\le1
 \right\},
\end{equation}
that is a norm on
$\R^{N\times n}$. In particular there exist constants
$m_\gamma,M_\gamma>0$ such that
\begin{equation}\label{eq:matrix-norm-equivalence-gamma}
 m_\gamma|\xi|^2
 \le
 \|\xi\|_{2,\gamma}^2
 \le
 M_\gamma|\xi|^2
 \qquad
 \text{for every }\xi\in\R^{N\times n}.
\end{equation}
Defining
\begin{equation}\label{eq:def-comparison-Phi-a-b-gamma}
 \Phi_\infty^a(\xi)
 :=
 \frac{1}{C_{\Phi}M_\gamma}\|\xi\|^2_{2,\gamma},
 \qquad
 \Phi_\infty^b(\xi)
 :=
 \frac{C_\Phi}{m_\gamma}\|\xi\|^2_{2,\gamma},
\end{equation}
where $C_\Phi$ is the constant in \eqref{eq:growth-Phi-infty}. We have that
\begin{equation}\label{eq:comparison-Phi-a-Phi-Phi-b}
 \Phi_\infty^a(\xi)
 \le
 \Phi_\infty(\xi)
 \le
 \Phi_\infty^b(\xi)
 \qquad
 \text{for every }\xi\in\R^{N\times n}.
\end{equation}
Let $g^a$ and $g^b$ denote the surface energy densities obtained by
replacing $\Phi_\infty$ in the definition of
$\callF^\infty_\eps$ with $\Phi_\infty^a$ and
$\Phi_\infty^b$, respectively. Since the bulk coefficient $\varphi'(0)\eps f(v)^2$ is non-negative, \eqref{eq:comparison-Phi-a-Phi-Phi-b} implies that
\begin{equation}\label{eq:comparison-functionals-a-original-b}
 \callF^{\infty,a}_\eps(u,v;A)
 \le
 \callF^\infty_\eps(u,v;A)
 \le
 \callF^{\infty,b}_\eps(u,v;A)
\end{equation}
for every admissible pair $(u,v)$, every open set $A$, and every
$\eps>0$. Therefore, taking the corresponding infima and lower limits in the
definitions of the surface densities, we obtain
\begin{equation}\label{eq:comparison-surface-densities-a-g-b}
 g^a(\zeta,\nu)
 \le
 g(\zeta,\nu)
 \le
 g^b(\zeta,\nu)
\end{equation}
for every $(\zeta,\nu)\in\R^N\times\Sf^{n-1}$. Observing that 
\begin{equation*}
    \Phi^a_\infty(\xi)=\|\xi\|^2_{\rho_a,\gamma} \qquad \textup{and} \qquad  \Phi^b_\infty(\xi)=\|\xi\|^2_{\rho_b,\gamma}, 
\end{equation*}
with $\rho_a(\zeta) = \frac{|\zeta|}{\sqrt{C_\Phi M_\gamma}}$ and $\rho_b(\zeta) = \sqrt{\frac{C_\Phi}{m_\gamma}}|\zeta|$,
applying Proposition \ref{prop:rho-gamma-identification}, we obtain
\begin{equation}
    \gamma(\nu)g_\scal\left(\frac{|\zeta|}{\sqrt{C_\Phi M_\gamma}}\right)\leq g(\zeta,\nu)\leq\gamma(\nu)g_\scal\left(\sqrt{\frac{C_\Phi}{m_\gamma}}|\zeta|\right).
\end{equation}
From \cite[Proposition 2.5]{ACF26}, we conclude the proof.
\end{proof}

We proceed to establish the subadditivity and continuity properties of the surface energy density $g$.

\begin{lemma}\label{lemmapropg2}
Under the assumptions of Theorem~\ref{thm:general-anisotropic-gamma-convergence}, the following hold:
\begin{enumerate}
    \item\label{lemmapropgsubadd}
    For every $\nu\in \Sf^{n-1}$ and $\zeta_1,\zeta_2\in\R^N$ one has
    \[
    g(\zeta_1+\zeta_2,\nu)\le g(\zeta_1,\nu)+g(\zeta_2,\nu);
    \]
    \item\label{lemmapropgcont}
    The function $g$ is continuous on $\R^N\times \Sf^{n-1}$.
\end{enumerate}
\end{lemma}

\begin{proof}
\noindent\textit{(i) Subadditivity.}
Fix $\zeta_1,\zeta_2\in\R^N$ and $\nu\in \Sf^{n-1}$.
Let $(u_j^i,v_j^i)$ be the sequences given in
Proposition~\ref{periodicity} corresponding to $\eps_j:=1/j$ and to the
pairs $(\zeta_i,\nu)$, for $i=1,2$.
We extend these functions periodically along the directions tangent to
$\Pi_{\nu}\cap Q^\nu$, and constantly along $\nu$. In particular, for
each $i\in\{1,2\}$ and all $j$, we have $u_j^i=\zeta_i$ and $v_j^i=1$
in $\{x\cdot\nu\ge\frac12\}$, and $u_j^i=0$, $v_j^i=1$ in
$\{x\cdot\nu\le-\frac12\}$.

We now use a rescaling analogous to that used in
Proposition~\ref{periodicity}.
Let $M_j\in\N$ with $M_j\to\infty$, and define $(u_j,v_j)$ by
\[
u_j(x):=
\begin{cases}
u_j^1(M_jx+\frac12\nu), & \text{if } x\cdot\nu<0,\\[3pt]
\zeta_1+u_j^2(M_jx-\frac12\nu), & \text{if } x\cdot\nu\ge0,
\end{cases}
\qquad
v_j(x):=
\begin{cases}
v_j^1(M_jx+\frac12\nu), & \text{if } x\cdot\nu<0,\\[3pt]
v_j^2(M_jx-\frac12\nu), & \text{if } x\cdot\nu\ge0.
\end{cases}
\]
Because of periodicity in the tangential directions, $(u_j,v_j)$
belongs to $H^1(Q^\nu;\R^{N+1})$.
Moreover, $u_j=0$ and $v_j=1$ on
$\{x\cdot\nu\le-\frac1{M_j}\}$, while
$u_j=\zeta_1+\zeta_2$ and $v_j=1$ on
$\{x\cdot\nu\ge\frac1{M_j}\}$, and the pair is
$\frac1{M_j}$--periodic along the directions of
$\Pi_{\nu}\cap Q^\nu$.

By a change of variables,
\begin{equation*}
\begin{split}
\|u_j-u_{\zeta_1+\zeta_2,\nu}\|_{L^1(Q^\nu)}
&=
\|u_j-u_{\zeta_1+\zeta_2,\nu}\|_
{L^1(Q^\nu\cap\{|x\cdot\nu|\leq\frac1{M_j}\})}=
\frac1{M_j}\|u_j^1\|_{L^1(Q^\nu)}
+\frac1{M_j}\|u_j^2-\zeta_2\|_{L^1(Q^\nu)}
\\
&\le
\frac1{M_j}\Big(
\|u_j^1-u_{\zeta_1,\nu}\|_{L^1(Q^\nu)}
+
\|u_j^2-u_{\zeta_2,\nu}\|_{L^1(Q^\nu)}
\Big)
+\frac{|\zeta_1|+|\zeta_2|}{2M_j},
\end{split}
\end{equation*}
hence $u_j\to u_{\zeta_1+\zeta_2,\nu}$ in
$L^1(Q^\nu;\R^N)$.

Furthermore, using also the positive one-homogeneity of the norm
$\gamma$, an immediate computation yields
\[
\callF^{\infty}_{\eps_j/M_j}(u_j,v_j;Q^\nu)
=
\callF^{\infty}_{\eps_j}(u_j^1,v_j^1;Q^\nu)
+
\callF^{\infty}_{\eps_j}(u_j^2,v_j^2;Q^\nu).
\]
Taking the limit as $j\to\infty$ gives the desired inequality in
\ref{lemmapropgsubadd}.

\medskip
\noindent\textit{(ii) Continuity.}
From \ref{lemmapropgsubadd} and
Proposition~\ref{p:crescita della gvettoriale}, we deduce that
\[
g(\zeta,\nu)
\le
g(\zeta',\nu)+K|\zeta-\zeta'|,
\]
where $K$ (cf.~\eqref{e:crescita della g vettoriale}) does not depend
on $\nu$.
Interchanging $\zeta$ and $\zeta'$ therefore gives
\[
|g(\zeta,\nu)-g(\zeta',\nu)|
\le K|\zeta-\zeta'|.
\]
Thus it remains only to show the continuity with respect to $\nu$.

Since $\Phi_{\infty}$ is continuous and strictly positive on the
compact unit sphere $\Sf^{N\times n-1}$, there exists a nondecreasing
modulus of continuity $\vartheta_\Phi$, with
$\vartheta_\Phi(0)=0$, such that
\[
\Phi_{\infty}(\xi)
\le
\bigl(1+\vartheta_\Phi(|\xi-\xi'|)\bigr)
\Phi_{\infty}(\xi')
\qquad
\text{whenever }|\xi|=|\xi'|=1.
\]
By the $2$-homogeneity of $\Phi_\infty$, this implies, after possibly
enlarging $\vartheta_\Phi$, that
\begin{equation}\label{eqpsiinfunifcont}
\Phi_\infty(\eta R)
\le
\bigl(1+\vartheta_\Phi(|R-\Id|)\bigr)
\Phi_\infty(\eta)
\qquad
\text{for every }\eta\in\R^{N\times n},\ R\in O(n).
\end{equation}
Thanks to \eqref{e:gamma lower euclidean suppenergy}, for every $\xi\in\R^n$ and $R\in O(n)$,
\[
\begin{split}
\gamma(\xi R)
\le
\gamma(\xi)+\gamma\bigl(\xi(R-\Id)\bigr)
\le
\gamma(\xi)+C_\gamma|\xi|\,|R-\Id|
\le
\left(1+\frac{C_\gamma}{c_\gamma}|R-\Id|\right)\gamma(\xi).
\end{split}
\]
Consequently,
\begin{equation}\label{eq:gammarotation}
\gamma^2(\xi R)
\le
\bigl(1+\vartheta_\gamma(|R-\Id|)\bigr)\gamma^2(\xi),
\end{equation}
where, 
\[
\vartheta_\gamma(t)
:=
2\frac{C_\gamma}{c_\gamma}t
+
\left(\frac{C_\gamma}{c_\gamma}\right)^2t^2.
\]
In particular, $\vartheta_\gamma$ is nondecreasing and
$\vartheta_\gamma(t)\to0$ as $t\to0$. Setting $\vartheta(t):=
\max\{\vartheta_\Phi(t),\vartheta_\gamma(t)\}
$, then both \eqref{eqpsiinfunifcont} and
\eqref{eq:gammarotation} hold with $\vartheta$ in place of
$\vartheta_\Phi$ and $\vartheta_\gamma$, respectively.

Fix $\nu\in\Sf^{n-1}$ and a sequence $\eps_j\to0$, and let
$(u_j,v_j)$ be the sequence given by Proposition~\ref{periodicity},
extended periodically along $\Pi_\nu$ and constantly in the
$\nu$--direction as in part~\ref{lemmapropgsubadd}.
Let $\tilde\nu\in\Sf^{n-1}$, $\tilde\nu\neq\nu$, and choose
$R\in O(n)$ such that
\[
\nu=R\tilde\nu,
\qquad
|R-\Id|\le c|\nu-\tilde\nu|
\]
(for instance, $R$ acts as the identity on the subspace orthogonal to
$\operatorname{span}\{\nu,\tilde\nu\}$ and rotates $\tilde\nu$ to
$\nu$ in that plane). Let $M_j\to\infty$ and define
\[
\tilde u_j(x):=u_j(M_jRx),
\qquad
\tilde v_j(x):=v_j(M_jRx).
\]
Since
\[
(Rx)\cdot\nu
=
x\cdot R^T\nu
=
x\cdot\tilde\nu,
\]
and $u_j\to u_{\zeta,\nu}$ locally, it follows that
\[
\tilde u_j\to u_{\zeta,\tilde\nu}
\qquad\text{in }L^1(Q^{\tilde\nu};\R^N).
\]
Moreover,
\[
\nabla\tilde u_j(x)
=
M_j\nabla u_j(M_jRx)R,
\qquad
\nabla\tilde v_j(x)
=
M_j\nabla v_j(M_jRx)R.
\]
Using \eqref{eqpsiinfunifcont}, we obtain
\[
\Phi_{\infty}(\nabla\tilde u_j)(x)
=
M_j^2
\Phi_{\infty}\bigl(\nabla u_j(M_jRx)R\bigr)
\le
M_j^2
\bigl(1+\vartheta(|R-\Id|)\bigr)
\Phi_{\infty}(\nabla u_j)(M_jRx).
\]
Likewise, by \eqref{eq:gammarotation},
\[
\gamma^2(\nabla\tilde v_j(x))
=
M_j^2
\gamma^2\bigl(\nabla v_j(M_jRx)R\bigr)
\le
M_j^2
\bigl(1+\vartheta(|R-\Id|)\bigr)
\gamma^2(\nabla v_j(M_jRx)).
\]
Therefore, substituting these estimates into the definition of
$\callF^\infty_{\eps_j/M_j}
(\tilde u_j,\tilde v_j;Q^{\tilde\nu})$
and performing the change of variables $y=M_jRx$, we obtain
\[
\callF_{\frac{\eps_j}{M_j}}^{\infty}
(\tilde u_j,\tilde v_j;Q^{\tilde\nu})
\le
\bigl(1+\vartheta(|R-\Id|)\bigr)
M_j^{1-n}
\callF_{\eps_j}^{\infty}
(u_j,v_j;M_jRQ^{\tilde\nu}).
\]

Since $(u_j,v_j)$ are periodic in the directions orthogonal to $\nu$,
the set
$\Pi_{\nu}\cap M_jRQ^{\tilde\nu}$ can be covered by at most $M_j^{n-1}+cM_j^{n-2}$ disjoint copies of $\Pi_{\nu}\cap Q^\nu$. Hence,
\begin{align*}
 g(\zeta,\tilde\nu)
&\le
\limsup_{j\to\infty}
\callF_{\frac{\eps_j}{M_j}}^{\infty}
(\tilde u_j,\tilde v_j;Q^{\tilde\nu})
\le
\bigl(1+\vartheta(|R-\Id|)\bigr)
\limsup_{j\to\infty}
\left(1+\frac{c}{M_j}\right)
\callF_{\eps_j}^{\infty}(u_j,v_j;Q^\nu)
\\
&=
\bigl(1+\vartheta(|R-\Id|)\bigr)g(\zeta,\nu)\le
\bigl(1+\vartheta(c|\nu-\tilde\nu|)\bigr)g(\zeta,\nu).
\end{align*}
Interchanging the roles of $\nu$ and $\tilde\nu$, we also obtain
\[
g(\zeta,\nu)
\le
\bigl(1+\vartheta(c|\nu-\tilde\nu|)\bigr)
g(\zeta,\tilde\nu).
\]
Since $\vartheta(t)\to0$ as $t\to0$, this proves the continuity of
$g(\zeta,\cdot)$ on $\Sf^{n-1}$ and hence, together with the
Lipschitz continuity with respect to $\zeta$, the continuity of $g$
on $\R^N\times\Sf^{n-1}$.
\end{proof}

We conclude this section by noting that, under the lower semicontinuity assumption on the effective limit functional $\mathcal{F}_{\varsigma}$, the surface energy density $g$ shares the same asymptotic behaviour near the origin as the recession function $h_{\varsigma}^{\qc,\infty}$. We omit the proof, as it follows verbatim the arguments of~\cite[Proposition~3.9]{CFI24}.

\begin{proposition}\label{p:behaviour g at 0}
Assume that the limit functional $\Fm_\varsigma$ (see \eqref{eq:def-limit-functional}) is
lower semicontinuous in $L^1(\Omega;\R^N)$.
Then, under the assumptions of Theorem~\ref{thm:general-anisotropic-gamma-convergence},
for every $\nu\in \Sf^{n-1}$ it holds that
\[
\lim_{\zeta\to 0}
\frac{g(\zeta,\nu)}
{h^\qcinfty_{\varsigma}(\zeta\otimes\nu)}
=1.
\]
\end{proposition}

\section{Lower bound}\label{sec:lower-bound}

In this section, we establish the $\Gamma$-$\liminf$ inequality for the sequence of functionals $\{\mathcal{F}_\varepsilon\}_{\varepsilon>0}$. The proof is conceptually divided into two parts, corresponding to the diffuse and surface contributions of the limit functional. The surface part relies on a localized blow-up argument around the jump points, whereas the diffuse part is handled by means of standard truncation techniques and the partial relaxation result established in \cite{CFI24}. Finally, we characterize the effective domain of the $\Gamma$-limit, demonstrating that the limit energy is infinite for functions belonging to $(L^1\setminus GBV(\Omega))^N$.

We begin by addressing the surface contribution within the $BV$ framework.

\begin{proposition}\label{p:lowerboundBVsfcvettoriale}
Let $u\in BV(\Omega,\R^N)$ and assume that a sequence
$(u_{\eps},v_{\eps})$ converges to $(u,1)$ strongly in
$L^1(\Omega,\R^{N+1})$.
Then, under the hypotheses of Theorem~\ref{thm:general-anisotropic-gamma-convergence}, we have the
lower bound
\begin{equation}\label{lbjumpvectoriale}
\int_{J_u\cap A}g([u],\nu_u)\,\diff \calH^{n-1}
\leq
\liminf_{\eps\to0}\Fm_\eps(u_\eps,v_\eps,A),
\end{equation}
for every $A\in\mathcal A(\Omega)$, where the surface energy density $g$ is
defined in \eqref{eq:def-g-gamma}.
\end{proposition}

\begin{proof}
Fix an arbitrary $A\in\mathcal A(\Omega)$.
We may initially assume that the inferior limit on the right-hand side
of \eqref{lbjumpvectoriale} is finite, since otherwise the claim holds
trivially.
Consequently, up to extracting a subsequence, we can suppose that
\[
(u_\eps,v_\eps)\in H^1(\Omega;\R^N\times[0,1])
\]
and that the aforementioned $\liminf$ is actually a limit.

For every $\eps>0$, we introduce the localized energy Radon measure
$\mu_\eps$ on $A$ defined by
\begin{equation*}
\mu_\eps
:=
\left(
\varphi(\eps f^2(v_\eps))\Phi(\nabla u_\eps)
+
\frac{\PotDann(1-v_\eps)}{4\eps}
+
\eps\gamma^2(\nabla v_\eps)
\right)
\calL^n\res A.
\end{equation*}
By weak$^*$ compactness, extracting a further subsequence if necessary,
we obtain
\[
\mu_\eps\weakto\mu
\]
for some limiting Radon measure $\mu$ on $A$. This weak convergence
guarantees that
\begin{equation*}
\liminf_{\eps\to0}
\mathcal F_\eps(u_\eps,v_\eps,A)
\geq \mu(A).
\end{equation*}
To prove \eqref{lbjumpvectoriale}, it is therefore sufficient to
establish the pointwise bound
\begin{equation}\label{e:claim lower sup}
\frac{d\mu}{d\calH^{n-1}\res J_u}(x_0)
\geq
g([u](x_0),\nu_u(x_0))
\end{equation}
for $\calH^{n-1}$-a.e.\ $x_0\in J_u\cap A$.

Let us fix a point $x_0\in J_u\cap A$ that simultaneously satisfies
the following regularity conditions:
\begin{equation*}
\frac{d\mu}
{d(\calH^{n-1}\res(J_u\cap A))}(x_0)
=
\lim_{\rho\to0}
\frac{\mu(Q^\nu_\rho(x_0))}
{\calH^{n-1}(J_u\cap Q^\nu_\rho(x_0))}
<\infty,
\end{equation*}
\begin{equation*}
\lim_{\rho\to0}
\frac{\calH^{n-1}
(J_u\cap Q^\nu_\rho(x_0))}
{\rho^{n-1}}
=1,
\end{equation*}
and
\begin{equation*}
\lim_{\rho\to0}
\frac1{\rho^n}
\int_{Q^\nu_\rho(x_0)}
|u(y)-u_{x_0,\zeta,\nu}(y)|\,dy
=0,
\end{equation*}
where we have set
\[
\nu:=\nu_u(x_0),
\qquad
\zeta:=[u](x_0).
\]
By standard theory on $BV$ functions, these conditions identify a set
of full $\calH^{n-1}$-measure within $J_u\cap A$.

From the weak$^*$ convergence of $\mu_\eps$, we can express the density
at $x_0$ as
\begin{equation}\label{e:doppio limite radon-nyk parte sup}
\frac{d\mu}{d\calH^{n-1}\res J_u}(x_0)
=
\lim_{\rho\to0}
\frac{\mu(Q^\nu_\rho(x_0))}{\rho^{n-1}}
=
\lim_{\substack{\rho\to0\\ \rho\in I(x_0)}}
\lim_{\eps\to0}
\frac{\mu_\eps(Q^\nu_\rho(x_0))}
{\rho^{n-1}},
\end{equation}
where $I(x_0)$ is the set of ``good'' radii
\begin{equation*}
I(x_0)
:=
\left\{
\rho\in
\left(
0,\frac{2}{\sqrt n}\dist(x_0,\partial A)
\right):
\mu(\partial Q^\nu_\rho(x_0))=0
\right\}.
\end{equation*}

Next, as in the proof of \cite[Proposition~3.3]{ACF26}, we fix a
threshold parameter $\delta\in(0,1)$ and, for every $\eps>0$, define
the truncation level
\[
\delta_\eps
:=
\sup\left\{
t\in[0,1):
\eps f^2(t)\leq\delta
\right\}.
\]
Since $f(0)=0$ and $f(t)\to\infty$ as $t\to1^-$,
$\delta_\eps\in(0,1)$ is actually a maximum by the continuity of $f$
on $[0,1)$. Moreover, for all $t\in(\delta_\eps,1)$,
\[
\delta<\eps f^2(t),
\qquad
\eps f^2(\delta_\eps)=\delta.
\]
In particular, the latter identity implies
\[
\delta_\eps\to1
\qquad\text{as }\eps\to0.
\]
Moreover, by (H1), there exists $t_0\in(0,1)$ such that $f$ is
increasing on $[t_0,1)$. Defining $M:=\max_{[0,t_0]}f$, we have for all $\eps$ sufficiently small, $\frac{\delta}{\eps}>M^2
\text{ and }
\delta_\eps\in(t_0,1).$ In particular, when $\eps$ is small enough,
$t\in[0,\delta_\eps]$ if and only if
\begin{equation}\label{e:bounddeltaciro bis}
\eps f^2(t)\leq\delta.
\end{equation}

We then introduce the truncated phase-field $\hat v_\eps:=v_\eps\land\delta_\eps$ and decompose the energy as follows:
\begin{align*}
\mathcal F_\eps
(u_\eps,\hat v_\eps,Q^\nu_\rho(x_0))
&=
\mathcal F_\eps
\bigl(
u_\eps,v_\eps,
Q^\nu_\rho(x_0)\setminus\{v_\eps>\delta_\eps\}
\bigr)+
\int_{Q^\nu_\rho(x_0)\cap\{v_\eps>\delta_\eps\}}
\varphi(\delta)\Phi(\nabla u_\eps)\,dx
\\
&\quad+
\frac{\PotDann(1-\delta_\eps)}{4\eps}
\calL^n
\bigl(
Q^\nu_\rho(x_0)\cap\{v_\eps>\delta_\eps\}
\bigr).
\end{align*}
Given that $\varphi$ is non-decreasing and using
\eqref{e:bounddeltaciro bis}, it follows that $\varphi(\delta)
\leq
\varphi(\eps f^2(v_\eps))\text{ on }\{v_\eps>\delta_\eps\}.$ Consequently,
\begin{align*}
\mathcal F_\eps
(u_\eps,\hat v_\eps,Q^\nu_\rho(x_0))
\leq
\mathcal F_\eps
(u_\eps,v_\eps,Q^\nu_\rho(x_0))
+
\frac{\PotDann(1-\delta_\eps)}{4\eps}
\calL^n
\bigl(
Q^\nu_\rho(x_0)\cap\{v_\eps>\delta_\eps\}
\bigr).
\end{align*}
Appealing to the definition of $\delta_\eps$, condition
\eqref{eq:def-varsigma}, and the fact that
$\delta_\eps\to1$ as $\eps\to0$, we conclude that, for $\eps$
sufficiently small (depending on $\delta$),
\begin{equation*}
\mathcal F_\eps
(u_\eps,\hat v_\eps,Q^\nu_\rho(x_0))
\leq
\mathcal F_\eps
(u_\eps,v_\eps,Q^\nu_\rho(x_0))
+
\frac{(\varsigma^2+\delta)\rho^n}{4\delta}.
\end{equation*}
Substituting this into
\eqref{e:doppio limite radon-nyk parte sup} and then letting
$\rho\to0$ provides the key lower estimate
\begin{equation}\label{e:penultima stima con der radon-nyk lower sup}
\frac{d\mu}{d\calH^{n-1}\res J_u}(x_0)
\geq
\limsup_{\substack{\rho\to0\\\rho\in I(x_0)}}
\limsup_{\eps\to0}
\frac{
\mathcal F_\eps
(u_\eps,\hat v_\eps,Q^\nu_\rho(x_0))
}{\rho^{n-1}}.
\end{equation}

For an arbitrary $\eta\in(0,1)$, conditions
\eqref{eq:uniform-recession-estimate} and
\eqref{eq:growth-Phi-infty} guarantee that $\Phi(\xi)\geq(1-\eta)\Phi_\infty(\xi)$
for sufficiently large $|\xi|$. Thus, by the continuity of the
recession function $\Phi_\infty$, there exists a constant
$C(\eta)>0$ ensuring the global bound
\begin{equation*}
\Phi(\xi)+C(\eta)
\geq
(1-\eta)\Phi_\infty(\xi)
\qquad
\text{for every }\xi\in\R^{N\times n}.
\end{equation*}
For each ``good'' radius $\rho\in I(x_0)$, we select
$\eta_\rho\in(0,1)$ such that
\begin{equation}\label{e:cicciomaccio}
\eta_\rho\to0,
\qquad
\rho C(\eta_\rho)\to0
\qquad
\text{as }\rho\to0.
\end{equation}
Using \eqref{e:bounddeltaciro bis} and hypothesis (H4), we can write
\begin{equation*}
\varphi(\eps f^2(\hat v_\eps))
\geq
(1-\tilde\vartheta(\delta))
\varphi'(0^+)\eps f^2(\hat v_\eps),
\end{equation*}
where $\tilde\vartheta(\delta)\to0^+$ as $\delta\to0^+$. Consequently,
\begin{align*}
\varphi(\eps f^2(\hat v_\eps))
\Phi(\nabla u_\eps)
\geq
(1-\tilde\vartheta(\delta))
\varphi'(0^+)\eps f^2(\hat v_\eps)
\bigl(
(1-\eta_\rho)\Phi_\infty(\nabla u_\eps)
-
C(\eta_\rho)
\bigr).
\end{align*}
Integrating this bound over the cube and recalling
\eqref{e:bounddeltaciro bis}, we deduce
\begin{align*}
\frac{
\mathcal F_\eps
(u_\eps,\hat v_\eps,Q^\nu_\rho(x_0))
}{\rho^{n-1}}
&\geq
\frac{
(1-\tilde\vartheta(\delta))
(1-\eta_\rho)
\mathcal F^\infty_\eps
(u_\eps,\hat v_\eps,Q^\nu_\rho(x_0))
}{\rho^{n-1}}-
\frac{
C(\eta_\rho)
\calL^n(Q^\nu_\rho(x_0))
(1-\tilde\vartheta(\delta))
\varphi'(0^+)\delta
}{\rho^{n-1}},
\end{align*}
where the asymptotic functional
$\mathcal F^\infty_\eps$ is defined in
\eqref{eq:def-recession-energy}.

Indeed, the potential and damage terms are nonnegative and therefore
may be multiplied by the factor
\[
(1-\tilde\vartheta(\delta))(1-\eta_\rho)\leq1.
\]

Therefore, taking the superior limits, we arrive at
\begin{equation}\label{e:ultima stima der radon nyk lowe sup}
\frac{d\mu}{d\calH^{n-1}\res J_u}(x_0)
\geq
(1-\tilde\vartheta(\delta))
\limsup_{\substack{\rho\to0\\\rho\in I(x_0)}}
\limsup_{\eps\to0}
\frac{
\mathcal F^\infty_\eps
(u_\eps,\hat v_\eps,Q^\nu_\rho(x_0))
}{\rho^{n-1}}.
\end{equation}

Finally, we perform the standard blow-up procedure. For $\rho>0$ set
\[
u^\rho_\eps(y)
:=
u_\eps(\rho y+x_0),
\qquad
\hat v^\rho_\eps(y)
:=
\hat v_\eps(\rho y+x_0).
\]
By the $2$-homogeneity of $\Phi_\infty$ and the $1$-homogeneity of the
norm $\gamma$, we have
\begin{equation}\label{e:blowup Finfty gamma}
\frac{
\mathcal F^\infty_\eps
(u_\eps,\hat v_\eps,Q^\nu_\rho(x_0))
}{\rho^{n-1}}
=
\mathcal F^\infty_{\eps/\rho}
(u^\rho_\eps,\hat v^\rho_\eps,Q^\nu).
\end{equation}
Because $u_\eps\to u$ strongly in $L^1(\Omega;\R^N)$ and
$v_\eps\to1$ strongly in $L^1(\Omega)$, while
$\delta_\eps\to1$, a standard diagonal argument allows us to choose
$\rho_j\to0$, with $\rho_j\in I(x_0)$, and
$\varepsilon_j\to0$ such that
\[
\tilde\varepsilon_j
:=
\frac{\varepsilon_j}{\rho_j}
\longrightarrow0, \quad \textup{and} \quad (u^{\rho_j}_{\varepsilon_j},\hat v^{\rho_j}_{\varepsilon_j})
\longrightarrow
(u_{\zeta,\nu},1)
\qquad\text{in }L^1(Q^\nu;\R^{N+1}),
\]
and 
\begin{align}
\limsup_{\substack{\rho\to0\\\rho\in I(x_0)}}
\limsup_{\eps\to0}
\frac{
\mathcal F^\infty_\eps
(u_\eps,\hat v_\eps,Q^\nu_\rho(x_0))
}{\rho^{n-1}}=
\lim_{j\to\infty}
\mathcal F^\infty_{\tilde\varepsilon_j}
\bigl(
u^{\rho_j}_{\varepsilon_j},
\hat v^{\rho_j}_{\varepsilon_j},
Q^\nu
\bigr).
\label{e:ultima stima lower sup}
\end{align}
By definition of $g$ in \eqref{eq:def-g-gamma}, we
therefore obtain
\[
\lim_{j\to\infty}
\mathcal F^\infty_{\tilde\varepsilon_j}
\bigl(
u^{\rho_j}_{\varepsilon_j},
\hat v^{\rho_j}_{\varepsilon_j},
Q^\nu
\bigr)
\geq
g(\zeta,\nu).
\]
Coupling this with
\eqref{e:ultima stima der radon nyk lowe sup} and \eqref{e:cicciomaccio} yield
\[
\frac{d\mu}{d\calH^{n-1}\res J_u}(x_0)
\geq
(1-\tilde\vartheta(\delta))
g(\zeta,\nu),
\]
concluding the proof, taking $\delta\to0$.
\end{proof}

Here, we recall the partial relaxation result from 
\cite[Proposition~4.2]{CFI24}, adapted to the present framework.
\begin{theorem}\label{t:CFI_sci}
Let $\zeta\in(0,\infty)$, and $\vartheta:\R^{N\times n}\to[0,\infty)$ be Borel, quasi-convex map such that for some constant $c>0$
\begin{equation*}
0\le \vartheta(\xi)\le c(1+|\xi|) \quad\textup{for all } \xi\in\R^{N\times n}\,.
\end{equation*}
Define $F:L^1(\Omega,\R^N)\times\mathcal A(\Omega)\to[0,\infty]$ by
\begin{equation*}
F(w,A):=\int_A \vartheta(\nabla w) \dx + \zeta\,\mathcal{H}^{n-1}(A\cap J_{w})
\end{equation*}
if $w\in SBV(\Omega,\R^N)$, and $\infty$ otherwise.  Then, for any $w_j,w\in BV\cap L^\infty(\Omega,\R^N)$, with $w_j\to w$ in $L^1(A,\R^N)$, we have
\begin{equation*}
\liminf_{j\to \infty} F(w_j,A)\geq\int_A\vartheta(\nabla w)\dx+ \int_A \vartheta^{\infty}(\diff D^c w)\,,
\end{equation*}    
where $\vartheta^{\infty}$ is the recession function of $\vartheta$.
\end{theorem}

Next, we establish the lower bound for the diffuse part on $BV\cap L^\infty$. 

\begin{proposition}\label{p:diffuseliminf sigmabarra finito vettoriale}
Let $u\in BV\cap L^\infty(\Omega,\R^N)$ and assume that the sequence $(u_{\eps},v_{\eps})\in L^1(\Omega,\R^{N+1})$ converges strongly to $(u,1)$ in $L^1(\Omega,\R^{N+1})$. Then, provided that the assumptions of Theorem~\ref{thm:general-anisotropic-gamma-convergence} hold, we have the following lower bound for the diffuse energy:
\begin{equation}\label{e:liminf diffuse sigmabarra finito vettoriale}
    \int_A h^{\qc}_{\varsigma}(\nabla u)\dx + \int_A h^{\qc,\infty}_{\varsigma}(\diff D^cu) \leq \liminf_{\eps \to 0} \Fm_\eps(u_{\eps},v_{\eps},A)\,,
\end{equation}
for every $A\in\mathcal A(\Omega)$.
\end{proposition}

\begin{proof}
Fix an arbitrary open set $A\in \mathcal{A}(\Omega)$. We can assume without loss of generality that the inferior limit on the right-hand side of \eqref{e:liminf diffuse sigmabarra finito vettoriale} is finite, as the inequality holds trivially otherwise. Moreover, since $u\in L^{\infty}(\Omega,\R^N)$, thanks to Lemma \ref{l:5.3} we may assume that $(u_\eps,v_\eps)\in H^1\cap L^\infty(\Omega,\R^{N+1})$ for every $\eps>0$.

Let $\eta\in(0,1\wedge \varsigma)$ and invoking the notation and results established in Lemma~\ref{l:suppenergy}, we can find $\delta_{\eta}\in(0,1)$ such that, for every $\delta\in(0,\delta_{\eta})$, the truncated function $u^\delta_\eps$ belongs to the space $SBV(A,\R^N)$.
This allows us to bound the effective energy from below as follows:
\begin{align*}
 H_\delta(\tilde u_{\eps}^\delta,A)&=\zeta_1(\delta)\int_A h^{\qc}_{\varsigma_\eta}(\nabla \tilde u^{\delta}_\eps) \dx 
+ \zeta_2(\delta)\mathcal{H}^{n-1}(A\cap J_{\tilde u^\delta_\eps}) 
 \leq\Fm_\eps(u_\eps,v_\eps,A)
+ h_{\varsigma_\eta}(0)\calL^{n}(A\cap \{v_{\eps}\leq 1-\delta\})\,,
\end{align*}
where the shifted parameter $\varsigma_{\eta}:=(\varsigma-\eta)(1-\eta)$ as in Lemma \ref{l:suppenergy}. 

Taking the inferior limit as $\eps\to0$ and applying the lower semicontinuity result of Theorem~\ref{t:CFI_sci}, we deduce:
\begin{equation*}
    \zeta_1(\delta)\int_A h_{\varsigma_\eta}^{\qc}(\nabla u)\dx 
    + \zeta_1(\delta)\int_A h^{\qc,\infty}_{\varsigma_\eta}(\diff D^c u)
    \leq \liminf_{\eps \to 0} \Fm_\eps(u_{\eps},v_{\eps},A)\,,
\end{equation*}
because $v_\eps\to 1$ in $L^1(\Omega)$. Finally, by passing to the limit first as $\delta\to0$ and subsequently as $\eta\to0$, we recover the desired lower bound \eqref{e:liminf diffuse sigmabarra finito vettoriale}. This sequential limit is fully justified by the asymptotic behavior in \eqref{e:limittruncat}, the monotonicity property \eqref{e:la h al variare di sigma}, and a direct application of Beppo Levi's Monotone Convergence Theorem.
\end{proof}

Combining Propositions~\ref{p:lowerboundBVsfcvettoriale} and~\ref{p:diffuseliminf sigmabarra finito vettoriale} via a standard localization argument, and appealing to the De Giorgi-Letta measure criterion together with \cite[Proposition~1.16]{Bra98} to exploit the mutual singularity of the Lebesgue, jump, and Cantor parts of $Du$, we directly obtain the full $\Gamma$-$\liminf$ inequality for functions in $BV\cap L^\infty$ (cf. \cite[Theorem 4.6]{CFI24}).

\begin{theorem}\label{thm:lbcompBV}
Let $u\in BV\cap L^\infty(\Omega,\mathbb R^N)$. Under the assumptions of Theorem~\ref{thm:general-anisotropic-gamma-convergence}, we have
\begin{equation}\label{e:lbBV}
\mathcal{F}_{\varsigma}(u,1;\Omega) \le \Gamma(L^1)\hbox{-}\liminf_{\eps\to0} \Fm_\eps(u, 1;\Omega).
\end{equation}
\end{theorem}

In the subsequent proposition, we characterize the domain of the $\Gamma$-$\limsup$. Specifically, we prove that it is contained in the space $GBV_{\star}(\Omega,\mathbb R^N)$, introduced in \cite{dalmaso_toader_2022, donati_2025}. This function space consists of all maps $u\in (GBV(\Omega))^N$ such that, for every component $i\in\{1,\dots,N\}$ and every truncation level $M>0$, one has
\begin{equation*}
\int_\Omega |\nabla \tau_M(u_i)| \dx 
+ |D^c(\tau_M(u_i))|(\Omega)
+ \int_{J_{\tau_M(u_i)}} (|[\tau_M(u_i)]|\wedge1)\, \diff \calH^{n-1}\leq L,
\end{equation*}
for some uniform constant $L>0$. Here, $\tau_M(s):=(s\wedge M)\lor(-M)$ denotes the standard truncation operator at level $M$.

\begin{proposition}\label{prop:domain direct}
Let $(u,v)\in L^1(\Omega;\mathbb R^{N+1})$ be such that $\displaystyle\big(\Gamma\textup{-}\liminf_{\eps\to 0} \mathcal{F}_\eps\big)(u,v;\Omega)<\infty$. Under the assumptions of Theorem~\ref{thm:general-anisotropic-gamma-convergence}, it necessarily follows that $v=1$ $\calL^n$-a.e. in $\Omega$, and $u\in GBV_{\star}(\Omega,\mathbb R^N)$ with $\mathcal{F}_{\varsigma}(u,1;\Omega)<\infty$.
\end{proposition}

\begin{proof}
Consider a sequence $(u_\eps,v_\eps)\to(u,v)$ in $L^1(\Omega,\mathbb R^{N+1})$ verifying 
\begin{equation}\label{e:prima equazione nella prop sul dominio del gamma limite}
\sup_\eps\callF_\eps(u_\eps,v_\eps;\Omega)<\infty.
\end{equation}
Due to the strict penalization induced by the term $\frac{\omega(1-v_\eps)}{\eps}$, it follows that $v=1$ $\calL^n$-a.e. in $\Omega$.  
Fix $\delta\in(0,1)$. Using standard truncation properties for phase-field approximations and exploiting the coercive growth conditions on $h^{\qc}_{\varsigma}$ (cf. \eqref{e:bulk-growth-h}), we obtain the uniform bound
\begin{align*}
\int_\Omega|\nabla {u}_\eps^\delta|\dx+\calH^{n-1}(J_{{u}_\eps^\delta})
\leq c\big(\callF_\eps(u_\eps,v_\eps;\Omega)+1\big),
\end{align*}
where $u_\eps^\delta$ represents a properly thresholded version of $u_\eps$, and $c>0$ is a constant depending on $\delta$ and the measure of $\Omega$. In particular, each component $({u}_\eps^\delta)_i$ belongs to $GSBV(\Omega)$ and satisfies
\begin{align*}
\int_\Omega|\nabla ({u}_\eps^\delta)_i|\dx+\calH^{n-1}(J_{({u}_\eps^\delta)_i})
\leq c\big(\callF_\eps(u_\eps,v_\eps;\Omega)+1\big).
\end{align*}
For any $M>0$, the truncated function $\tau_M(({u}_\eps^\delta)_i)$ is in $BV_{\textup{loc}}(\Omega)$ for every $\eps>0$. The previous estimate enforces that 
\[
|D(\tau_M(({u}_\eps^\delta)_i))|(\Omega)\leq C_M,
\]
for some constant $C_M>0$ dependent on $M$ but strictly independent of $\eps$. Hence, up to a subsequence, $\tau_M(({u}_\eps^\delta)_i)$ converges weakly in $BV(\Omega)$. Since $v_\eps \to 1$, we also have ${u}_\eps^\delta \to u$ strongly in $L^1(\Omega,\mathbb R^N)$, implying that $\tau_M(u_i) \in BV(\Omega)$ for every $M>0$. Consequently, $u_i\in GBV(\Omega)$ for each $i$, ensuring that $u\in (GBV(\Omega))^N$.

Furthermore, letting $\tilde \tau_M(w):=(\tau_M(w_1),\dots,\tau_M(w_N))$ for every $w\in \mathbb R^N$, there exists a constant $L\in(0,\infty)$ such that 
\begin{equation*}
\sup_{M>0} \sup_{\eps>0} \callF_{\eps}(\tilde \tau_M(u_\eps),v_\eps;\Omega)\leq L.
\end{equation*}
Because $\tilde \tau_M(u_{\eps})\to \tilde \tau_M(u)$ in $L^1(\Omega,\mathbb R^N)$, we can leverage Theorem~\ref{thm:lbcompBV}, along with the growth properties \eqref{e:bulk-growth-h} and \eqref{e:crescita della g vettoriale}, to deduce that for all $i$ and $M>0$:
\begin{equation*}
\int_\Omega |\nabla \tau_M(u_i)| \dx 
+ |D^c(\tau_M(u_i))|(\Omega)
+ \int_{J_{\tau_M(u_i)}} (|[\tau_M(u_i)]|\wedge1) \diff \calH^{n-1}\leq L.
\end{equation*}
This explicitly classifies $u\in GBV_{\star}(\Omega,\mathbb R^N)$. By exploiting the fundamental structural properties of $GBV$ functions (see, for instance, \cite{Bra98}) and applying the Monotone Convergence Theorem, we infer that
\begin{equation}\label{e:funz finito su GBV alla N}
\int_\Omega |\nabla u_i| \dx 
+ |D^cu_i|(\Omega)
+  \int_{J_{u_i}} (|[u_i]|\wedge1) \diff \calH^{n-1}<\infty.
\end{equation}
Coupling this integrability with the growth constraints of $h_{\varsigma}^{\qc}$, $h_{\varsigma}^{\qc,\infty}$, and $g$ (cf. \eqref{e:bulk-growth-h} and \eqref{e:crescita della g vettoriale}), we finally conclude that $\mathcal{F}_{\varsigma}(u,1;\Omega)<\infty$.
\end{proof}

We are now in the position to extend the validity of the lower bound (Theorem~\ref{thm:lbcompBV}) to the whole space $(GBV(\Omega))^N$. We first verify that the limit functional $\mathcal{F}_{\varsigma}$ is continuous under compositions with truncations.

\begin{proposition}\label{prop:contF0}
 Under the assumptions of Theorem~\ref{thm:general-anisotropic-gamma-convergence}, for all $u\in (GBV(\Omega))^N$ with $\mathcal{F}_{\varsigma}(u,1;\Omega)<\infty$ (i.e., $u\in GBV_{\star}(\Omega)^N$), we have  
\[
\lim_{k\to\infty}\mathcal{F}_{\varsigma}(\mathcal{T}_k(u),1;\Omega)=\mathcal{F}_{\varsigma}(u,1;\Omega)\,,
\]
\end{proposition}
where $\mathcal{T}_k$ is defined in \eqref{e:Tk}.
\begin{proof}
We assess the convergence of the volume, Cantor, and surface components separately.  

For the bulk part, \eqref{e:bulk-growth-h} ensures that $|\nabla u|\in L^1(\Omega)$. Since $\nabla(\mathcal{T}_k(u))=\nabla u$ for $\calL^n$-a.e. $x\in \Omega_k:=\{|u|\le a_k\}$, applying the linear growth bound yields
\[
\Big|\int_\Omega h_{\varsigma}^\qc\big(\nabla(\mathcal{T}_k(u))\big)\dx-
\int_\Omega h_{\varsigma}^\qc(\nabla u)\dx\Big|
\leq K_{\varsigma}\int_{\Omega\setminus \Omega_k}(1+|\nabla u|)\dx.
\]
As $a_k\to\infty$, the domain $\Omega\setminus \Omega_k$ vanishes in measure, producing
\[
\lim_{k\to\infty} \int_\Omega h_{\varsigma}^\qc\big(\nabla(\mathcal{T}_k(u))\big)\dx
=\int_\Omega h_{\varsigma}^\qc(\nabla u)\dx.
\]

Concerning the surface contribution, we note that $J_{\mathcal{T}_k(u)}\subseteq J_u$ for every $k$, and $\nu_{\mathcal{T}_k(u)}=\nu_u$ for $\mathcal H^{n-1}$-a.e. $x\in J_{\mathcal{T}_k(u)}$. Because $\mathcal{F}_{\varsigma}(u,1;\Omega)<\infty$, condition \eqref{e:funz finito su GBV alla N} is satisfied. Consequently, the left and right traces verify $(\mathcal{T}_k(u))^\pm\to u^\pm$ and $\chi_{J_{\mathcal{T}_k(u)}}\to \chi_{J_u}$ as $k\to\infty$, with the uniform domination $|[\mathcal{T}_k(u)]|\le |[u]|$ on $J_u$. Thus,
\begin{align*}
\lim_{k\to\infty} \int_{J_{\mathcal{T}_k(u)}}g([\mathcal{T}_k(u)],\nu_{\mathcal{T}_k(u)})\diff \calH^{n-1}
&=\lim_{k\to\infty} \int_{J_u}g([\mathcal{T}_k(u)],\nu_u)\chi_{J_{\mathcal{T}_k(u)}}\diff \calH^{n-1}=\int_{J_u}g([u],\nu_u)\diff \calH^{n-1},
\end{align*}
relying on the continuity of $g$ (Lemma~\ref{lemmapropg2}) and the Dominated Convergence Theorem.

For the Cantor part, the chain rule in $GBV$ guarantees
\[
D^c(\mathcal{T}_k(u))\res \Omega_k=D^c u\res \Omega_k,
\qquad
|D^c(\mathcal{T}_k(u))|\ll |D^c u|,
\qquad
\frac{\diff|D^c(\mathcal{T}_k(u))|}{\diff|D^c u|}\le 1.
\]
Then, exploiting the recession bound \eqref{e:bulk-growth-hrecession}, we obtain
\begin{equation*}
\Big|\int_{\Omega}h_{\varsigma}^{\qc,\infty}(\diff D^c(\mathcal{T}_k(u)))-
\int_{\Omega}h_{\varsigma}^{\qc,\infty}(\diff D^cu)\Big|
\leq K_{\varsigma}\int_{\Omega\setminus \Omega_k}\diff |D^cu|
= K_{\varsigma}|D^cu|(\Omega\setminus \Omega_k).
\end{equation*}
Taking the limit as $k\to\infty$ nullifies the right-hand side, yielding
\[
\lim_{k\to\infty}\int_{\Omega}h^{\qc,\infty}_{\varsigma}(\diff D^c(\mathcal{T}_k(u)))=
\int_{\Omega}h^{\qc,\infty}_{\varsigma}(\diff D^cu),
\]
which completes the proof.
\end{proof}

We conclude this section by establishing the ultimate lower bound theorem for generalized functions of bounded variation.

\begin{theorem}\label{thm:lbcomp}
Let $u\in (GBV(\Omega))^N$. Under the assumptions of Theorem~\ref{thm:general-anisotropic-gamma-convergence}, it holds that
\begin{equation}\label{e:lbcomp}
\mathcal{F}_{\varsigma}(u,1;\Omega) \le \Gamma(L^1)\hbox{-}\liminf_{\eps\to0} \Fm_\eps(u, 1;\Omega).
\end{equation}
\end{theorem}

\begin{proof}
Let $(u_\eps,v_\eps)\in H^1(\Omega, \mathbb R^{N}\times [0,1])$ be a sequence converging to $(u,1)$ in $L^1(\Omega,\mathbb R^{N+1})$. We may seamlessly assume that $\liminf_{\eps\to0}\callF_{\eps}(u_\eps,v_\eps;\Omega)<\infty$, that this limit is attained along a subsequence, and that $(u_\eps,v_\eps)\to(u,1)$ pointwise $\calL^n$-a.e. in $\Omega$. By Proposition~\ref{prop:domain direct}, this guarantees that $u\in(GBV(\Omega))^N$ and $\mathcal{F}_{\varsigma}(u,1;\Omega)<\infty$. 

Recalling the truncation operator $\mathcal{T}_k$, we recognize that $\mathcal{T}_k(u_\eps)\to \mathcal{T}_k(u)$ strongly in $L^1(\Omega,\mathbb R^N)$ for any fixed $k$, and that $\mathcal{T}_k(u)\in BV(\Omega,\mathbb R^N)$. Thus, invoking Theorem~\ref{thm:lbcompBV}, we possess the bound
\begin{equation}\label{estTkA}
\mathcal{F}_{\varsigma}(\mathcal{T}_{k_{M}}(u),1;\Omega)\leq \liminf_{\eps\to0}\callF_{\eps}(\mathcal{T}_{k_{M}}(u_\eps),v_\eps;\Omega),
\end{equation}
for every integer choice $k_M$. 

We claim that for each $M\in\mathbb N$, there exists an optimal index $k_{M}\in\{M+1,\dots,2M\}$ (independent of $\eps$) such that, upon passing to a suitable subsequence,
\begin{equation}\label{estTk}
\callF_{\eps}(\mathcal{T}_{k_{M}}(u_\eps),v_\eps;\Omega)\leq 
\Big(1+\frac{c}{M}\Big)\callF_{\eps}(u_\eps,v_\eps;\Omega)
+c\,\calL^n(\{|u_\eps|>a_M\}),
\end{equation}
for some uniform constant $c>0$. Accepting this claim momentarily, substituting \eqref{estTk} into \eqref{estTkA} and noting that the measure of the super-level sets decays to zero, we infer that
\[
\limsup_{M\to\infty}\mathcal{F}_{\varsigma}(\mathcal{T}_{k_M}(u),1;\Omega)
\leq\liminf_{\eps\to0}\callF_{\eps}(u_\eps,v_\eps;\Omega).
\]
Applying the truncation continuity of $\mathcal{F}_{\varsigma}$ proven in Proposition~\ref{prop:contF0}, we effortlessly recover the general lower bound \eqref{e:lbcomp}. 

It remains to justify \eqref{estTk}. We execute De Giorgi's averaging-slicing method across the codomain levels. For each $k\in\mathbb N$, the energy partitions additively:
\begin{multline}\label{e:Feps* Tk2}
\callF_{\eps}(\mathcal{T}_k(u_\eps),v_\eps;\Omega)
=\callF_{\eps}(u_\eps,v_\eps;\{|u_\eps|\leq a_k\})\\
+\callF_{\eps}(\mathcal{T}_k(u_\eps),v_\eps;\{a_k<|u_\eps|< a_{k+1}\})
+\callF_{\eps}(0,v_\eps;\{|u_\eps|\geq a_{k+1}\}).
\end{multline}
Due to the structural growth limits \eqref{eq:growth-Phi} and the non-expansiveness of $\mathcal{T}_k$, the transition layer energy is bounded by
\begin{align}\label{e:Feps* Tk 22}
\callF_{\eps}(\mathcal{T}_k(u_\eps),v_\eps;\{a_k<|u_\eps|<a_{k+1}\})
&\leq c\int_{\{a_k<|u_\eps|< a_{k+1}\}}\varphi(\eps f^2(v_\eps))\Phi(\nabla u_\eps)\dx\notag\\
&\quad +c\calL^n(\{a_k<|u_\eps|<a_{k+1}\})
+\callF_{\eps}(0,v_\eps;\{a_k<|u_\eps|< a_{k+1}\}).
\end{align}
Summing \eqref{e:Feps* Tk2} and \eqref{e:Feps* Tk 22}, and subsequently averaging over $k \in \{M+1, \dots, 2M\}$, ensures the existence of a highly favorable index $k_{M,\eps}$ fulfilling
\begin{align*}
\callF_{\eps}(\mathcal{T}_{k_{M,\eps}}(u_\eps),v_\eps;\Omega)
&\leq \frac{1}{M}\sum_{k=M+1}^{2M}\callF_{\eps}(\mathcal{T}_k(u_\eps),v_\eps;\Omega)\leq \Big(1+\frac{c}{M}\Big)\callF_\eps(u_\eps,v_\eps;\Omega)
+c\,\calL^n(\{|u_\eps|>a_M\}).\end{align*}
As $\eps\to0$, the pigeonhole principle allows us to extract a subsequence for which $k_{M,\eps} \equiv k_M$ locks into a constant value independent of $\eps$, verifying \eqref{estTk} and concluding the proof.
\end{proof}

\section{Upper bound}\label{s:upperbound}

In this section, we establish the $\Gamma$-$\limsup$ inequality stated in Theorem~\ref{thm:general-anisotropic-gamma-convergence}. Following the methodology developed in \cite{CFI24}, we introduce a coercive perturbation of the functional $\mathcal{F}_{\eps}$. We denote this augmented functional by $\widehat{\mathcal{F}}_{\eps}:L^1(\Omega,\mathbb R^{N+1})\times \mathcal{A}(\Omega)\to[0,\infty]$, defined as
\begin{equation}\label{eq:defcalfepseta}
\widehat{\mathcal{F}}_{\eps}(u,v;A):=
\begin{cases}
\Fm_\eps(u,v;A)+\eta_\eps \displaystyle\int_A \Phi(\nabla u)\dx, & \text{if } u\in H^1(\Omega,\mathbb R^N),\\[0.5em]
\infty, & \text{otherwise},
\end{cases}
\end{equation}
where $\eta_{\eps}>0$ is chosen such that $\eta_{\eps}/\eps\to 0$ as $\eps\to 0$.  

The supplementary regularization term $\eta_\eps\int \Phi(\nabla u)\dx$ ensures the strict coercivity of $\widehat{\mathcal{F}}_{\eps}$, thereby facilitating the derivation of the $\Gamma$-limit. Specifically, we will prove the upper bound directly for this modified functional by relying on the integral representation formula from \cite{bouchitte1998global}.  

As a preliminary step, we demonstrate that the set function $\overline{\Gamma}\text{-}\displaystyle\lim_{\eps\to0}\widehat{\mathcal{F}}_{\eps}(u,1;\cdot)$ identifies a Borel measure by proving the weak subadditivity of $\Gamma\text{-}\displaystyle\limsup_{\eps\to 0}\widehat{\mathcal{F}}_\eps$.

\begin{lemma}\label{l:wsub}
Let $u\in L^1(\Omega,\R^N)$, let $A',A,B\in\mathcal{A}(\Omega)$
with $A'\subset\subset A$, then under the assumptions of
Theorem~\ref{thm:general-anisotropic-gamma-convergence},
\begin{equation}\label{e:subadd}
\widehat{\mathcal{F}}''(u,1;A'\cup B)
\leq \widehat{\mathcal{F}}''(u,1;A)
+\widehat{\mathcal{F}}''(u,1;B),
\end{equation}
where $\widehat{\mathcal{F}}''(u,1;\cdot)
:=\Gamma$-$\displaystyle\limsup_{\eps\to0}
\widehat{\mathcal{F}}_{\eps}(u,1;\cdot)$.
\end{lemma}

\begin{proof}
It is not restrictive to assume that the right-hand side of
\eqref{e:subadd} is finite.
Let $(u_\eps^J,v_\eps^J)\in
H^1(\Omega,\R^N)\times H^1(\Omega,[0,1])$, for $J\in\{A,B\}$,
be such that
\begin{equation}\label{e:rec1}
(u_\eps^J,v_\eps^J)\to(u,1)
\quad\text{in }L^1(\Omega,\R^{N+1}),
\end{equation}
and
\begin{equation}\label{e:rec2}
\limsup_{\eps\to0}
\widehat{\mathcal{F}}_\eps(u_\eps^J,v_\eps^J;J)
=\widehat{\mathcal{F}}''(u,1;J).
\end{equation}

\medskip
\noindent
{\bf Step 1. A gluing estimate.}
Let
\[
\delta:=\dist(A',\partial A)>0
\]
and fix $M\in\N$ with $M\ge4$. For $i=1,\ldots,M$ we set
\[
A_i
:=
\left\{
x\in A:
\dist(x,A')<\frac{i\delta}{M}
\right\},
\qquad
A_0:=A',
\]
then
\[
A'=A_0\subset\subset A_1\subset\subset\cdots
\subset\subset A_M\subseteq A.
\]
For every $i=1,\ldots,M$, let $\chi_i\in C_c^1(A_i;[0,1])$
be such that
\begin{equation}\label{e:cutoff Ai weak sub}
\chi_i=1
\quad\text{on }A_{i-1},
\qquad
|\nabla\chi_i|
\leq\frac{2M}{\delta}.
\end{equation}
Setting $w_\eps:=v_\eps^A\wedge v_\eps^B$, we have
\begin{equation}\label{e:grad min weak sub}
\nabla w_\eps
=
\nabla v_\eps^A\,
\mathbf 1_{\{v_\eps^A\leq v_\eps^B\}}
+
\nabla v_\eps^B\,
\mathbf 1_{\{v_\eps^B<v_\eps^A\}}
\qquad\text{a.e. in }\Omega.
\end{equation}
Fix now $i\in\{2,\ldots,M-1\}$ and define $u_\eps^i$ on $\Omega$ by
\begin{equation}\label{e:uki}
u_\eps^i
:=
\chi_i u_\eps^A+(1-\chi_i)u_\eps^B
\in H^1(\Omega,\R^N),
\end{equation}
and $v_\eps^i\in H^1(\Omega,[0,1])$ by
\begin{equation}\label{e:vki}
v_\eps^i
:=
\begin{cases}
\chi_{i-1}v_\eps^A+(1-\chi_{i-1})w_\eps,
&\text{in }A_{i-1},
\\
\chi_{i+1}w_\eps+(1-\chi_{i+1})v_\eps^B,
&\text{in }\Omega\setminus\operatorname{supp}\chi_{i-1}.
\end{cases}
\end{equation}
For simplicity, let us introduce the auxiliary functional
$\mathcal{G}_\eps:L^1(\Omega)\times\mathcal{A}(\Omega)
\to[0,\infty]$ defined by
\begin{equation*}
\mathcal{G}_\eps(v;A):=
\begin{cases}
\displaystyle\int_A
\frac{\PotDann(1-v)}{4\eps}+\eps\gamma^2(\nabla v)\,\diff x,
&\text{if }v\in H^1(\Omega,[0,1]),
\\
\infty,&\text{otherwise}.
\end{cases}
\end{equation*}
From \eqref{e:uki} and \eqref{e:vki}, we have that
\begin{align}\label{e:stima iniziale limsup}
\widehat{\mathcal{F}}_{\eps}
(u^i_{\eps},v^i_{\eps};A'\cup B)
&\leq
\widehat{\mathcal{F}}_{\eps}
(u^A_{\eps},v^A_{\eps};A_{i-2})
+\widehat{\mathcal{F}}_{\eps}
(u^B_{\eps},v^B_{\eps};B\setminus A_{i+1})+
\widehat{\mathcal{F}}_{\eps}
(u^i_{\eps},v^i_{\eps};B\cap(A_{i+1}\setminus A_{i-2})).
\end{align}

We now estimate just the last term on the right-hand side of
\eqref{e:stima iniziale limsup} by splitting it into three contributions.
Starting with
$\widehat{\mathcal{F}}_{\eps}
(u_\eps^i,v_\eps^i;B\cap(A_{i-1}\setminus A_{i-2}))$,
and using \eqref{eq:growth-Phi},
\eqref{e:gamma lower euclidean suppenergy} together with
Lemma~\ref{l:quasi monotonia funzione di degradamento}, we obtain
\begin{align}\label{e:i-1i-2}
&\widehat{\mathcal{F}}_{\eps}
(u_\eps^i,v_\eps^i;B\cap(A_{i-1}\setminus A_{i-2}))
\notag\\
&\le
c\int_{B\cap(A_{i-1}\setminus A_{i-2})}
\varphi(\eps f^2(v_\eps^A))|\nabla u_\eps^A|^2\,\dx
+\mathcal{G}_\eps(v_\eps^i;B\cap(A_{i-1}\setminus A_{i-2}))
\notag\\
&\quad+
c(\varphi(\infty)+\eta_\eps)
\mathcal{L}^n(B\cap(A_{i-1}\setminus A_{i-2}))
+\frac{cM^2\eps}{\delta^2}
\int_{B\cap(A_{i-1}\setminus A_{i-2})}
|v_\eps^A-v_\eps^B|^2\,\dx
\notag\\
&\quad+
\eta_\eps\int_{B\cap(A_{i-1}\setminus A_{i-2})}
\Phi(\nabla u^A_\eps)\,\dx.
\end{align}
Similarly, for the intermediate region we get
\begin{align}\label{e:ii-1}
&\widehat{\mathcal{F}}_{\eps}
(u_\eps^i,v_\eps^i;B\cap(A_i\setminus A_{i-1}))
\notag\\
&\le
c\int_{B\cap(A_i\setminus A_{i-1})}
\varphi(\eps f^2(v_\eps^A\wedge v_\eps^B))
\Big(
|\nabla u_\eps^A|^2+|\nabla u_\eps^B|^2
+\frac{M^2}{\delta^2}|u_\eps^A-u_\eps^B|^2
\Big)\,\dx
\notag\\
&\quad+
\mathcal{G}_{\eps}
(v_\eps^A\wedge v_\eps^B;B\cap(A_i\setminus A_{i-1}))
+c(\varphi(\infty)+\eta_\eps)
\calL^n(B\cap(A_i\setminus A_{i-1}))
\notag\\
&\quad+
c\eta_\eps\int_{B\cap(A_i\setminus A_{i-1})}
\big(\Phi(\nabla u^A_\eps)+\Phi(\nabla u^B_\eps)\big)\,\dx
+\frac{c\eta_\eps M^2}{\delta^2}
\int_{B\cap(A_i\setminus A_{i-1})}
|u_\eps^A-u_\eps^B|^2\,\dx,
\end{align}
and for the third region:
\begin{align}\label{e:i+1i}
&\widehat{\mathcal{F}}_{\eps}
(u_\eps^i,v_\eps^i;B\cap(A_{i+1}\setminus A_i))
\notag\\
&\le
c\int_{B\cap(A_{i+1}\setminus A_i)}
\varphi(\eps f^2(v_\eps^B))|\nabla u_\eps^B|^2\,\dx
+\mathcal{G}_{\eps}(v_\eps^i;B\cap(A_{i+1}\setminus A_i))
\notag\\
&\quad+
c(\varphi(\infty)+\eta_\eps)
\mathcal{L}^n(B\cap(A_{i+1}\setminus A_i))
+\frac{cM^2\eps}{\delta^2}
\int_{B\cap(A_{i+1}\setminus A_i)}
|v_\eps^A-v_\eps^B|^2\,\dx
\notag\\
&\quad+
\eta_\eps\int_{B\cap(A_{i+1}\setminus A_i)}
\Phi(\nabla u^B_\eps)\,\dx.
\end{align}
Moreover, setting $D_i:=B\cap(A_{i+1}\setminus A_{i-2})$,
the monotonicity of $\PotDann$ and the triangle inequality for
$\gamma$ give
\[
\mathcal{G}_\eps(v_\eps^i;D_i)
\le
c\big(\mathcal{G}_\eps(v_\eps^A;D_i)
+\mathcal{G}_\eps(v_\eps^B;D_i)\big)
+\frac{cM^2\eps}{\delta^2}
\int_{D_i}|v_\eps^A-v_\eps^B|^2\,\dx.
\]
Since $\eta_\eps\to0$ as $\eps\to0$, summing up
\eqref{e:stima iniziale limsup}--\eqref{e:i+1i}, using the preceding
estimate, and averaging over $i=2,\dots,M-1$, we find an index
$i_\eps\in\{2,\ldots,M-1\}$ such that
\begin{align*}
&\widehat{\mathcal{F}}_{\eps}
(u_\eps^{i_\eps},v_\eps^{i_\eps};A'\cup B)
\\
&\le
\left(1+\frac{c}{M-2}\right)
\left(
\widehat{\mathcal{F}}_{\eps}(u_\eps^A,v_\eps^A;A)
+\widehat{\mathcal{F}}_{\eps}(u_\eps^B,v_\eps^B;B)
\right)
+\frac{c}{M-2}\calL^n(B\cap(A\setminus A'))
\\
&\quad+
\frac{cM^2}{\delta^2(M-2)}
\int_{B\cap(A\setminus A')}
\big(|u_\eps^A-u_\eps^B|^2
+\eps|v_\eps^A-v_\eps^B|^2\big)\,\dx.
\end{align*}
We apply this estimate below to suitably truncated recovery sequences.

\medskip
\noindent
{\bf Step 2. Truncation and diagonal argument.}

Using the same truncation argument as in Theorem~\ref{thm:lbcomp},
for every $M\ge4$ and $\eps>0$ sufficiently small there exists
$k_{\eps,M}\in\{M+1,\dots,2M\}$ such that
\begin{align}\label{e:truncation Psik}
&\widehat{\callF}_\eps
(\mathcal{T}_{k_{\eps,M}}(u_\eps^A),v_\eps^A;A)
+\widehat{\callF}_\eps
(\mathcal{T}_{k_{\eps,M}}(u_\eps^B),v_\eps^B;B)
\notag\\
&\le
\Big(1+\frac{c}{M}\Big)
\Big(
\widehat{\callF}_\eps(u_\eps^A,v_\eps^A;A)
+\widehat{\callF}_\eps(u_\eps^B,v_\eps^B;B)
\Big)
\notag\\
&\quad+
c\calL^n\big(
\{|u_\eps^A|\ge a_{M+1}\}
\cup\{|u_\eps^B|\ge a_{M+1}\}
\big),
\end{align}
for some constant $c$ independent of $\eps$ and $M$.
Set
\[
U_\eps^{J,M}:=\mathcal{T}_{k_{\eps,M}}(u_\eps^J),
\qquad J\in\{A,B\}.
\]
Since the same truncation is used for both families,
$\operatorname{Lip}(\mathcal{T}_k)\le1$, and
$\|\mathcal{T}_k\|_{L^\infty(\R^N)}\le a_{k+1}$, we have
\[
\|U_\eps^{A,M}-U_\eps^{B,M}\|_{L^2(\Omega)}^2
\le
2a_{2M+1}\|u_\eps^A-u_\eps^B\|_{L^1(\Omega)}
\longrightarrow0
\qquad\text{for every fixed }M.
\]
Apply the construction and the averaged estimate of Step~1, with
$M$ layers, to $(U_\eps^{A,M},v_\eps^A)$ and
$(U_\eps^{B,M},v_\eps^B)$, and denote the resulting pair by
$(\bar u_\eps^M,\bar v_\eps^M)$.
By \eqref{e:rec1},
\[
\limsup_{\eps\to0}
\calL^n\big(
\{|u_\eps^A|\ge a_{M+1}\}
\cup\{|u_\eps^B|\ge a_{M+1}\}
\big)
\le \calL^n(\{|u|\ge a_{M+1}\}).
\]
Thus, using \eqref{e:rec2} and \eqref{e:truncation Psik}, and
absorbing constants, we obtain
\begin{align}\label{e:truncation Psik 2}
\limsup_{\eps\to0}
\widehat{\callF}_\eps(\bar u_\eps^M,\bar v_\eps^M;A'\cup B)
&\le
\Big(1+\frac{c}{M}\Big)
\Big(
\widehat{\callF}''(u,1;A)
+\widehat{\callF}''(u,1;B)
\Big)
\notag\\
&\quad+
\frac{c}{M}\calL^n(B\cap(A\setminus A'))
+c\calL^n(\{|u|\ge a_{M+1}\}).
\end{align}
Moreover, by the Lipschitz property of the truncations and the
definition of the glued displacement,
\[
\|\bar u_\eps^M-u\|_{L^1(\Omega)}
\le
\|u_\eps^A-u\|_{L^1(\Omega)}
+\|u_\eps^B-u\|_{L^1(\Omega)}
+2\int_{\{|u|>a_{M+1}\}}|u|\,\dx.
\]
The phase-field construction also gives
\[
\|\bar v_\eps^M-1\|_{L^1(\Omega)}
\le
\|v_\eps^A-1\|_{L^1(\Omega)}
+\|v_\eps^B-1\|_{L^1(\Omega)}.
\]
Since $u\in L^1(\Omega,\R^N)$ and $a_{M+1}\to\infty$,
the tail terms vanish as $M\to\infty$.
A diagonal choice $M=M(\eps)\to\infty$ sufficiently slowly yields $(\bar u_\eps^{M(\eps)},\bar v_\eps^{M(\eps)})
\to(u,1)$ in $L^1(\Omega,\R^{N+1})$ and
\[
\limsup_{\eps\to0}
\widehat{\callF}_\eps
(\bar u_\eps^{M(\eps)},\bar v_\eps^{M(\eps)};A'\cup B)
\le
\widehat{\callF}''(u,1;A)
+\widehat{\callF}''(u,1;B).
\]
This proves \eqref{e:subadd} by the definition of
$\widehat{\callF}''$.
\end{proof}
We are now ready to prove the upper bound inequality.
\begin{theorem}\label{t:limsupndim}
Under the assumptions of Theorem~\ref{thm:general-anisotropic-gamma-convergence}, we have that for every $(u,v)\in L^1(\Omega,\R^{N+1})$ it holds
	\begin{equation}\label{e:gammalisup estimate}
		{\Gamma\textup{-}\limsup_{{\eps\to0}}\widehat{\Fm}_{\eps}(u,v;\Omega)}\leq \Fm_\varsigma(u,v;\Omega),
	\end{equation}
    where $\Fm_{\varsigma}$ is defined in \eqref{eq:def-limit-functional}.
\end{theorem}
\begin{proof}
	Given a subsequence $(\widehat{\callF}_{\eps_k})$ of $(\widehat{\callF}_\eps)$, we can extract (without relabeling) a further subsequence that $\overline{\Gamma}$–converges to some functional $\overline{\callF}$. In particular,
\begin{equation}\label{e:additivity}
\overline{\callF}=(\callF')_-=(\callF'')_-,
\end{equation}
where $\callF'$ and $\callF''$ denote the $\Gamma(L^1)$–lower and upper limits of $\widehat{\callF}_{\eps_k}$, and the subscript $_-$ stands for the inner regular envelope of the corresponding functional (see \cite[Definition~16.2 and Theorem~16.9]{dalmaso1993introduction}).

We observe that $\overline{\callF}(u,v;\cdot)$ is the restriction to open sets of a Borel measure by \cite[Theorem~14.23]{dalmaso1993introduction}. Indeed, by construction it is increasing and inner regular; additivity follows from \eqref{e:additivity} once one verifies that $(\callF')_-$ is superadditive and $(\callF'')_-$ is subadditive. The former is a direct consequence of the additivity of $\widehat{\callF}_\eps(u,v;\cdot)$ together with \cite[Proposition~16.12]{dalmaso1993introduction}. The latter is obtained from Lemma~\ref{l:wsub}, along the same lines as \cite[Proposition~18.4]{dalmaso1993introduction}, replacing \cite[(18.6)]{dalmaso1993introduction} with Lemma~\ref{l:wsub}.

We split the proof of \eqref{e:gammalisup estimate} into several steps. Note first that it suffices to treat the case $v=1$ $\calL^n$–a.e. in $\Omega$.

{\bf Step 1. Diffuse estimate for $u\in BV(\Omega,\R^N)$.}
We first derive a global upper bound for $\callF''$ which is sharp on the diffuse part when $u\in BV(\Omega,\R^N)$. Define $H:L^1(\Omega,\R^N)\times\mathcal A(\Omega)\to[0,\infty]$ by
\begin{equation}\label{e:H2}
 H(u;A):=\int_A h_{\varsigma}(\nabla u)\,\dx
\end{equation}
for $u\in W^{1,1}(\Omega;\R^N)$, and $H(u;A)=\infty$ otherwise, where $h_{\varsigma}$ is given by \eqref{eq:def-h-varsigma}. We claim that
\begin{equation}\label{linearest}
\callF''(u,1;A)\leq H(u;A)
\end{equation}
for $u\in W^{1,1}(\Omega;\R^N)$ and $A\in\mathcal{A}(\Omega)$. Assuming \eqref{linearest}, set $H^{*}:L^1(\Omega;\R^N)\times\mathcal A(\Omega)\to[0,\infty]$ as
\begin{equation}\label{e:Hqc}
 H^{*}(u;A):=\int_A h_{\varsigma}^{\qc}(\nabla u)\,\dx
\end{equation}
for $u\in W^{1,1}(\Omega;\R^N)$, and $H^{*}(u;A)=\infty$ otherwise. By the $L^1$–lower semicontinuity of $\callF''$ and the relaxation theorem in \cite[Statement~III.7]{AcerbiFusco1984semicontinuity} (whose assumptions follow from Proposition~\ref{prop:bulk-growth-section2}), we deduce $\callF''(u,1;A)\leq H^{*}(u;A).$ Then \cite[Theorem~4.1]{AmbrosioDalmaso1992} yields, for every $u\in BV(\Omega;\R^N)$,
\begin{equation}\label{e:lsc envelope H}
\mathcal H(u;A):=\mathrm{sc}^-(L^1)\hbox{-}H^{*}(u;A)
=\int_A h_{\varsigma}^{\qc}(\nabla u)\,\dx+\int_A h_{\varsigma}^\qcinfty\big(\diff D^su\big)\,,
\end{equation}
and hence
\begin{equation}\label{e:stima limsup diffusa}
\callF''(u,1;A)\leq \mathcal H(u;A)
\end{equation}
for all $u\in BV(\Omega,\R^N)$ and $A\in\mathcal A(\Omega)$.

To verify \eqref{linearest}, consider first an affine map $u(x)=\xi x+b$ with $\xi\in\R^{N\times n}$ and $b\in\R^N$. By \eqref{eq:def-h-varsigma} there exists $\tau_{\xi}\in[0,\infty)$ such that
\begin{equation}\label{e:optimal tauxi vettoriale}
    h_{\varsigma}(\xi)=\varphi\Big(\frac{1}{\tau_{\xi}}\Big)\Phi(\xi)+\frac{\varsigma^2}{4}\tau_\xi,
\end{equation}
with the convention that $\varphi(\frac1\tau)$ is extended continuously at $\tau=0$ by $\varphi(\infty)\in(0,\infty)$. Define $(u_k,v_k):=(u,\lambda_k)$ with
\begin{equation}\label{e:def di tau kappa}
\eps_kf^2(\lambda_k)=\eps_k\frac{\ell(\lambda_k)}{Q(1-\lambda_k)}=\frac{1}{\tau_k},
\end{equation}
where $\tau_k\to \tau_\xi$ as $k\to \infty$. Then {$\lambda_k\to 1$ as $k\to\infty$, thus} $(u_k,v_k)\to(u,1)$ in $L^1(\Omega;\R^{N+1})$, and
\begin{align}\label{e:gamma limsup vettoriale affine functions}
  \widehat{\Fm}_{\eps_k}(u_k,v_k,A)
 &= \Big[\varphi\Big(\frac{1}{\tau_k}\Big)\Phi(\xi)+\frac{\PotDann(1-\lambda_k)}{4\eps_k}+\eta_{\eps_k} \Phi(\xi)\Big]\calL^n(A) \nonumber\\
 & = \Big[\varphi\Big(\frac{1}{\tau_k}\Big)\Phi(\xi)+\frac{\PotDann(1-\lambda_k)}{4Q(1-\lambda_k)}{\ell(\lambda_k)}\tau_k+\eta_{\eps_k}\Phi(\xi)\Big]\calL^n(A).
\end{align}
Using \eqref{eq:def-varsigma}, $\tau_k\to \tau_\xi$ and $\eta_\eps=o(\eps)$, we get
\[
\limsup_{k\to\infty}\widehat{\callF}_{\eps_k}(u_{k},v_{k};A)
{=} h_{\varsigma}(\xi)\calL^n(A)=H(u,A),
\]
which proves \eqref{linearest} for affine $u$ by \eqref{e:optimal tauxi vettoriale}–\eqref{e:gamma limsup vettoriale affine functions}.

Next, let $u\in C^0(\Omega,\R^N)$ be piecewise affine, say $u(x)=\sum_{i=1}^M(\xi_i x+b_i)\mathbf 1_{\Omega_i}(x)$ with pairwise disjoint $\Omega_i\in\mathcal{A}(\Omega)$ of Lipschitz boundary and $\calL^n(\Omega\setminus\cup_i \Omega_i)=0$. Set $u_k:=u$ and $v_k:=\sum_{i=1}^M \lambda_k^i\chi_i$, where for each $i$ and $k$, $\eps_k f^2(\lambda_{k}^i)= \frac{1}{\tau^i_{k}}$ as in \eqref{e:def di tau kappa}, and $\{\chi_i\}$ is a partition of unity subordinate to the cover $\{\Omega_i^\delta\}$: $\sum_{i=1}^M \chi_i(x)=1$ on $\Omega$,
\begin{equation*}
\chi_i\in C_c^\infty(\Omega^\delta_i), \qquad 0\leq \chi_i\leq 1, \qquad \chi_i\equiv 1 \, \textup{on $\Omega_i^{-\delta}$} \qquad \textup{and} \qquad |\nabla \chi_i| \leq \frac{2}{\delta} \quad \textup{$i=1,\dots,M,$}
\end{equation*}
adopting the following conventions
\begin{equation}
    B^{\delta}:=\{x\in \R^n \, :\, \operatorname{dist}(x,B)<\delta\} \qquad \textup{and} \qquad B^{-\delta}:=\{x\in B \, :\, \operatorname{dist}(x,\partial B)>\delta\},
\end{equation}
for every $\delta>0$ and $B\subseteq \R^n$. Since $\sum_{i=1}^M\chi_i=1$ on $\Omega$, we have $1-v^k=\sum_{i=1}^M (1-\lambda^i_k)\chi_i\le 1-\lambda^{i_k}_k$, for some $i_k$. Arguing as in \eqref{e:gamma limsup vettoriale affine functions}, using the monotonicity of $\PotDann$ and \eqref{e:gamma lower euclidean suppenergy},
\begin{align}\label{e:stima lim sup caso affine generale}
& \limsup_{k\to\infty}\widehat{\Fm}_{\eps_k}(u_{k},v_{k};A)
\le \sum_{i=1}^M \calL^n(\Omega_i^{-\delta}\cap A)h_{\varsigma}(\xi_i)\nonumber\\
&\quad+\limsup_{k\to \infty} \Big[\Big(C_u+\ell(\lambda^{i_k}_k)\tau^{i_k}_{k}\frac{\PotDann(1-\lambda^{i_k}_k)}{Q(1-\lambda^{i_k}_k)}+\frac{4\eps_kMC^2_\gamma}{\delta^2}\Big)\sum_{i=1}^M \calL^n((\Omega^\delta_i\setminus \Omega^{-\delta}_i)\cap A)\Big],
\end{align}
where $C_u:= (\varphi(\infty)+1)\max_{i=1,\dots,M}\Phi(\xi_i)$. Since $\lambda^{i_k}_k\to1$, $(u_k,v_k)\to(u,1)$ in $L^1(\Omega,\R^{N+1})$, and letting $\delta\downarrow0$ in \eqref{e:stima lim sup caso affine generale} gives
\[
 \callF''(u,1;A)\le \sum_{i=1}^M \calL^n(\Omega_i\cap A)h_{\varsigma}(\xi_i)=H(u;A).
\]
Finally, for general $u\in W^{1,1}(\Omega,\R^N)$, extend $u$ (still denoted by $u$) to $W^{1,1}_0(\Omega',\R^N)$ for some bounded $\Omega'\supset\supset \Omega$ (using that $\Omega$ is Lipschitz), approximate by piecewise affine $u_k\to u$ in $W^{1,1}(\Omega',\R^N)$ (see \cite[Chapter X, Prop.~2.1]{ekeland1999convex}), and conclude by the continuity of $H$ in $W^{1,1}$ and the $L^1$–lower semicontinuity of $\callF''$.

\medskip
{\bf Step 2. Inner regularity of $\callF''(u,1;\cdot)$ and existence of the $\Gamma(L^1)$–limit for $u\in BV(\Omega,\R^N)$.}
We claim that for $u\in BV(\Omega,\R^N)$,
\begin{equation}\label{inner}
\callF''(u,1;\cdot)=(\callF'')_-(u,1;\cdot).
\end{equation}
Given $A$ open and $\delta>0$, choose $A'\subset\subset A''\subset\subset A$ and an open set $C$ with $A\setminus A'\subset C$ so that $\mathcal H(u;C)\le\delta$, where $\mathcal H$ is defined in \eqref{e:lsc envelope H}. By Lemma~\ref{l:wsub} and \eqref{e:stima limsup diffusa},
\[
\callF''(u,1;A)\le \callF''(u,1;A'\cup C)\le \callF''(u,1;A'')+\mathcal H(u;C)\le \callF''(u,1;A'')+\delta,
\]
which gives \eqref{inner}. Hence, from \eqref{e:additivity},
\begin{equation}\label{e:Gammabarlimite uguale al gammalimite}
\overline{\callF}(u,1;\cdot)\le \callF'(u,1;\cdot)\le \callF''(u,1;\cdot)=\overline{\callF}(u,1;\cdot),
\end{equation}
so the $\Gamma$–limit of $\widehat{\callF}_{\eps_k}(u,1;\cdot)$ exists and coincides with $\overline{\callF}(u,1;\cdot)$ for every $u\in BV(\Omega,\R^N)$.

\medskip
{\bf Step 3. Integral representation of the $\Gamma(L^1)$–limit on $BV(\Omega,\R^N)\times\{1\}$.}
To represent $\overline{\callF}$ in integral form and to estimate its densities, we consider the coercive perturbation
\[
\overline{\callF}_\lambda(u,1;A):=\overline{\callF}(u,1;A)+\lambda{|Du|(A)},
\]
for $u\in BV(\Omega,\R^N)$, $A\in \mathcal{A}(\Omega)$ and $\lambda\in(0,1]$. By \eqref{e:bulk-growth-h}, \eqref{e:lbcomp}, \eqref{e:stima limsup diffusa}, and \eqref{e:Gammabarlimite uguale al gammalimite},
\[
{\lambda|Du|(\Omega)-c\calL^n(\Omega)\le
\overline{\callF}_\lambda(u,1;\Omega)\le
c(|Du|(\Omega)+\calL^n(\Omega))\,,}
\]
for some $c>0$ and all $u\in BV(\Omega,\R^N)$. Moreover, translation invariance holds:
\[
\widehat{\callF}_{\eps_k}(u(\cdot-z),v(\cdot-z);z+A)=\widehat{\callF}_{\eps_k}(u,v;A),\qquad
\widehat{\callF}_{\eps_k}(u+b,v;A)=\widehat{\callF}_{\eps_k}(u,v;A),
\]
whence the analogous properties for $\overline{\callF}$.
By \cite[Theorem~3.12]{bouchitte1998global}, for $u\in BV(\Omega,\R^N)$ and $A\in\mathcal{A}(\Omega)$,
\[
\overline{\callF}_{\lambda}(u,1;A)=
\int_A h_\lambda(\nabla u)\dx
+\int_A h_\lambda^\infty({\diff D^cu})
+\int_{J_u\cap A}g_\lambda([u],\nu_u)\diff\calH^{n-1},
\]
with
\begin{align}\label{e:flambdadff}
h_\lambda(\xi):=\limsup_{r\downarrow0}\frac{1}{r^n}
\inf\Big\{\overline{\callF}_\lambda(w,1;Q_r):
\,w\in BV(Q_r,\R^N),\;
w(x)=\xi x\text{ on }\partial Q_r\Big\},
\end{align}
for $\xi\in\R^{N\times n}$,
\begin{equation}\label{e:flambda}
g_\lambda(\zeta,\nu):=\limsup_{r\downarrow0}\frac{1}{r^{n-1}}
\inf\Big\{\overline{\callF}_\lambda(w,1;Q_r^\nu):
\,w\in BV(Q_r^\nu,\R^N),\;
w=u_{\zeta,\nu}\text{ on }\partial Q_r^\nu\Big\},
\end{equation}
for $\zeta\in\R^N$, $\nu\in\Sf^{n-1}$, and $h^\infty_\lambda(\xi):=\limsup_{t\to\infty}\frac{h_\lambda(t\xi)}{t}$. Using \eqref{e:stima limsup diffusa},
\begin{equation}\label{lambdalin}
h_\lambda(\xi)\leq
\frac{1}{r^n}\overline{\callF}_\lambda(\xi x,1;Q_r)
\leq h_\varsigma^{\qc}(\xi)+\lambda|\xi|,
\end{equation}
hence
\begin{equation}\label{recessione}
h_\lambda^\infty(\xi)\leq
h_\varsigma^{\qcinfty}(\xi)+\lambda|\xi|.
\end{equation}
We now prove
\begin{equation}\label{e:upbvsalto}
g_\lambda(\zeta,\nu)\leq g(\zeta,\nu)+\lambda|\zeta|,
\end{equation}
for all $\zeta\in\R^N$, $\nu\in\Sf^{n-1}$. From \eqref{e:flambda},
\begin{align}\label{sfc}
g_\lambda(\zeta,\nu)
\leq\limsup_{r\downarrow0}\frac{1}{r^{n-1}}
\overline{\callF}_\lambda(u_{\zeta,\nu},1;Q_r^\nu)
=\limsup_{r\downarrow0}\frac{1}{r^{n-1}}
\overline{\callF}(u_{\zeta,\nu},1;Q_r^\nu)
+\lambda|\zeta|.
\end{align}
By the definition of $\overline\callF$, for any
$(\tilde u_k,\tilde v_k)\to(u_{\zeta,\nu},1)$
in $L^1(Q_r^\nu;\R^{N+1})$,
\begin{equation}\label{recupp}
\overline{\callF}(u_{\zeta,\nu},1;Q_r^\nu)
\leq\limsup_{k\to\infty}
\widehat{\callF}_{\eps_k}(\tilde u_k,\tilde v_k;Q_r^\nu).
\end{equation}
Thus, it suffices to construct a suitable recovery sequence
$(\tilde u_k,\tilde v_k)$ depending on $r$, $\zeta$, and $\nu$.
By applying Proposition~\ref{propuvjump} with
$\eps_k^*:=\eps_k/r$ and $\eta_k^*:=\eta_{\eps_k}$,
we obtain sequences
$(u_k^*,v_k^*)\to(u_{\zeta,\nu},1)$
in $L^2(Q^\nu,\R^{N+1})$ such that
\begin{equation}\label{estg}
\lim_{k\to\infty}
\callF^\infty_{\eps_k^*}(u_k^*,v_k^*;Q^\nu)
=g(\zeta,\nu)
\quad\text{and}\quad
\lim_{k\to\infty}
\eta_k^*\|\nabla u_k^*\|_{L^2(Q^\nu)}^2=0.
\end{equation}
Define $\tilde u_k(y):=u_k^*\!\left(\frac{y}{r}\right)$ on $Q_r^\nu$, so that $\tilde u_k\to u_{\zeta,\nu}$ in $L^2(Q_r^\nu,\R^N)$. Fix $\tau>0$ and $\delta\in(0,1)$.
By \eqref{eq:uniform-recession-estimate}, we have
$\Phi(\xi)\leq(1+\tau)\Phi_\infty(\xi)$ for $|\xi|$ large,
and hence, for all $\xi$,
\begin{equation}\label{e:stima con recessione quadratica in upperbound vettoriale}
\Phi(\xi)\leq(1+\tau)\Phi_\infty(\xi)+C(\tau).
\end{equation}
Let $K\geq1$ be the constant in
Lemma~\ref{l:quasi monotonia funzione di degradamento}.
For every $\eps>0$ sufficiently small (depending on $\delta$), define
\[
\textstyle \delta_\eps:=
\max\Big\{t\in[0,1):
\eps f^2(t)\leq\frac{\delta}{K}\Big\}.
\]
Since $f^2$ is continuous on $[0,1)$ and diverges as $t\uparrow1$,
this maximum is well defined, belongs to $(0,1)$, and satisfies $\eps f^2(\delta_\eps)=\frac{\delta}{K}.$ Moreover, $\delta_\eps\to1$ as $\eps\to0$, since every fixed $t<1$
belongs to the above sublevel set for $\eps$ sufficiently small.
By Lemma~\ref{l:quasi monotonia funzione di degradamento},
\[
\eps f^2(t)\leq K\eps f^2(\delta_\eps)=\delta
\qquad\text{for every }t\in[0,\delta_\eps],
\]
whereas the maximality of $\delta_\eps$ gives
$f^2(t)>f^2(\delta_\eps)$ for every $t\in(\delta_\eps,1)$.
Consequently,
\begin{equation}\label{e:stima di soglia per f in upperbound vettoriale}
\eps f^2(t\wedge\delta_\eps)\leq\delta,
\qquad
f^2(t\wedge\delta_\eps)\leq f^2(t)
\qquad\text{for every }t\in[0,1).
\end{equation}

Setting $v^*_{k,\delta}:=v_k^*\wedge\delta_{\eps_k}$,
the bulk term does not increase by
\eqref{e:stima di soglia per f in upperbound vettoriale}.
The gradient term does not increase either, since
\[
\nabla v^*_{k,\delta}
=\mathbf1_{\{v_k^*<\delta_{\eps_k}\}}\nabla v_k^*
\qquad\text{a.e. in }Q^\nu.
\]
Thus, by the monotonicity of $\PotDann$ and $\calL^n(Q^\nu)=1$,
\begin{equation}\label{e:prima approssimazione per troncamento in upperbound vettoriale}
\callF_{\eps_k^*}^\infty(u_k^*,v_{k,\delta}^*;Q^\nu)
\leq
\callF_{\eps_k^*}^\infty(u_k^*,v_k^*;Q^\nu)
+r\frac{\PotDann(1-\delta_{\eps_k})}{4\eps_k}.
\end{equation}
Moreover, using
$\eps_k f^2(\delta_{\eps_k})=\delta/K$
and \eqref{eq:def-varsigma},
\begin{equation}\label{e:potential cost threshold upperbound vettoriale}
\frac{\PotDann(1-\delta_{\eps_k})}{4\eps_k}
=\frac{K}{4\delta}\ell(\delta_{\eps_k})
\frac{\PotDann(1-\delta_{\eps_k})}{Q(1-\delta_{\eps_k})}
\longrightarrow\frac{K\varsigma^2}{4\delta}
\qquad\text{as }k\to\infty.
\end{equation}
Now letting
$\tilde v_{k,\delta}(y):=v^*_{k,\delta}\!\left(\frac{y}{r}\right)$,
we have $\nabla\tilde v_{k,\delta}(y)
=\frac{1}{r}(\nabla v^*_{k,\delta})\!\left(\frac{y}{r}\right)$, thus from the $2$-homogeneity of $\gamma^2$, we get
\begin{equation*}
\textstyle \int_{Q_r^\nu}
\eps_k\gamma(\nabla\tilde v_{k,\delta})^2\diff y
=
r^{n-1}\int_{Q^\nu}
\eps_k^*\gamma(\nabla v^*_{k,\delta})^2\diff x.
\end{equation*}
Using analogously the $2$-homogeneity of $\Phi_\infty$,
changing variable we obtain
\begin{equation}\label{eqdeltaresc}
\callF_{\eps_k}^\infty
(\tilde u_k,\tilde v_{k,\delta};Q_r^\nu)
=
r^{n-1}\callF_{\eps_k^*}^\infty
(u_k^*,v_{k,\delta}^*;Q^\nu),
\quad
\|\nabla\tilde u_k\|_{L^2(Q_r^\nu)}^2
=
r^{n-2}\|\nabla u_k^*\|_{L^2(Q^\nu)}^2.
\end{equation}

Since $\varphi$ is differentiable at $0$ (cf. (H4)),
there exists a modulus of continuity $\vartheta$ such that
\begin{equation}\label{e:modulo continuità derivata di varphi in upperbound vettoriale}
\varphi(\eps_k f^2(\tilde v_{k,\delta}))
\leq
(1+\vartheta(\delta))\varphi'(0^+)
\eps_k f^2(\tilde v_{k,\delta}),
\end{equation}
for every $k$ sufficiently large.
Combining \eqref{eqdeltaresc},
\eqref{e:modulo continuità derivata di varphi in upperbound vettoriale},
\eqref{e:prima approssimazione per troncamento in upperbound vettoriale},
and \eqref{e:stima con recessione quadratica in upperbound vettoriale},
and using $\eps_k f^2(\tilde v_{k,\delta})\leq\delta<1$, we deduce
\begin{align*}
\callF_{\eps_k}(\tilde u_k,\tilde v_{k,\delta};Q_r^\nu)
&\leq
(1+\vartheta(\delta))
\int_{Q_r^\nu}\!\Big(
\varphi'(0^+)\eps_k f^2(\tilde v_{k,\delta})
\Phi(\nabla\tilde u_k)
+\frac{\PotDann(1-\tilde v_{k,\delta})}{4\eps_k}
+\eps_k\gamma(\nabla\tilde v_{k,\delta})^2
\Big)\,\dx
\\
&\leq
(1+\vartheta(\delta))\Big(
r^{n-1}(1+\tau)
\callF_{\eps_k^*}^\infty(u_k^*,v_{k,\delta}^*;Q^\nu)
+r^n C(\tau)
\Big)
\\
&\leq
(1+\vartheta(\delta))\Big(
r^{n-1}(1+\tau)
\callF_{\eps_k^*}^\infty(u_k^*,v_k^*;Q^\nu)
+r^n(1+\tau)\frac{\PotDann(1-\delta_{\eps_k})}{4\eps_k}
+r^n C(\tau)
\Big),
\end{align*}
where $C(\tau)$ may be enlarged and is independent of $k$, $r$,
and $\delta$. Moreover, from \eqref{eq:growth-Phi} and \eqref{eqdeltaresc},
\[
\textstyle \eta_{\eps_k}\!\int_{Q_r^\nu}
\Phi(\nabla\tilde u_k)\,\dx
\leq
C_\Phi\eta_{\eps_k}r^{n-2}
\|\nabla u_k^*\|_{L^2(Q^\nu)}^2
+C_\Phi\eta_{\eps_k}r^n.
\]

For every fixed $r>0$ and $\delta\in(0,1)$,
since $\delta_{\eps_k}\to1$, we have $(\tilde u_k,\tilde v_{k,\delta})
\to(u_{\zeta,\nu},1)$ in $L^2(Q_r^\nu,\R^{N+1})$. Summing up and passing to the limit as $k\to\infty$,
by \eqref{recupp}, \eqref{estg}, and
\eqref{e:potential cost threshold upperbound vettoriale}, we obtain
\begin{equation*}
\overline{\callF}(u_{\zeta,\nu},1;Q_r^\nu)
\leq
(1+\vartheta(\delta))\Big(
r^{n-1}(1+\tau)g(\zeta,\nu)
+r^n(1+\tau)\frac{K\varsigma^2}{4\delta}
+C(\tau)r^n
\Big).
\end{equation*}
Dividing by $r^{n-1}$ and letting $r\downarrow0$
with $\delta$ and $\tau$ fixed, in view of \eqref{sfc}, 
\[
g_\lambda(\zeta,\nu)
\leq
(1+\vartheta(\delta))(1+\tau)g(\zeta,\nu)
+\lambda|\zeta|.
\]
Letting first $\delta\downarrow0$ and then $\tau\downarrow0$,
inequality~\eqref{e:upbvsalto} follows. Finally, letting $\lambda\downarrow0$,
inequalities \eqref{lambdalin}, \eqref{recessione},
and \eqref{e:upbvsalto} yield, for every $u\in BV(\Omega,\R^N)$, $\overline{\callF}(u,1;\Omega)
\leq\Fm_\varsigma(u,1;\Omega).$ Combined with the lower bound (Theorem~\ref{thm:lbcompBV}),
this identifies the $\Gamma$–limit of the subsequence
$\widehat{\callF}_{\eps_k}$.
By Urysohn’s property~\cite[Proposition~8.3]{dalmaso1993introduction},
the result extends to the entire family $\widehat{\callF}_\eps$.

\medskip
{\bf Step 4. Extension to $GBV$.}
To extend \eqref{e:gammalisup estimate} to $u\in (GBV\cap L^1(\Omega))^N$, we argue by truncation. For $k\in\N$, let $\mathcal{T}_k$ be as in \eqref{e:Tk}. By Steps~1–3,
\[
 \callF''(\mathcal{T}_k(u),1;\Omega)\le \Fm_{\varsigma}(\mathcal{T}_k(u),1;\Omega).
\]
The claim follows by the $L^1$–lower semicontinuity of $\callF''$ and Proposition~\ref{prop:contF0}.
\end{proof}

We finally prove Theorem \ref{thm:general-anisotropic-gamma-convergence}.
\begin{proof}[Proof of Theorem~\ref{thm:general-anisotropic-gamma-convergence}]
The lower bound 
has been proven in Theorem~\ref{thm:lbcomp}.
The upper bound follows by Theorem~\ref{t:limsupndim} and $\callF_{\eps}\leq \widehat{\callF}_\eps$.
\end{proof}

\section{Acknowledgements}
The author warmly thanks Matteo Focardi for inspiring discussions and continuous encouragement.

\section{Funding}
\begin{center}
    \includegraphics[width=5 cm]{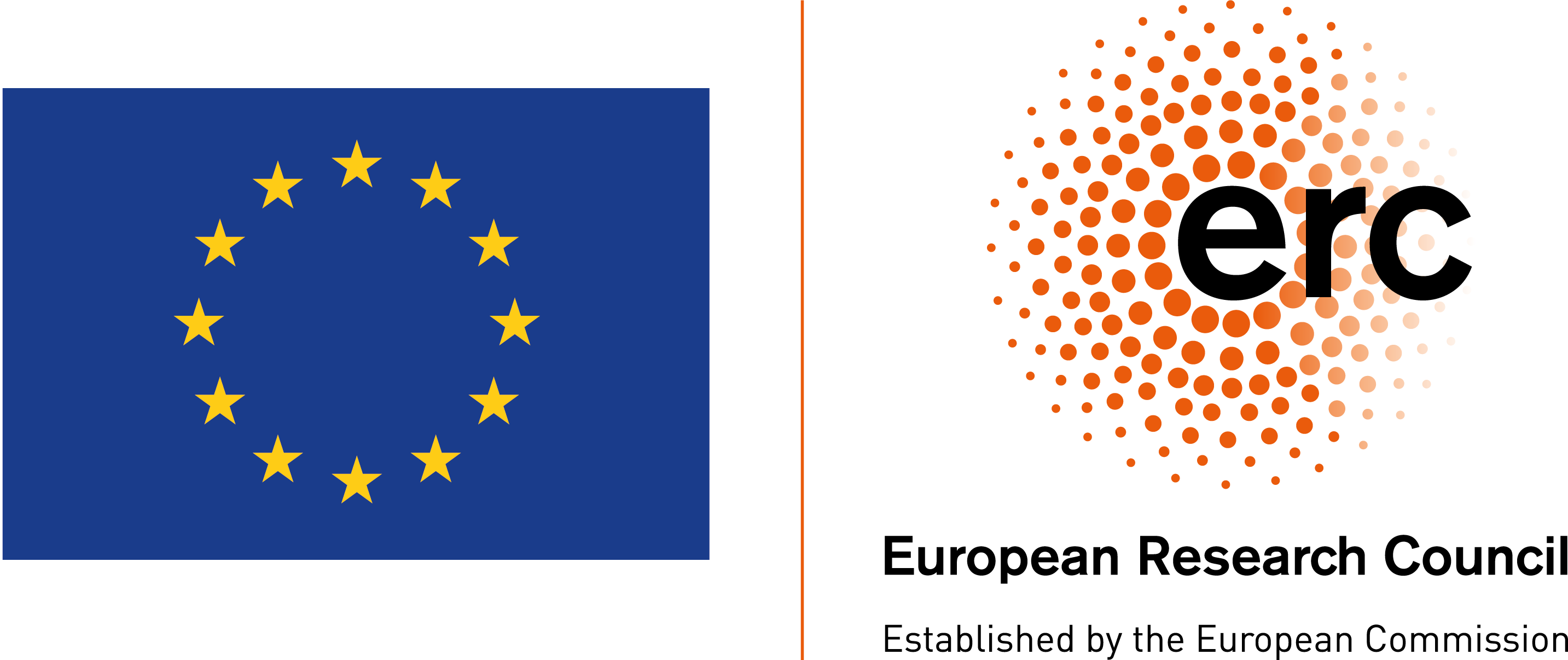}
\end{center}
Funded by the European Union. Views and opinions expressed are however those of the
author(s) only and do not necessarily reflect those of the European Union or the
European Research Council Executive Agency. Neither the European Union nor the
granting authority can be held responsible for them. This project has received funding
from the European Research Council (ERC) under the European Union’s Horizon
Europe research and innovation programme (grant agreement No. 101198055, project
acronym NEITALG).

\bibliographystyle{alpha} 
\bibliography{bibtex}

\end{document}